\pdfoutput=1
\documentclass[11pt,letterpaper]{amsart}
\usepackage[english,french]{babel}
\usepackage{graphicx}
\usepackage{subcaption}
\usepackage{tabularx}
\usepackage{booktabs}
\usepackage{float}
\usepackage{array}
\usepackage{amsmath}
\usepackage{amsfonts}
\usepackage{amssymb}
\usepackage{amsthm}
\usepackage[scr]{rsfso}
\usepackage{enumitem}
\usepackage{permute}
\usepackage{slashed}
\usepackage[usenames,dvipsnames]{xcolor}
\usepackage[pagebackref = true, colorlinks, linkcolor = Red, citecolor = Green, bookmarksdepth=2, linktocpage=true]{hyperref}
\usepackage[capitalize]{cleveref}
\usepackage{comment}
\usepackage{changepage}
\usepackage{dsfont}

\usepackage[T1]{fontenc}
\usepackage{makecell}

\usepackage{tikz}
\usetikzlibrary{calc}

\allowdisplaybreaks

\DeclareMathOperator{\Hom}{Hom}

\DeclareMathOperator{\Id}{Id}

\DeclareMathOperator{\Stab}{Stab}

\DeclareMathOperator{\vol}{vol}
\DeclareMathOperator{\supp}{supp}

\DeclareMathOperator{\SL}{SL}

\DeclareMathOperator{\SO}{SO}
\DeclareMathOperator{\SU}{SU}
\DeclareMathOperator{\Sp}{Sp}

\DeclareMathOperator{\Lip}{Lip}

\DeclareMathOperator{\Leb}{Leb}

\DeclareMathOperator{\ad}{ad}
\DeclareMathOperator{\Ad}{Ad}

\DeclareMathOperator{\inj}{inj}

\DeclareMathOperator{\Span}{span}

\DeclareMathOperator{\Fix}{Fix}

\DeclareMathOperator{\Gr}{Gr}

\newcommand{\N}{\mathbb{N}}
\newcommand{\Z}{\mathbb{Z}}
\newcommand{\Q}{\mathbb{Q}}
\newcommand{\R}{\mathbb{R}}
\newcommand{\C}{\mathbb{C}}

\newcommand{\LieG}{\mathfrak{g}}
\newcommand{\LieA}{\mathfrak{a}}

\newcommand{\LieU}{\mathfrak{u}}
\newcommand{\LieK}{\mathfrak{k}}
\newcommand{\LieM}{\mathfrak{m}}
\newcommand{\LieP}{\mathfrak{p}}

\newcommand{\LieH}{\mathfrak{h}}

\newcommand{\height}{\mathrm{ht}}

\newcommand{\rad}{\mathrm{rad}}

\newcounter{consta}
\renewcommand{\theconsta}{{A_{\arabic{consta}}}}
\newcommand{\consta}{\refstepcounter{consta}{\color{red}\theconsta}}

\newcounter{constm}
\renewcommand{\theconstm}{{M_{\arabic{constm}}}}
\newcommand{\constm}{\refstepcounter{constm}{\color{red}\theconstm}}

\newcounter{constc}
\renewcommand{\theconstc}{{C_{\arabic{constc}}}}
\newcommand{\constc}{\refstepcounter{constc}{\color{red}\theconstc}}

\newcounter{constE}
\renewcommand{\theconstE}{{{E}_{\arabic{constE}}}}
\newcommand{\constE}{\refstepcounter{constE}{\color{red}\theconstE}}

\newcounter{constd}
\renewcommand{\theconstd}{{D_{\arabic{constd}}}}
\newcommand{\constd}{\refstepcounter{constd}{\color{red}\theconstd}}

\makeatletter
\newcommand{\mytag}[2]{%
\text{#1}%
\@bsphack
\begingroup
\@onelevel@sanitize\@currentlabelname
\edef\@currentlabelname{%
\expandafter\strip@period\@currentlabelname\relax.\relax\@@@%
}%
\protected@write\@auxout{}{%
\string\newlabel{#2}{%
{#1}%
{\thepage}%
{\@currentlabelname}%
{\@currentHref}{}%
}%
}%
\endgroup
\@esphack
}
\makeatother

\DeclareFontFamily{U}{mathb}{\hyphenchar\font45}
\DeclareFontShape{U}{mathb}{m}{n}{
<5> <6> <7> <8> <9> <10> gen * mathb
<10.95> mathb10 <12> <14.4> <17.28> <20.74> <24.88> mathb12
}{}
\DeclareSymbolFont{mathb}{U}{mathb}{m}{n}
\DeclareMathSymbol{\bigast}{2}{mathb}{"06}
\DeclareMathAlphabet{\mathpzc}{OT1}{pzc}{m}{it}

\def\XXint#1#2#3{{\setbox0=\hbox{$#1{#2#3}{\int}$}
\vcenter{\hbox{$#2#3$}}\kern-.5\wd0}}

\theoremstyle{plain}
\newtheorem{theorem}{Theorem}[section]
\newtheorem{proposition}[theorem]{Proposition}
\newtheorem{lemma}[theorem]{Lemma}
\newtheorem*{theorem*}{Theorem}
\newtheorem*{proposition*}{Proposition}
\newtheorem*{lemma*}{Lemma}
\newtheorem{corollary}[theorem]{Corollary}

\newtheorem*{claim*}{Claim}
\theoremstyle{definition}
\newtheorem{definition}[theorem]{Definition}

\newtheorem{example}{Example}[section]
\newtheorem{construction}[theorem]{Construction}
\theoremstyle{remark}
\newtheorem{remark}[theorem]{Remark}

\Crefname{enumi}{Property}{Properties}
\Crefname{alternativei}{Alternative}{Alternatives}
\Crefname{subsection}{Subsection}{Subsections}

\begin{document}
\selectlanguage{english}


\title[Effective equidistribution for certain unipotent subgroups]{Polynomially effective equidistribution for certain unipotent subgroups in quotients of perfect Lie groups}

\author{Zuo Lin}
\address{Department of Mathematics, University of California, Berkeley, CA 94720-3840}
\email{zuo\_lin@berkeley.edu}

\date{\today}

\begin{abstract}
We prove an effective equidistribution theorem for orbits of certain unipotent subgroups in arithmetic quotients of perfect Lie groups with a polynomial error term. Even for semisimple quotients, our result provides the first infinite family of examples where effective equidistribution with polynomial error rate is obtained for non-horospherical unipotent subgroups. 

The proof is based on the spectral gap of the ambient space, an effective closing lemma, Bourgain's discretized projection theorem, and a sub-modularity inequality in irreducible representation. The sub-modularity inequality is crucial to our proof and is of independent interest. 

As applications, we obtain effective estimates on distribution of lattice orbits on homogeneous spaces, as well as an effective version of the Oppenheim conjecture for indefinite quadratic forms with a polynomial error rate in all dimension $d \geq 3$. 
\end{abstract}
\maketitle

\setcounter{tocdepth}{1}
\tableofcontents
\setcounter{part}{-1}

\part{Introduction}
\section{Introduction}
An important theme in homogeneous dynamics is the behavior of orbits of $\Ad$-unipotent subgroups from \emph{any} initial point. More precisely, let $G$ be a Lie group, $\Gamma < G$ be a lattice and $U \leq G$ be an $\Ad$-unipotent subgroup. Raghunathan conjectured that for \emph{any} initial point $x \in X = G/\Gamma$, the orbit closure $\overline{U.x}$ is a periodic orbit $L.x$ of some subgroup $U \leq L \leq G$. We say $L.x$ is periodic if $\Stab(x) \cap L$ is a lattice in $L$. In the literature the conjecture was first stated in the paper \cite{Dan81} and in a more general form in \cite{Mar90} where the subgroup $U$ is not necessarily $\Ad$-unipotent but generated by $\Ad$-unipotent elements. 

Raghunathan's conjecture was proved in full generality by Ratner \cite{Rat90a,Rat90b,Rat91a,Rat91b}. In her landmark work, Ratner also classified all ergodic invariant probability measures under the action of $U$ and proved an equidistribution theorem for orbits of $U$. These remarkable theorems have been highly influential and have led to a lot of important applications. 

Prior to Ratner's proof, the conjecture was known in certain important cases. In his seminal work \cite{Mar89}, Margulis proved Oppenheim conjecture by showing every $\SO(2, 1)$-orbit in $\SL_3(\R)/\SL_3(\Z)$ is either periodic or unbounded. Later, Dani and Margulis \cite{DM89,DM90} showed that any $\SO(2, 1)$-orbit is either periodic or dense. They also classified possible orbit closures of a one-parameter unipotent subgroup of $\SO(2, 1)$ in $\SL_3(\R)/\SL_3(\Z)$. We refer to the book by Morris \cite{Mor05} for a detailed historical background. 

Based on equidistribution results for unipotent subgroups, information on asymptotics of distribution of values of indefinite quadratic forms with signature $(p, q)$ on integer points when $p \geq 3$ or $(p, q) = (2, 2)$ is provided by Eskin, Margulis and Mozes in \cite{EMM98,EMM05}. Recently, Kim established asymptotics of distribution of values of indefinite quadratic forms with signature $(2, 1)$ in \cite{Kim24b}. 

Because of its intrinsic interest and in view of the applications, \emph{effective} results on distribution of the orbits of unipotent groups have been sought after for some time. 

An effective proof of the Oppenheim conjecture with a poly-logarithmic rate was obtained by Lindenstrauss and Margulis in \cite{LM14}. 

In the landmark works by Lindenstrauss and Mohammadi \cite{LM23} and later with Wang and Yang and by Yang \cite{LMW22,Yan24,LMWY25}, effective density and equidistribution theorems with polynomial rate for orbits of unipotent flows are established in quotients
of quasi-split, semisimple linear algebraic groups of absolute rank $2$. In \cite{LMWY25}, they established a proof of effective Oppenheim conjecture with a polynomial rate when the dimension $d = 3$ building on their effective equidistribution theorem. For nonsemisimple quotients, an effective equidistribution of expanding translate of some non-horospherical unipotent subgroups in the space of affine lattices is obtained by Str{\"o}mbergsson \cite{Str15} in dimension $2$, Prinyasart \cite{Pri18} in dimension $3$ and Kim \cite{Kim24a} in all dimensions. 

In a closely related setting, effective results for random walk on homogeneous spaces have also been sought after. Most relevant to us are works on random walks on torus by Bourgain--Furman--Lindenstrauss--Mozes and He--de Saxc{\'e} \cite{BFLM11,HdS22}, on simple homogeneous spaces by B{\'e}nard--He \cite{BH24,BH25b}, and related results on Diophantine approximation by B{\'e}nard--He--Zhang \cite{BHZ25}.

We refer to \cite{LMW22} and \cite{Moh23} for a throughout survey on both historical background and recent progress. 

The main motivation for this paper is to cover a wide class of non-horospherical unipotent subgroups in a maximal semisimple subgroup $H$ of a perfect group $G$. Before we state the main theorem, let us introduce the following notions. 

Let $\mathbf{G} \leq \SL_N$ be a connected perfect $\Q$-group, let $G = \mathbf{G}(\R)^\circ$ be the identity component of its group of $\R$-points (in Hausdorff topology), and let $\LieG = \mathrm{Lie}(G)$ be its Lie algebra. Let $\Gamma = G \cap 
\SL_{N}(\Z)$. It is a lattice of $G$ by the theorem of Borel--Harish-Chandra. Let $X = G/\Gamma$ and let $\mu_X$ be the unique probability Haar measure on $X$. 

We fixed an inner product on $\LieG$ induced by the standard Euclidean metric on $\mathfrak{sl}_N(\R)$ and we use $\|\cdot\|_2$ to denote the induced norm. It induces a right $G$-invariant Riemannian metric on $G$ and hence a Riemannian metric $d_X$ on $X = G/\Gamma$. 

Let $\mathbf{H}$ be a connected semisimple $\R$-subgroup of $\mathbf{G}$ and let $H = \mathbf{H}(\R)$ be its group of $\R$-points. Suppose $H$ has no compact simple factor. Let $\LieH$ be its Lie algebra and let $\LieH = \bigoplus_{i = 1}^k \LieH_i$ be the unique decomposition of $\LieH$ into its simple ideals. Since $\mathbf{H}$ is semisimple, there exists a decomposition of $\LieG$ into $H$-invariant complement as $\LieG = \LieH \oplus \mathfrak{r}$. 

Let $\mathbf{a} \in \LieH$ be a semisimple element with $\|\mathbf{a}\|_2 = 1$ whose projection to each simple factor is non-zero. Suppose $\ad\mathbf{a}$ is diagonalizable over $\R$. Let $a_t = \exp(t\mathbf{a})$. 

Let $U$ be the expanding horospherical subgroup of $H$ as the following:
\begin{align*}
    U = \{u \in H: a_{-t} u a_{t} \to e \text{ as } t \to +\infty\}.
\end{align*}
Let $\LieU$ be the Lie algebra of $U$. Let $\mathcal{B}_r^{\LieU}$ be the ball of radius $r$ under the norm $\|\cdot\|_2$ centered at the origin and let $\mathcal{B}_r^U = \exp(\mathcal{B}_r^{\LieU})$. We fix a Haar measure $\mathrm{d}\mu_U(u)$ on $U$ so that the measure of $\mathcal{B}_1^U$ is of measure $1$. 

Throughout this paper, we fixed the data $(G, \Gamma, H, \mathfrak{r}, \mathbf{a}, U)$ as above. By \emph{absolute constant}, we mean that the constant depending explicitly on the data $(G, \Gamma, H, \mathfrak{r}, \mathbf{a}, U)$. We will discuss more details on this dependence in Section~\ref{sec:Global prelim}. 

The following are standing assumption throughout this paper. 

\newtheoremstyle{named}{}{}{\itshape}{}{\bfseries}{.}{ }{#1 \thmnote{#3}}
\theoremstyle{named}
\newtheorem{namedhypothesis}{}
\newcommand{\Irrep}{\nameref{hyp:Irrep}\xspace}
\begin{namedhypothesis}[Irreducible]\label{hyp:Irrep}
The subspace $\mathfrak{r}$ is an irreducible representation of $\mathbf{H}$. 
\end{namedhypothesis}

\newcommand{\Schu}{\nameref{hyp:Schu}\xspace}
\begin{namedhypothesis}[Schubert]\label{hyp:Schu}
Let $\Fix(U)$ be the fixed point of $U$ on $\mathfrak{r}$. The orbit $H.\Fix(U)$ in $\Gr_{\dim \Fix(U)}(\mathfrak{r})$ is not contained in any proper Schubert variety. 
\end{namedhypothesis}

\begin{theorem}\label{thm:main equidistribution}
Suppose the data $(G, \Gamma, H, \mathfrak{r}, \mathbf{a}, U)$ satisfy \Irrep{} and \Schu, then there exist absolute constants $\consta\label{a:main equidistribution1} > \consta\label{a:main equidistribution2} \geq 1$ and $\kappa > 0$ so that the following holds. For all $x_0 \in X$ and large enough $R$ depending explicitly on $x_0$, for any $T \geq R^{\ref{a:main equidistribution1}}$, at least one of the following is true. 
\begin{enumerate}
\item For all $\phi \in \mathrm{C}_c^\infty(X)$, 
\begin{align*}
\Biggl|\int_{\mathcal{B}_1^U} \phi(a_{\log T} u.x_0) \,\mathrm{d}\mu_U(u) - \int_X \phi \,\mathrm{d}\mu_X\Biggr| \leq \mathcal{S}(\phi) R^{-\kappa}
\end{align*}
where $\mathcal{S}(\phi)$ is a certain Sobolev norm. 
\item There exists $x \in X$ so that $H.x$ is periodic with $\vol(H.x) \leq R$ and 
\begin{align*}
    d_X(x_0, x) \leq T^{-\frac{1}{\ref{a:main equidistribution2}}}.
\end{align*}
\end{enumerate}
\end{theorem}

\begin{remark}
We remark that the dependence of $R$ on $x_0$ is of the form $R \gg \inj(x_0)^{-\star}$. See Section~\ref{sec:Global prelim} for the precise definition of $\inj(x_0)$ and the convention on $\star$-notations. 
\end{remark}

\subsection{Examples}
We provide examples so that the data $(G, \Gamma, H, \mathfrak{r}, \mathbf{a}, U)$ satisfy {}\Irrep{} and \Schu{} hence our theorem applies. We also indicate the relation with prior works. Note that under the condition $\LieG$ being perfect, {}\Irrep{} implies that $\LieG$ is either semisimple or with semi-direct product structure, see Lemma~\ref{lem:structure under irrep condition}. 

Let us also introduce the following property. 
\newcommand{\Prox}{\nameref{hyp:Prox}\xspace}
\begin{namedhypothesis}[Proximal]\label{hyp:Prox}
The action of $\Ad(a_t)|_{\mathfrak{r}}$ on $\mathfrak{r}$ is proximal, that is,  the largest eigenvalue of $\Ad(a_1)|_{\mathfrak{r}}$ on $\mathfrak{r}$ has multiplicity $1$. 
\end{namedhypothesis}
Note that \Prox{} implies \Schu{}, see Lemma~\ref{lem:proximal implies schubert}. Note also that if $\mathbf{H}$ is $\R$-split, then {}\Prox{} follows from {}\Irrep{} for all regular semisimple element $\mathbf{a} \in \LieH$. Since those conditions depend only on $(G, H, \mathbf{a})$, we will only describe them in this subsection. 

We first give examples when $\mathbf{G}$ is semisimple. 
\begin{example}
    Let $\mathbf{G}$ be a quasi-split semisimple group with absolute rank $2$ and let $\mathbf{H}$ be a principal $\SL_2$ in $\mathbf{G}$. Under this condition, $\mathbf{G}$ is locally isomorphic to one of the following:
    \begin{align*}
        \SL_2(\C), && \SL_2(\R) \times \SL_2(\R), && \SL_3(\R), && \SU(2, 1), && \Sp_4(\R), && \mathrm{G}_2(\R).
    \end{align*}
    In this case, the complement $\mathfrak{r}$ is an irreducible representation of $\SL_2$ and {}\Irrep{} and \Prox{} are automatically satisfied. These are the cases proved by Lindenstrauss--Mohammadi--Wang--Yang in \cite{LMWY25}. The first two cases here was studied by Lindenstrauss, Mohammadi and Wang in \cite{LM23,LMW22}. 
\end{example}

\begin{example}
    Let $\mathbf{G}_0$ be a $\Q$-group that is $\R$-simple and $\R$-split, e.g., $\SL_d$, $\SO(n,n + 1)$, $\Sp_{2m}$, $\SO(n,n)$. Let $\mathbf{G} = \mathbf{G}_0 \times \mathbf{G}_0$ and let 
    \begin{align*}
        \mathbf{H} = \{(h, h): h \in \mathbf{G}_0\} < \mathbf{G}.
    \end{align*}
    Let $\mathbf{a}$ be a \emph{regular} semisimple element in $\LieH$, that is, it lies in the interior of a Weyl chamber. In this case, {}\Prox{} follows from the fact that $\mathbf{G}_0$ is $\R$-simple and {}\Prox{} follows from the fact that $\mathbf{G}_0$ is $\R$-split and $\mathbf{a}$ is regular. 
\end{example}

\begin{example}
    Let $\mathbf{G} = \SL_{2n}$ and let $\mathbf{H} = \Sp_{2n}$. Let $\mathbf{a}$ be a \emph{regular} element in $\LieH$. A direct computation shows that the fix point of $U$ in $\mathfrak{r}$ is $1$-dimensional, which is equivalent to {}\Irrep{} and \Prox{} (see Lemma~\ref{lem:equi Fix U}). 
\end{example}

\begin{example}
We present a generalization of the previous example. A Lie subgroup $H < G$ is said to be symmetric if $H^\circ \leq H \leq H^\Sigma$ where the latter is the closed subgroup of fixed points of an involutive automorphism $\Sigma: G \to G$. We note that in this case {}\Irrep{} holds if and only if $\LieH$ is \emph{maximal}. Therefore our theorem applies to all triple $(G, H, \mathbf{a})$ where $H$ is a $\R$-split maximal symmetric semisimple subgroup of $G$ and $\mathbf{a}$ is a regular semisimple element in $\LieH$. 
\end{example}

\begin{example}
    Let $n \geq 3$. Let $Q_n$ be the quadratic form of signature $(n, 1)$ defined as the following:
    \begin{align*}
        Q_n(x_1, \ldots, x_{n+1}) = 2x_1 x_{n+1} - x_{2}^2 - \cdots - x_{n}^2.
    \end{align*}
    Let $\mathbf{G} = \SO(Q_n)$ and let $\mathbf{H} = \Stab_\mathbf{G}(e_n)$. In this case both $\mathbf{G}$ and $\mathbf{H}$ is of $\R$-rank $1$. Let $a_t$ be the diagonal flow in $\mathbf{H}(\R)$ as the following
    \begin{align*}
        a_t e_1 = e^t e_1,&& a_t e_{n+1} = e^{-t}e_{n + 1},&& a_te_i = e_i \text{ for } i = 2, \ldots, n.
    \end{align*}
    These data satisfy \Irrep{} and \Prox{}. We remark that $\mathbf{G}(\R)^\circ$ is the identity component of the isometry group of $n$-dimensional hyperbolic space and $\mathbf{H}(\R)^\circ$ is the identity component of the isometry group of a totally geodesic $(n - 1)$-dimensional hyperbolic hyperplane. One can show {}\Prox{} via the horospheres in this setting. 
\end{example}

\begin{example}
    Let $d = p + q$ and let $Q_0$ be the quadratic form of signature $(p, q)$ defined as the following:
    \begin{align*}
        Q_0(x_1, \ldots, x_{p + q}) = 2x_1 x_{p + q} + x_2^2 + \cdots + x_p^2 - x_{p + 1}^2 - \cdots - x_{p + q - 1}^2.
    \end{align*}
    Let $\mathbf{G} = \SL_d$ and let $\mathbf{H} = \SO(Q_0)$. Let $e_1, \ldots, e_{p + q}$ be the standard basis of $\R^d$. Let $a_t$ be the following diagonal flow in $H$ defined as the following:
    \begin{align*}
        a_t e_1 = e^t e_1,&& a_t e_{p+q} = e^{-t}e_{p+q},&& a_te_i = e_i \text{ for } i = 2, \ldots, p+q-1.
    \end{align*}
    A direct computation shows that the fix point of $U$ in $\mathfrak{r}$ is $1$-dimensional, which is equivalent to {}\Irrep{} and \Prox{} (see Lemma~\ref{lem:equi Fix U}). This example is a generalization of the case in \cite{LMWY23}. 

    We remark that these data have been used by Eskin--Margulis--Mozes in \cite{EMM98,EMM05} to prove an asymptotic formula for values of distribution of irrational quadratic form on integer points. 

    We remark that a) the group $H$ is \emph{not} $\R$-split if $p - q \geq 2$, b) if $p = q \geq 3$, this choice of $\{a_t\}_{t \in \R}$ is \emph{not} regular. 
\end{example}

\begin{example}
    We also provide a family of examples where $\mathbf{H}$ is semisimple but not simple. We identify $\R^{nm} \cong \R^n \otimes \R^m$. Let $\mathbf{G} = \SL_{mn}$ and let $\mathbf{H} = \SL_{n} \times \SL_m$. The group $\mathbf{H} = \SL_{n} \times \SL_m$ is embedded into $\mathbf{G}$ via the tensor product of irreducible representations $\R^n$ and $\R^m$ of each of its factor. Let $\mathbf{a}$ be a regular element in $\mathfrak{sl}_n \oplus \mathfrak{sl}_m$. In this case, the complement $\mathfrak{r}$ is isomorphic to $\mathfrak{sl}_n \otimes \mathfrak{sl}_m$ as a representation of $\SL_n \times \SL_m$. 

    We remark that in the previous example when $Q_0$ has signature $(2, 2)$, $\mathbf{H}$ is locally isomorphic to $\SL_2 \times \SL_2$ and $\mathfrak{r} \cong \mathfrak{sl}_2 \otimes \mathfrak{sl}_2$ as in this example. 
\end{example}

\begin{example}
    Let $G = \SL_d(\C) \times \SL_d(\C)$ and 
    \begin{align*}
        H = \{(h, h): h \in \SL_d(\C)\} < G.
    \end{align*}
    Let $\mathbf{a}$ be an element in $\LieH$ in the interior of a Weyl chamber. In this case the data satisfy the \Irrep{} and \Schu{} but not \Prox{}. 
\end{example}

We now give examples when $\mathrm{R}_u(\mathbf{G}) \neq \{1\}$. 

\begin{example}
    Let $G = H \ltimes V$ where $H$ is a semisimle real algebraic group and $V$ is a non-trivial irreducible representation of $H$. Suppose we pick $\mathbf{a} \in \LieH$ so that the corresponding $U$ satisfies the following. Let $\Fix(U)$ be the fixed point of $U$ in $V$. We assume $\dim \Fix(U) = 1$. In this case the data satisfy \Irrep{} and \Prox{}, see Lemma~\ref{lem:equi Fix U}. 

    One explicit family of examples is of the following. Let $G = \SL_d(\R) \ltimes V$, $H = \SL_d(\R)$ and $V$ be any irreducible representation of $\SL_d(\R)$. Let $\mathbf{a}$ be a regular semisimple element in $H$. When $V \cong \R^d$ is the standard representation, this is the case proved by Str{\"o}mbergsson \cite{Str15} for $d = 2$, by Prinyasart \cite{Pri18} for $d = 3$ and by Kim \cite{Kim24a} for general $d$. (See also the next example. )

    Another family of examples where $\mathbf{a}$ is possibly not regular is of the following. Let $G = \SO(p, q) \ltimes \R^{p+q}$ and $H = \SO(p, q)$ acting on $\R^{p + q}$ via standard representation. Let $a_t$ be the following diagonal flow in $H$ defined as the following:
    \begin{align*}
        a_t e_1 = e^t e_1,&& a_t e_{p+q} = e^{-t}e_{p+q},&& a_te_i = e_i \text{ for } i = 2, \ldots, p+q-1.
    \end{align*}
    The case where $(p, q) = (2, 1)$ is studied by Seong in his PhD thesis \cite{Seo26} under the context of effective Oppenheim conjecture for irrational inhomogeneous quadratic form with rational homogeneous part. 
\end{example}

\begin{example}
We now give a family of examples where $\dim \Fix(U) > 1$. Let $G = \SL_d(\R) \ltimes \R^d$, $H = \SL_d(\R)$ acting on $\R^d$ via the standard representation. In this case, we let $d = m + n$ and
\begin{align*}
    a_t = \begin{pmatrix}
        e^{nt} \Id_n & \\
         & e^{-mt}\Id_m
    \end{pmatrix}.
\end{align*}
This family satisfies \Irrep{} and \Schu{} but not \Prox{} if $n \geq 2$. {}\Schu{} can be proved using the irreducibility of $\wedge^n \R^d$ as $\SL_d(\R)$-representation. This is the case proved by Kim \cite{Kim24a}. 
\end{example}

\subsection{Applications}
As mentioned before, homogeneous dynamics has had several applications in number theory and geometry. In this brief section, we mention two applications of Theorem~\ref{thm:main equidistribution}.  


\subsubsection{Distribution of lattice orbit in homogeneous spaces}
Motivated by central problems in arithmetic geometry, homogeneous dynamics has been successful in providing sharp  estimates in orbital counting problems on homogeneous varieties. In the above notation, this concerns $\Gamma$-orbits in $H\backslash G$. 

There has been extensive research on this problem especially on the whole group $G$ and the symmetric space $K \backslash G$ where all $\Gamma$-orbits are discrete (see \cite{BO12} and references therein). There has also been research, albeit to a lesser extent, on non-discrete $\Gamma$-orbits, where one expects a limiting distribution of $\Gamma$-orbits on $H \backslash G$ according to some natural limiting process. 

A major work in this direction was done by Gorodnik--Weiss \cite{GW07}. For any closed subgroup $S \leq G$, we denote by $B_r^S(x) \subset S$ the open ball of radius $r > 0$ centered at $x \in X$ with respect to the metric induced by the Riemannian metric on $G$. We write 
\begin{align*}
    H_T := H \cap KB_T^G(e)K \quad\text{ and }\quad\Gamma_T := \Gamma \cap KB_T^G(e)K
\end{align*}
for all $T > 0$. For all $y_0 \in H \backslash G$, there is a canonical measure $\nu_{y_0} \ll \mu_{H \backslash G}$ as defined in \cite[Eq. (12) and Proposition 5.1]{GW07} which we normalize as $\hat{\nu}_{y_0} := \vol(X)^{-1}\nu_{y_0}$. In \cite{GW07}, they proved that $\hat{\nu}_{y_0}$ is the limiting density of the orbit $\Gamma. y_0 \subset H \backslash G$ whenever it is dense in the following sense.

\begin{theorem}[{\cite[Theorem 1.1]{GW07}}]
\label{thm:GWTheorem}
Let $y_0 \in H \backslash G$ such that $\Gamma.y_0 \subset H \backslash G$ is dense. Then, for all $\psi \in C_{\mathrm{c}}(H \backslash G)$, we have
\begin{align*}
\lim_{T \to +\infty}\frac{1}{\mu_H(H_T)} \sum_{\gamma \in \Gamma_T} \psi(\gamma.y_0) = \int_{H \backslash G} \psi \, d\hat{\nu}_{y_0}.
\end{align*}
\end{theorem}

In a previous joint paper by Sarkar and the author, we establish an estimate of error term  under an effective equidistribution hypothesis, see \cite[Theorem 1.12]{LS24}. Combining it with Theorem~\ref{thm:main equidistribution}, we have the following unconditional result. 

Before we state it, let us introduce the following notation. For all Lipschitz function $\psi$ on $H \backslash G$ with compact support, we define a corresponding constant
\begin{align*}
D_\psi := \inf\bigl\{r > 0: \supp(\psi) \subset B_r^G(e) \cdot H \subset H \backslash G\bigr\} > 0.
\end{align*}

\begin{theorem}
Suppose there exists $\mathbf{a}$ lying in interior of a Weyl chamber of $\LieH$ so that the data $(G, \Gamma, H, \mathfrak{r}, \mathbf{a}, U)$ satisfy \Irrep{} and \Schu. 
There exists $\varsigma_H > 0$ depending only on $H$ and $\kappa > 0$ such that the following holds. Let $\psi$ be a Lipschitz function on $H \backslash G$ with compact support, $g_0 \in G$, $x_0 = g_0\Gamma \in X$, and $y_0 = Hg_0 \in H \backslash G$. There exists $M_{\psi, g_0} > 0$ such that for all $R \gg_{G, H, \Gamma, g_0, \chi} 1$ and $T \gg_{G, H, \Gamma} \log(R) + M_{\psi, g_0}$, at least one of the following holds.
\begin{enumerate}
\item We have:
\begin{enumerate}
\item if $\mathrm{rank}(G) = 1$, then
\begin{align*}
\left|\frac{1}{\mu_H(H_T)} \sum_{\gamma \in \Gamma_T} \psi(\gamma.y_0) - \int_{H \backslash G} \psi \, d\hat{\nu}_{y_0} \right| \ll_{D_\psi} \|\psi\|_{\Lip} R^{-\kappa};
\end{align*}
\item if $\mathrm{rank}(G) \geq 2$, then
\begin{align*}
\left|\frac{1}{\mu_H(H_T)} \sum_{\gamma \in \Gamma_T} \psi(\gamma.y_0) - \int_{H \backslash G} \psi \, d\hat{\nu}_{y_0} \right| \ll_{D_\psi} \|\psi\|_{\Lip} \bigl(T^{-\frac{1}{2\dim(G) + 5}} + R^{-\kappa}\bigr).
\end{align*}
\end{enumerate}
\item There exists $x \in X$ with
\begin{align*}
d(x_0, x) \leq e^{-\frac{\varsigma_H}{2}T}
\end{align*}
such that $Hx$ is periodic with $\vol(Hx) \leq R$.
\end{enumerate}
\end{theorem}

\subsubsection{Effective version of the Oppenheim conjecture}
In a forthcoming paper, we will investigate the applications of Theorem~\ref{thm:main equidistribution} to distribution of value of indefinite quadratic forms on integer points in all dimension $d \geq 3$ and further counting results. In particular, we will obtain an effective version of the Oppenheim conjecture for quadratic forms with at least three variables. 

\begin{theorem}\label{thm:main quadratic form}
For all integer $d \geq 3$, there exist absolute constants $A_1 > A_2 \geq 1$ and $\kappa > 0$ so that the following holds. Let $Q$ be a non-degenerate indefinite quadratic form with $d$ variables and $\det Q = 1$. For all $R$ large enough depending on $\|Q\|$ and all $T \geq R^{A_1}$, at least one of the following is true. 
\begin{enumerate}
    \item For every $s \in [-R^{\kappa}, R^{\kappa}]$, there exists a primitive vector $v \in \Z^d$ with $0 < \|v\| \leq T$ so that 
    \begin{align*}
        |Q(v) - s| \leq R^{-\kappa}.
    \end{align*}
    \item There exists an integral quadratic form $Q'$ with $\|Q'\| \leq R$ so that 
    \begin{align*}
        \|Q - \lambda Q'\| \leq T^{-\frac{1}{A_2}} \text{ where } \lambda = \det(Q')^{-\frac{1}{d}}.
    \end{align*}
\end{enumerate}
\end{theorem}

For $d = 3$, the theorem is proved by Lindenstrauss, Mohammadi, Wang and Yang \cite{LMWY25} utilizing homogeneous dynamics. For $d \geq 5$, the theorem is proved by Buterus, G\"{o}tze, Hille and Margulis \cite{BGHM22} utilizing analytic number theory and a quantitative version of Meyer's theorem. 

In contrast, the proof here will be purely based on homogeneous dynamics and we deal with all dimension $d \geq 3$ at once. 

We have the following corollary for algebraic irrational indefinite quadratic forms. 
\begin{corollary}\label{thm:main quadratic form alg}
For all integer $d \geq 3$, and all algebraic irrational indefinite quadratic form $Q$, there exists $\kappa = \kappa_Q > 0$ so that the following holds. For all $T$ large enough depending explicitly on $Q$ and all
    \begin{align*}
        s \in [-T^{\kappa}, T^{\kappa}],
    \end{align*}
    there exists a primitive vector $v \in \Z^d$ with $0 < \|v\| \leq T$ so that
    \begin{align*}
        |Q(v) - s| \leq T^{-\kappa}.
    \end{align*}
\end{corollary}

\subsection{On the proof of Theorem~\ref{thm:main equidistribution}}
We now discuss the proof of Theorem~\ref{thm:main equidistribution}. 

The strategy of the proof of Theorem~\ref{thm:main equidistribution} is similar to the general strategy in \cite{LMW22,LMWY25}. However, due to the complication from $H$ and the $\Ad(H)$-invariant complement $\mathfrak{r}$, the achievement to higher dimension is more challenging and requires novel ingredients. Before we point out the difficulties and the solutions, let us recall the general strategy developed in \cite{LMW22,LMWY25}. 

In \cite{LMW22,LMWY25}, the proof can be roughly divided into three phases: 
\begin{enumerate}
    \item Initial dimension from effective closing lemma;
    \item Improving dimension using ingredients from projection theorems;
    \item From large dimension to equidistribution. 
\end{enumerate}

\subsubsection{Difficulties}
The major difficulties in our setting come from phase~(1) and (2). Due to the complexity of $H$, especially the case where $H$ is semisimple but not simple, phase~(1) cannot be proved directly as in \cite[Proposition 6.1]{LM23}. However, thanks to the effective closing lemma for long unipotent orbits proved by Lindenstrauss, Margulis, Mohammadi, Shah and Wieser in \cite{LMMSW24}, we obtain an initial dimension. One of the crucial ingredients is an effective version of {\L}ojasiewiecz's inequality. This phase is carried out in Part~\ref{part:closinglemma}. 

The difficulty from phase~(2) is more severe. In \cite{LMWY25}, phase~(2) can be very roughly speaking further divided into three steps. First, one establishes a dimension improving result for the linear $\Ad(H)$-action on $\mathfrak{r}$. Building on this, one establishes a Margulis function estimate which provides dimension improvement in the transverse direction in $X$. With the Margulis function estimate, one runs a bootstrap process to get a high dimension (close to $\dim(\mathfrak{r})$) in the transverse direction to $H$. The major difficulty comes from the first step. 

\subsubsection{Failure of optimal projection theorems}
The $\Ad(H)$-invariant complement $\mathfrak{r}$ can be decomposed into eigenspaces $\mathfrak{r}_\lambda$ for $\mathbf{a}$. Let $\mathfrak{r}^{(\mu)} = \bigoplus_{\lambda \geq \mu} \mathfrak{r}_\lambda$ and let $\pi^{(\mu)}$ be the orthogonal projection to $\mathfrak{r}^{(\mu)}$. Let $\pi^{(\mu)}_u = \pi^{(\mu)} \circ \Ad(u)$. Roughly speaking, we say the family of maps $\pi^{(\mu)}_u:\mathfrak{r}\to \mathfrak{r}^{(\mu)}$ is optimal if for all set $A$ and almost all parameter $u \in \mathcal{B}_1^U$ one has $\dim \pi^{(\mu)}_u(A) = \min\{\dim A, \dim \mathfrak{r}^{(\mu)}\}$. The $\dim$ here stands for a suitable dimension notion for fractal-like set. 

In \cite{LMW22,LMWY25}, the first step is obtained by optimal projection theorems to \emph{all} the space $\mathfrak{r}^{(\mu)}$. Those optimal projection theorems in turns, heavily rely on the feature that $\mathfrak{r}$ is an irreducible representation of $\SL_2(\R)$ and a deep projection theorem for moment curve proved by Gan--Guo--Wang using the decoupling inequality \cite{GGW24}. 

The case of general irreducible representation however, cannot be proved in this way. In fact, for some $\mathfrak{r}^{(\mu)}$ in our setting, the family of projections $\{\pi^{(\mu)} \circ \Ad(u)\}_{u \in U}$ is \emph{never} optimal due to \emph{algebraic} obstructions, see the following example. 

\begin{example}
Let $(G, H) = (\SL_4(\R), \SO(2, 2))$. As a representation of $\mathfrak{so}(2, 2) \cong \mathfrak{sl}_2(\R) \oplus \mathfrak{sl}_2(\R)$, $\mathfrak{r}$ is isomorphic to $\mathfrak{sl}_2(\R) \otimes \mathfrak{sl}_2(\R)$. Let $e$ be the fixed vector in $\mathfrak{sl}_2(\R)$ by the adjoint action of strictly upper triangular matrices. We set $E$ to be $\mathfrak{sl}_2(\R) \otimes \R e$. We have
\begin{align*}
    2 = \dim \pi_{u}^{(2)}(E) < \min\{\dim E, 3\} = 3, \quad \forall u \in U.
\end{align*}
See Figure~\ref{fig:Non degeneracy fails}. We remark that one can construct a family of infinite similar examples via $(G, H) = (\SL_{mn}(\R), \SL_m(\R) \times \SL_n(\R))$. 
\end{example}

\begin{figure}
\centering
\includegraphics{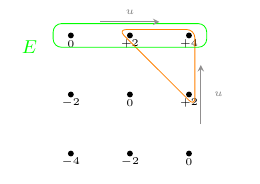}
\caption{Failure of an ideal optimal projection theorem. }
\label{fig:Non degeneracy fails}
\end{figure}

\subsubsection{Entropy distribution}
To overcome the failure of ideal optimal projection theorems, we provide the following different strategy. 

Our solution can be roughly conclude as the following and is inspired by \cite{BFLM11}. If $\mu$ corresponds to the \emph{fastest} expanding direction in $\mathfrak{r}$, we show that one can gain dimension from it. For all the other $\mu$'s, we show that one does not lose dimension from it. 

More precisely, for all $\mu$'s, we prove the following (sub)-critical estimates in Theorem~\ref{thm:subcritical entropy general irrep}. We say the family of maps $\pi^{(\mu)}_u:\mathfrak{r}\to \mathfrak{r}^{(\mu)}$ is (sub)-critical if for all set $A$ and almost all parameter $u \in \mathcal{B}_1^U$ one has 
\begin{align*}
    \dim \pi^{(\mu)}_u(A) \geq \frac{\dim \mathfrak{r}^{(\mu)}}{\dim \mathfrak{r}}\dim A.
\end{align*}

For the largest $\mu$, we establish an $\epsilon$-improvement for the above inequality using representation theory and recent developments on Bourgain's discretized projection theorem \cite{Bou10,He20}. The main theorem is Theorem~\ref{thm:supcritical}. This is the only place in this paper which uses {}\Schu{} and it will be interesting to prove a stronger projection theorem without this condition. It will also be interesting to utilize the optimal projection theorem for moment curve by Gan--Guo--Wang \cite{GGW24} to prove an optimal projection theorem for general family of projections (possibly even not from representations) under a condition similar to {}\Schu{}\footnote{In a previous draft which dealt with $(G, H) = (\SL_4(\R), \SO(2, 2))$ and $(G, H) = (\SL_4(\R), \SO(3, 1))$, we established an optimal projection theorem for the largest $\mu$ using the projection theorem proved by Gan--Guo--Wang, see \cite[Section 13]{Lin25}. }.

Combining these two pieces together via submodularity inequality, we obtain a dimension improvement in $\mathfrak{r}$. 

\subsubsection{A submodularity inequality}
We now discuss the (sub)-critical bound for $\{\pi^{(\mu)}_u\}$ where $\mu$ might not be the largest eigenvalue. Usually, the subcritical estimate can be easily obtained if the family of spaces $\{h.\mathfrak{r}^{(\mu)}\}_{h \in H}$ satisfies non-degenerate condition, see e.g. \cite[Proposition 29]{He20}. However, for general irreducible representation, the family of subspaces $\{h.\mathfrak{r}^{(\mu)}\}_{h \in H}$ can be very degenerate, see the following example. 

\begin{example}
Let $(G, H) = (\SL_4(\R), \SO(2, 2))$ be as in the previous example. The representation $\mathfrak{r}$ has dimension $9$ while the subspace $\mathfrak{r}^{(2)}$ has dimension $3$. Ideally, one expects that a generic triple of subspaces in $\{h.\mathfrak{r}^{(2)}\}_{h \in H}$ is transverse and generates the whole space $\mathfrak{r}$. However, one can compute that \emph{for all} $h_1, h_2, h_3 \in H$, we have
\begin{align*}
    \dim h_1.\mathfrak{r}^{(2)} + h_2.\mathfrak{r}^{(2)} + h_3.\mathfrak{r}^{(2)} \leq 8.
\end{align*}
We remark that same phenomenon happens in the case where $(G, H) = (\SL_4(\R), \SO(3, 1))$. 
\end{example}

In a previous draft \cite{Lin25}, the author obtained such bound in the cases where $(G, H) = (\SL_4(\R), \SO(2, 2))$ and $(G, H) = (\SL_4(\R), \SO(3, 1))$. In this paper, we obtain a (sub)-critical bound for projections in any general irreducible representations using a combinatorial method. We refer to Section~\ref{sec:outline subcritical} for a more detailed exposition. 

A key ingredient and also an interesting bi-product of the proof of (sub)-critical estimate is the following submodularity inequality. Let $F$ be a field and let $\tilde{H}$ be a connected linear algebraic group over $F$. Let $V$ be a regular irreducible representation of $\tilde{H}$ defined over $F$. 
\begin{theorem}\label{thm:dimension bound one subspace}
    For all subspaces $W, W' \subseteq V$, there exists $\tilde{h} \in \tilde{H}$ so that 
    \begin{align}\label{eqn:linear critical}
        \dim [(\tilde{h}.W) \cap W'] \leq \frac{\dim W}{\dim V} \dim W'.
    \end{align}
    Note that the set of such $\tilde{h}$ forms a Zariski open subset of $\tilde{H}$.
\end{theorem}

\begin{remark}
    This inequality is recently independently proved by B{\'e}nard--He. We refer to Remark~\ref{rem:history} for a more detailed overview on related works. 
\end{remark}

\begin{remark}
    Note that any irreducible representation of connected algebraic group factors through its reductive quotient since the set of fixed points of the unipotent radical is a sub-representation. The above theorem is essentially about connected reductive groups. 
\end{remark}

\begin{remark}
    The dimension bound in Theorem~\ref{thm:dimension bound one subspace} is optimal in general. Suppose $F$ is algebraically closed, $\tilde{H} = \tilde{H}_1 \times \tilde{H}_2$ and let $V_1$ and $V_2$ be two irreducible representation of $\tilde{H}_1$ and $\tilde{H}_2$ respectively. Let $V = V_1 \otimes V_2$ and let $v_i \in V_i$ be nonzero vectors. Since $F$ is algebraically closed, $V$ is an irreducible representation of $\tilde{H}$. Set $W = V_1 \otimes Fv_2$ and $W' = Fv_1 \otimes V_2$. For all $\tilde{h} = (\tilde{h}_1, \tilde{h}_2)$,
    \begin{align*}
        \dim [(\tilde{h}.W) \cap W'] = \dim F(v_1 \otimes (h_2.v_2)) = 1 = \frac{\dim W}{\dim V} \dim W'.
    \end{align*}
    It would be interesting to see whether the case of equality always arises in the above tensor-product-type example. 
\end{remark}

We refer to Section~\ref{sec:outline subcritical} for a more detailed exposition of the proof and the relation to the (sub)-critical estimate. 

\subsubsection{Outline}
Part~\ref{part:closinglemma} and Part~\ref{part:projection} are the main novel part of this paper and they are independent. Part~\ref{part:closinglemma} is devoted to phase~(1). Its main theorem is Theorem~\ref{thm:Closing lemma many scale}. Part~\ref{part:projection} is devoted to the linear dimension improvement result in phase~(2). Its main theorem is Theorem~\ref{thm:energy Improvement}. We refer to Part~\ref{part:phase three plus proof of main} for phase~(3) and also a sketch of an adaption for framework from \cite{LMWY25} to prove Theorem~\ref{thm:main equidistribution}. Some details for such adaption is provided in Appendix~\ref{app:decomposition initial} for readers' convenience. The estimate in Part~\ref{part:projection} utilizes estimate on the Brascamp--Lieb constant from \cite{Gre21}. For readers' convenience, we include the details for that estimate in Appendix~\ref{app:BL}. 

\subsection*{Acknowledgment}
I am extremely grateful to my advisor Amir Mohammadi for introducing the topic and for his guidance, supports, encouragements and many helpful discussions throughout. 

I would like to thank Hong Wang for suggesting using Brascamp--Lieb inequality to simplify the proof of subcritical estimate during the workshop Reading Groups in Analysis 2025 in University of Pennsylvania. I would like to thank Ruixiang Zhang for enlightening discussion on Brascamp--Lieb inequality, especially for pointing out the difference between Theorem~\ref{thm:subcritical linear general irrep intersection} and Theorem~\ref{thm:dimension bound one subspace}. I would also like to thank Lei Yang and Pengyu Yang for enlightening discussions and encouragements. I would like to thank Jiangtao Li and Jiajia Wang for helping me check a lot of computations in an early stage of this project. 

This project initiated when I was a Ph.D. student in UC San Diego. I would like to thank the department of mathematics and my friends in UC San Diego for helps, supports and encouragements. 

\section{Notations and preliminaries}\label{sec:Global prelim}
In this section, we introduce the notations and preliminaries used throughout this whole paper. New notations and preliminaries needed inside each part will be introduced in those 'preparation' sections. We remark that \emph{inside} \cref{part:projection,part:closinglemma}, we might slightly change the conventions for simplicity. We will always clarify the changes at the beginning of each section. 

\subsection{Lie groups and Lie algebras}
We use corresponding Fraktur letters for the Lie algebras of Lie groups throughout the paper. For example, $\mathfrak{s}$ is the Lie algebra of Lie group $S$. For a Lie group $S$, we use $S^\circ$ to denote its identity component under the \emph{Hausdorff topology}. For an algebraic group $\mathbf{S}$, we use $\mathbf{S}^0$ to denote its identity component under the \emph{Zariski topology}. For a group $G$ acting on a space $X$, we use $g.x$ to denote this action. Sometimes the action is clear from the context and we will use $g.x$ without introducing it explicitly. For example, for $v \in \LieG$ and $g \in G$, we write $g.v = \Ad(g)v$. 

\subsection{The data \texorpdfstring{$(G, \Gamma, H, \mathfrak{r}, \mathbf{a}, U)$}{(G, Gamma, H, r, a, U)}}
We first recall that from the introduction the definition and condition on the data $(G, \Gamma, H, \mathfrak{r}, \mathbf{a}, U)$. 

Let $\mathbf{G} \leq \SL_N$ be a connected perfect $\Q$-group, let $G = \mathbf{G}(\R)$ be its group of $\R$-points, and let $\LieG = \mathrm{Lie}(G)$ be its Lie algebra. Let $\rad(\LieG)$ be the radical of the $\LieG$, i.e., the maximal solvable ideal of $\LieG$. Let $\Gamma = G \cap 
\SL_{N}(\Z)$. It is a lattice of $G$ by the theorem of Borel--Harish-Chandra. Let $X = G/\Gamma$ and let $\mu_X$ be the unique probability Haar measure on $X$. 

We fixed an inner product on $\LieG$ induced by the standard Euclidean metric on $\mathfrak{sl}_N(\R)$ and we use $\|\cdot\|_2$ to denote the induced norm. It induces a right $G$-invariant Riemannian metric on $G$ and hence a Riemannian metric $d_X$ on $X = G/\Gamma$. 

Let $\mathbf{H}$ be a connected semisimple $\R$-subgroup of $\mathbf{G}$ and let $H = \mathbf{H}(\R)$ be its group of $\R$-points. Suppose $H$ has no compact factor. Let $\LieH$ be its Lie algebra and let $\LieH = \bigoplus_{i = 1}^k \LieH_i$ be the unique decomposition of $\LieH$ into its simple ideals. Since $\mathbf{H}$ is semisimple, there exists a decomposition of $\LieG$ into $H$-invariant complement as $\LieG = \LieH \oplus \mathfrak{r}$. We remark that $\mathfrak{r}$ is a \emph{non-trivial} representation of $\LieH$ since $\mathbf{G}$ is perfect. 

Let $\mathbf{a} \in \LieH$ be a semisimple element with $\|\mathbf{a}\|_2 = 1$ whose projection to each simple factor is non-zero. Suppose $\ad\mathbf{a}$ is diagonalizable over $\R$. Let $a_t = \exp(t\mathbf{a})$ and let $U$ be the expanding horospherical subgroup of $H$ as the following:
\begin{align*}
    U = \{u \in H: a_{-t} u a_{t} \to e \text{ as } t \to +\infty\}.
\end{align*}
Let $\LieU$ be the Lie algebra of $U$. Let $\mathcal{B}_r^{\LieU}$ be the ball of radius $r$ under the norm $\|\cdot\|_2$ centered at the origin and let $\mathcal{B}_r^U = \exp(\mathcal{B}_r^{\LieU})$. 

We fix a Cartan involution of $\LieH$ so that $\theta(\mathbf{a}) = -\mathbf{a}$. We have the Cartan decomposition $\LieH = \LieK_H \oplus \LieP_H$. Let $K_H$ be the corresponding maximal compact subgroup of $H$ so that $\LieK_H = \mathrm{Lie}(K_H)$. Let 
\begin{align*}
    \LieH = \left(\bigoplus_{\alpha \in \Phi^+} \LieH_{-\alpha}\right) \oplus \mathfrak{Z}_{\LieK}(\LieA) \oplus \LieA \oplus \left(\bigoplus_{\alpha \in \Phi^+} \LieH_{\alpha}\right)
\end{align*}
be the corresponding restricted root decomposition where $\Phi^+$ is a set of positive roots so that $\mathbf{a}$ lies in the closed positive Weyl chamber. Let $\Phi = \Phi^+ \cup -\Phi^+$ be the set of all restricted roots. 

Let $\LieH_{\lambda}$ be the eigenspaces of $\ad \mathbf{a}$ on $\LieH$, we have the following relationship between them and the restricted root decomposition:
\begin{align*}
    \LieU^+ := \LieU = \bigoplus_{\lambda > 0} \LieH_{\lambda} = \bigoplus_{\alpha \in \Phi^+, \alpha(\mathbf{a}) > 0} \LieH_{\alpha}.
\end{align*}
Note
\begin{align*}
    \LieH_0 = \mathfrak{Z}_{\LieK}(\LieA) \oplus \LieA \oplus \bigoplus_{\alpha \in \Phi, \alpha(\mathbf{a}) = 0} \LieH_{\alpha}.
\end{align*}
This is a Lie subalgebra of $\LieH$. Let 
\begin{align*}
    \LieU^- = \bigoplus_{\lambda < 0} \LieH_{\lambda} = \bigoplus_{\alpha \in \Phi^+, \alpha(\mathbf{a}) > 0} \LieH_{-\alpha}.
\end{align*}
Let $A_H = \exp(\LieA) \leq H$, and $U^- = \exp(\LieU^-)$. Let $U^+ = U$. We remark that the projection of $\mathbf{a}$ to each simple ideal factor in $\LieH$ is nontrivial, so is $\LieU$ and we can write $\LieU = \oplus_{i} \LieU_i$ where $\LieU_i < \LieH_i$ and $\LieU_i \neq \{0\}$. 

We set $\lambda_{\min}$ to be the least positive eigenvalue of $\mathbf{a}$ on $\LieU$, $\lambda_{\max}$ to be the largest positive eigenvalue of $\mathbf{a}$ on $\LieU$ and 
\begin{align*}
    \lambda_{\vol} = \log \left(\frac{\Leb(\Ad (a_t)\mathcal{B}_1^{\LieU})}{\Leb(\mathcal{B}_1^{\LieU})}\right).
\end{align*}

In Part~\ref{part:closinglemma}, we need a $\Q$-structure on $\LieH$. Let us recall the following result in \cite{Kam14}. See also \cite{Mor04,Mor15}. We remark that when $\mathbf{H}$ is $\R$-split, the existence of such basis is a classical result by Chevalley and Serre. 
\begin{proposition}[\text{\cite[Theorem 4.1]{Kam14}}]\label{pro:General Chevalley basis}
    There exists a basis $\{x_i\}_{i = 1}^h$ of $\LieH$ compatible to the restricted root space decomposition so that the structure constants $c_{i, j}^k \in \frac{1}{2}\Z$ where $[x_i, x_j] = \sum_{k = 1}^h c_{ij}^k x_k$. 
\end{proposition}

Throughout this paper, by dependence on the data $(G, \Gamma, H, \mathfrak{r}, \mathbf{a}, U)$, we allow dependence on the size of the above basis of $\LieH$. 

\subsection{On condition {}\Irrep{}}
\begin{lemma}\label{lem:structure under irrep condition}
    Suppose $\LieG$ is perfect, $\LieH$ is semisimple and the $\LieH$-invariant complement $\mathfrak{r}$ is an irreducible $\LieH$-representation. Then one of the following holds. 
    \begin{enumerate}
        \item The Lie algebra $\LieG$ is semisimple. 
        \item The Lie algebra $\LieG$ is isomorphic to $\LieH \ltimes \rad(\LieG)$ where $\rad(\LieG)$ is a non-trivial irreducible $\LieH$-representation and it is abelian. 
    \end{enumerate}
    In particular, in both cases, the complement $\mathfrak{r}$ is a \emph{non-trivial} $\LieH$-representation. 
\end{lemma}

\begin{proof}
    If $\rad(\LieG) = \{0\}$, we are in case~1. If not, we have $\LieH \subsetneq \LieH \oplus \rad(\LieG)$. Irreducibility implies that $\LieG = \LieH \oplus \rad(\LieG)$. Note that $[\rad(\LieG), \rad(\LieG)] \subsetneq \rad(\LieG)$ is also an ideal of $\LieG$ and hence representation of $\LieH$. Irreducibility implies that $[\rad(\LieG), \rad(\LieG)] = 0$. 
\end{proof}

In the latter case, the radical is a $\Q$-subalgebra of $\LieG$ of height bounded by power of height of $\mathbf{G}$, see \cite[Proposition 3.1]{MSGT23}. In this case, we always fix $\mathfrak{r} = \rad(\LieG)$. 

\subsection{On condition {}\Schu{} and {}\Prox{}} 
We now state an equivalent description for {}\Irrep{} and {}\Prox{} which is easier to check. Let 
\begin{align*}
    \Fix(U) = \{v \in \mathfrak{r}: u.v = v,  \forall u \in U\}.
\end{align*}

\begin{lemma}\label{lem:equi Fix U}
    The $H$-invariant subspace $\mathfrak{r}$ satisfies {}\Irrep{} and {}\Prox{} if and only if $\dim \Fix(U) = 1$. 
\end{lemma}

\begin{proof}
We have the decomposition $\mathfrak{r} = \bigoplus_\mu \mathfrak{r}_{\mu}$ for $\mathfrak{r}$ as eigenspaces of $\mathbf{a}$. Let $\mu_{\max}$ be the largest eigenvalue. Note that $\mathfrak{r}_{\mu_{\max}} \subseteq \Fix(U)$. If $\dim \Fix(U) = 1$, then $\mathfrak{r}$ is irreducible representation of $H$ and $\dim \mathfrak{r}_{\mu_{\max}} = 1$. 

Conversely, suppose $\mathfrak{r}$ is irreducible and $\dim \mathfrak{r}_{\mu_{\max}} = 1$ but $\dim \Fix(U) \geq 2$. Then there exists $\mu' \neq \mu_{\max}$, $w \in \mathfrak{r}_{\mu'}$ so that $w$ is fixed by $U$. Using the restricted root decomposition, $B_1^H.w \subseteq \oplus_{\mu < \mu_{\max}} \mathfrak{r}_{\mu}$ and this contradict with the assumption that $\mathfrak{r}$ is irreducible.  
\end{proof}

We now show that \Prox{} implies \Schu{}. 
\begin{lemma}\label{lem:proximal implies schubert}
    Under the condition \Irrep{}, \Prox{} implies \Schu{}. 
\end{lemma}
\begin{proof}
    Suppose not, since $\dim \Fix(U) = 1$, there exists a proper subspace $W \subseteq \mathfrak{r}$ so that $H.\Fix(U) \subseteq W$, contradict to \Irrep{}. 
\end{proof}

\subsection{Constants and \texorpdfstring{$\star$-notations}{⋆-notations}}
For $A \ll B^{\star}$, we mean there exist constants $C > 0$ and $\kappa > 0$ depending at most on the data $(G, \Gamma, H, \mathfrak{r}, \mathbf{a}, U)$ such that $A \leq C B^{\kappa}$. For $A \asymp B$, we mean $A \ll B$ and $B \ll A$. We also use the notion of $O(\cdot)$ where $f = O(g)$ is the same as $|f| \ll g$. For $A \ll_D B$, we mean there exist constant $C_D > 0$ depending on $D$ and at most on the data $(G, \Gamma, H, \mathfrak{r}, \mathbf{a}, U)$ so that $A \leq C_D B$. 

\subsection{Norms and balls}
Let $\|\cdot\|_\infty$ be the maximum norm from $\mathfrak{sl}_N(\R)$. For any subspace $W \subseteq \LieG \subset \mathfrak{sl}_N(\R)$ and $v \in W$, we define
\begin{align*}
B_r^{W}(v) = \{w \in W:\|w - v\| \leq r\}.
\end{align*}
If $v = 0$, we often omit it and denote the ball by $B_r^{W}$. 

We take a basis according to Proposition~\ref{pro:General Chevalley basis} and let $\mathsf{B}$ be balls (or more precisely boxes) according to the max-norm on coordinates with respect to this basis. We define
\begin{align*}
    \mathsf{B}^{+}_r = \mathsf{B}^U_r = \exp(\mathsf{B}^{\LieU}_r), \mathsf{B}^0_r = \exp(\mathsf{B}^{\LieH_0}_r), \mathsf{B}^{-}_r = \mathsf{B}^{U^-}_r = \exp(\mathsf{B}^{\LieU^-}_r).
\end{align*}
We set
\begin{align*}
    \mathsf{B}_r^H = \mathsf{B}^{-}_r\mathsf{B}_r^{0} \mathsf{B}^{+}_r, \quad 
    \mathsf{B}_r^{s,H} = \mathsf{B}^{-}_r\mathsf{B}_r^{0},
\end{align*}
and $\mathsf{B}_r^G = \mathsf{B}_r^H\exp(B_r^{\mathfrak{r}})$. 

We give a summary of the notations for different balls in this paper. For balls from the Riemannian metric, we use $\mathcal{B}$. For balls from the max-norm, we use $B$. For balls in $H$ according to the stable/central/unstable direction of $a_t$, we use $\mathsf{B}$. We remark that in Part~\ref{part:projection}, on $\mathfrak{r}$ we will use an inner product from Mostow's theorem which might be different from all metrics above. The changes in metric will only produce a multiplicative absolute constant. 

\subsection{Natural measures}
Note that $U = \exp(\LieU)$. Let $\tilde{m}_U$ be the push-forward of the standard Lebesgue measure on $\LieU \subseteq \mathfrak{sl}_N(\R)$ under the exponential map. Let $m_U$ be the rescaling of $\tilde{m}_U$ so that it assign $\mathsf{B}_1^U$ with measure $1$. This is a $U$-invariant measure on $U$. For the ball $\mathsf{B}_1^U \subset U$, we use $m_{\mathsf{B}_1^U}$ to denote the restriction of $m_U$ to $\mathsf{B}_1^U$. In integration, we always write $\mathrm{d}u$ instead of $\mathrm{d}m_U(u)$ or $\mathrm{d}m_{\mathsf{B}_1^U}(u)$ for simplicity. 

Similarly, we can define $m_{U^{-}}$ via the push-forward of the standard Lebesgue measure on subspaces in $\mathfrak{sl}_N(\R)$. They are Haar measures on the corresponding groups. Let $m_H$ be the corresponding Haar measure on $H$. It is proportional to the measure defined by the volume form induced by the Riemannian metric from the Cartan involution $\theta$. 

Recall that since $\Gamma$ is a lattice in $G$, there is a unique probability $G$ invariant measure $\mu_X$ on $X = G/\Gamma$. This measure is proportional to the measure defined by the volume form induced by the Riemannian metric $d_X$. 

\subsection{Commutation relations}
We record the following consequences of Baker--Campbell--Hausdorff formula. 

\begin{lemma}\label{lem:BCH}
    There exists $\eta_0 > 0$ and $C_0 > 0$ so that the following holds for all $0 < \eta \leq \eta_0$. For all $w_1, w_2 \in B_{\eta}^{\mathfrak{r}}(0)$, there exists $h \in H$ and $\bar{w} \in \mathfrak{r}$ with
    \begin{align*}
        h \in \mathsf{B}_{C_0\eta}^H \text{, and }\quad \|\bar{w} - (w_1 - w_2)\| \leq C_0\eta\|w_1 - w_2\|
    \end{align*}
    so that 
    \begin{align*}
        \exp(w_1) \exp(-w_2) = h \exp(\bar{w}).
    \end{align*}
    In particular, 
    \begin{align*}
        \frac{1}{2}\|w_1 - w_2\| \leq \|\bar{w}\| \leq \frac{3}{2}\|w_1 - w_2\|.
    \end{align*}
\end{lemma}
\begin{proof}
    See \cite[Lemma 2.1]{LM23}. 
\end{proof}

We take a further minimum so that for all $\eta \leq \eta_0$ the following holds. 
\begin{enumerate}
\item The exponential map restrict to $B_\eta^{\LieG}$ is a bi-analytic map. 
\item The maps
\begin{align}
\begin{aligned}
&\mathsf{B}_\eta^{\LieU^+} \times \mathsf{B}_\eta^{\LieH_0} \times 
\mathsf{B}_\eta^{\LieU^-} \to H\\
&(X_{\LieU^+}, X_{\LieH_0}, X_{\LieU^-}) \mapsto \exp(X_{\boldsymbol{\cdot}})\exp(X_{\Diamond})\exp(X_{\Box})
\end{aligned}
\end{align}
are bi-analytic map to their images where $(\boldsymbol{\cdot}, \Diamond, \Box)$ is any permutation of $(\LieU^+, \LieH_0, \LieU^+)$.
\item The map 
\begin{align*}
    &\mathsf{B}_\eta^{H} \times B_\eta^{\mathfrak{r}} \to G\\
    &(\mathsf{h}, X_{\mathfrak{r}}) \mapsto \mathsf{h}\exp(X_{\mathfrak{r}})
\end{align*}
is a bi-analytic map to its image.
\item Lemma~\ref{lem:BCH} holds. 
\end{enumerate}

\subsection{Injectivity radius}
For all $x \in X = G/\Gamma$, we set
\begin{align*}
    \inj(x) = \sup\{\eta: \mathsf{B}_{10^{10}C_0\eta}^G\to \mathsf{B}_{10^{10}C_0\eta}^G.x \text{ is a diffeomorphism}\}.
\end{align*}
The constant $C_0$ comes from Lemma~\ref{lem:BCH}. Taking a further minimum if necessary, we always assume that the injectivity
radius of $x$ defined using the Riemannian metric $d_X$ dominates $\inj(x)$. 

For all $\eta > 0$, let
\begin{align*}
    X_\eta = \{x \in X: \inj(x) \geq \eta\}.
\end{align*}

\subsection{A different formulation for Theorem~\ref{thm:main equidistribution}}\label{subsection:different form on balls}
Recall that we set $\mathsf{B}^U_1 = \exp(B^{\LieU}_1)$ and assign $m_U$ to be the Haar measure on $U$ so that $m_U(\mathsf{B}^U_1) = 1$. The following theorem is a slightly different formulation of Theorem~\ref{thm:main equidistribution}. 

\begin{theorem}\label{thm:main equidistribution different form}
Suppose the data $(G, \Gamma, H, \mathfrak{r}, \mathbf{a}, U)$ satisfy {}\Irrep{} and \Schu, then there exist absolute constants $\consta\label{a:main equidistribution1diff} > \consta\label{a:main equidistribution2diff} \geq 1$ and $\kappa > 0$ so that the following holds. For all $x_0 \in X$ and large enough $R$ depending explicitly on $x_0$, for any $T \geq R^{\ref{a:main equidistribution1diff}}$, at least one of the following is true. 
\begin{enumerate}
\item For all $\phi \in \mathrm{C}_c^\infty(X)$, 
\begin{align*}
\Biggl|\int_{\mathsf{B}^U_1} \phi(a_{\log T} u.x_0) \,\mathrm{d}m_U(u) - \int_X \phi \,\mathrm{d}\mu_X\Biggr| \leq \mathcal{S}(\phi) R^{-\kappa}
\end{align*}
where $\mathcal{S}(\phi)$ is a certain Sobolev norm. 
\item There exists $x \in X$ so that $H.x$ is periodic with $\vol(H.x) \leq R$ and 
\begin{align*}
    d_X(x_0, x) \leq T^{-\frac{1}{\ref{a:main equidistribution2diff}}}.
\end{align*}
\end{enumerate}
\end{theorem}

Theorem~\ref{thm:main equidistribution different form} is equivalent to Theorem~\ref{thm:main equidistribution}. The only different is the shape of ball. The equivalency can be easily deduced from a covering argument. Therefore we will focus on the study of the orbit of $a_{\log T}\mathsf{B}_1^U$ in this paper. 

\part{Closing lemma and initial dimension}\label{part:closinglemma}
The main result of this part is Theorem~\ref{thm:Closing lemma many scale}. All results in this part utilize only {}\Irrep{}. 

The following theorem asserts that for an initial point $x_1$ with suitable Diophantine condition, if the expanding time $t$ is long enough, then the measure $\mu_t = a_tm_{\mathsf{B}_1^U}.x_1$ has a small coarse dimension in the transverse direction. Moreover, the weaker Diophantine condition is provided, the longer the time is needed. Recall $\eta_0$ and $C_0$ are two constants defined in Section~\ref{sec:Global prelim} and Lemma~\ref{lem:BCH}. For any $t \geq 0$, we set
\begin{align*}
    \nu_t = (a_t)_\ast m_{\mathsf{B}_1^U}.
\end{align*} 

\begin{theorem}\label{thm:Closing lemma many scale}
    Suppose the data $(G, \Gamma, H, \mathfrak{r}, \mathbf{a}, U)$ satisfy {}\Irrep{}. Then there exist constants $\consta\label{a:closinglemma frostman} > 1$, $\constc\label{c:closinglemma frostman} > 1$, $\constE\label{e:closinglemma frostman} > 1$, $\constd\label{d:closinglemma frostman} > 1$, $\constm\label{m:closinglemma frostman} > 1$, $d_0 > 1$ and $\epsilon_1 > 0$ so that the following holds. For all $D \geq \ref{d:closinglemma frostman} + 1$, parameters $R \gg 1 $ and $\eta \ll \eta_0$ satisfying $R \gg \eta^{-\ref{e:closinglemma frostman}}$, and $x_1 \in X_\eta$, let $M = \ref{m:closinglemma frostman} + \ref{c:closinglemma frostman} D$, $t \geq M\log R$, $\mu_t = \nu_{t} \ast \delta_{x_1}$ and $\delta_0 = R^{-\frac{1}{\ref{a:closinglemma frostman}}}$. 

    Suppose that for all periodic orbit $H.x'$ with $\vol(H.x') \leq R$, we have
    \begin{align*}
        d(x_1, x') > R^{-D}.
    \end{align*}
    Then for all $y \in X_{3\eta}$, $r_H \leq 10^4C_0\eta$, $r \in [\delta_0, \eta^{d_0}]$, we have
    \begin{align*}
        \mu_t((\mathsf{B}_{r_H}^H)^{\pm 1}\exp(B_{r}^{\mathfrak{r}}).y) \ll r^{\epsilon_1}.
    \end{align*}
\end{theorem}

The above theorem is deduced from a single-scale estimate (Lemma~\ref{lem:Closing lemma one scale}) and an avoidance principle proved by Sanchez--Seong \cite{SS24} (Proposition~\ref{pro:avoidance}). The single scale estimate in turns, heavily relies on the effective closing lemma for orbit large balls in unipotent groups proved by Lindenstrauss--Margulis--Mohammadi--Shah--Wieser \cite{LMMSW24}.

We now sketch an outline for Part~\ref{part:closinglemma}. In Section~\ref{sec:preparation for discriminant}, we recall the relations between different measurements for complexity of a periodic orbit. We prove a linear algebra lemma in Section~\ref{sec:linear algebra closing lemma} and combine it with the effective closing lemma for large balls in unipotent orbits (Theorem~\ref{thm:closing lemma long unipotent}) to prove Lemma~\ref{lem:Closing lemma one scale}, a single scale version of Theorem~\ref{thm:Closing lemma many scale} in Section~\ref{sec:apply large ball unipotent}. The deduction of Theorem~\ref{thm:Closing lemma many scale} from Lemma~\ref{lem:Closing lemma one scale} is in Section~\ref{sec: Improve many scales}. 

\section{Preparation I: Measurement for complexity of periodic orbits}\label{sec:preparation for discriminant}
For a periodic orbit, there are various ways to measure its complexity. We briefly recall their relations in this section. We refer to \cite[Section 17]{EMV09} for a detailed exposition. For a periodic orbit $Hg\Gamma$ inside $X = G/\Gamma$, one can attach the following quantities. 

First, from the Riemannian metric on $X = G/\Gamma$, there is a natural volume form on all its embedded submanifolds. Therefore, we can define the volume of a periodic orbit $Hg \Gamma$. We use $\vol(Hg\Gamma)$ to denote this quantity. 

Second, if $Hg\Gamma$ is periodic, $g^{-1} H g \cap \Gamma$ is a lattice in $g^{-1} H g$ and therefore it is Zariski dense in $g^{-1} H g$. There exists a $\Q$-subgroup $\mathbf{M} \leq \mathbf{G}$ so that $g^{-1} H g = \mathbf{M}(\R)$. We set $\mathpzc{v}_M$ to be one of the primitive integer vector of the line $\wedge^{\dim \mathbf{M}} \LieM$ inside $\wedge^{\dim \mathbf{M}} \LieG$. The height of $\mathbf{M}$ is defined to be $\height(\mathbf{M}) = \|\mathpzc{v}_M\|$. However, the group $\mathbf{M}$ \emph{does} depends on the choice of representative $g$: if we change $g$ to $g\gamma$, then we need to change $\mathbf{M}$ to $\gamma^{-1}\mathbf{M} \gamma$. The length $\|\Ad(\gamma^{-1})\mathpzc{v}_M\|$ can be significantly different from $\|\mathpzc{v}_M\|$. 

We have the following direct proposition. 
\begin{proposition}\label{pro:Volume and Height}
    There exists an absolute constant $c > 1$ so that the following holds. For all $\Q$-subgroup $\mathbf{M}$ with $\mathbf{M}(\R) = g^{-1} H g$, we have
    \begin{align*}
        \vol(Hg\Gamma) \ll \height(\mathbf{M})^{c}.
    \end{align*}
\end{proposition}

We also need the following lemma relates the height of a $\Q$-subgroup and the height of its normalizer. 

\begin{proposition}\label{pro:Volume and Height and normalizer}
    There exists an absolute constant $c_0 > 1$ so that the following holds. For all semisimple $\Q$-subgroup $\mathbf{M}$ with $N_{\mathbf{G}}(\mathbf{M}) = g^{-1} \mathbf{H} g$, we have
    \begin{align*}
        \vol(Hg\Gamma) \ll \height(\mathbf{M})^{c_0}.
    \end{align*}
\end{proposition}

\begin{proof}
    It follows from \cite[Lemma 4.8]{LMMS} and the fact that the normalizer of a semisimple subgroup in a semisimple group is reductive, cf. \cite[Chapter VII, \S 1, no.3, Lemma 2]{Bou05}. 
\end{proof}

\section{A linear algebra lemma}\label{sec:linear algebra closing lemma}
The main result for this section is Lemma~\ref{lem:Closing lemma Linear Algebra}. It is devoted to deduce Lemma~\ref{lem:Closing lemma one scale} from the output of \cite[Theorem 1.4]{LMMSW24}. The main ingredient in this section is an effective {\L}ojasiewiecz's inequality, see Theorem~\ref{thm:Lojasiewiecz}. 

\subsection{\texorpdfstring{Isolation of $\LieU$ from proper ideals of $\LieG$}{Isolation of Lie(U) from proper ideals of Lie(G)}}
We record the following lemma in this subsection. It follows directly from the unique ideal decomposition of semisimple Lie algebras. 
\begin{lemma}\label{lem:U isolate from proper ideal}
    There exists $c_1 > 0$ so that the following holds. Suppose $\mathfrak{I}$ is either a proper ideal of $\LieH$ or a proper ideal of $\LieG$ containing $\rad(\LieG)$. Let $v_{\mathfrak{J}}$ be a vector in the line $\wedge^{\dim \mathfrak{I}} \mathfrak{I} \subseteq \wedge^{\dim \mathfrak{I}} \LieG$. Then we have
    \begin{align*}
        \sup_{\mathpzc{z} \in \LieU, \|\mathpzc{z}\| = 1}\|\mathpzc{z} \wedge v_{\mathfrak{J}}\| \geq c_1 \|v_{\mathfrak{J}}\|.
    \end{align*}
\end{lemma}

\subsection{An equivariant projection and its consequence} 
We record an equivariant projection from \cite{EMV09}. In this section, $S$ is a semisimple subgroup of $G$ with Lie algebra $\mathfrak{s}$. Let $\bar{v}_S$ be a \textit{unit vector} in the line corresponding to $\mathfrak{s}$ in $\wedge^{\dim \mathfrak{s}} \LieG$. 
\begin{lemma}\label{lem:equivariant proj}
    There exists a neighborhood $\mathcal{N}_S$ of $\bar{v}_S$ and a projection map $\Pi: \mathcal{N}_S \to G.\bar{v}_S$ so that the following holds. 
    For all $v \in \mathcal{N}_S$ with $g.v = v$ for some $g \in B_1^G$, we have $g.\Pi(v) = \Pi(v)$. 
\end{lemma}

\begin{proof}
    See \cite[Lemma 13.2]{EMV09}. 
\end{proof}

The following consequence is useful in later proofs. 

\begin{proposition}\label{pro:conjugate of ss group}
    Let $\mathbf{S}$ be a connected semisimple $\R$-subgroup of $\mathbf{G}$ with Lie algebra $\mathfrak{s}$ and let $\mathbf{S'}$ be a connected $\R$-subgroup of $\mathbf{G}$ with Lie algebra $\mathfrak{s}'$. Let $S = \mathbf{S}(\R)$ and $S' = \mathbf{S'}(\R)$. Suppose the radical of $\mathbf{S'}$ is unipotent and $\dim \mathfrak{s} = \dim \mathfrak{s}'$. Let $\bar{v}_{S'}$ be a \textit{unit vector} in the line corresponding to $\mathfrak{s}'$ in $\wedge^{\dim \mathfrak{s}} \LieG$. 
    
    If $\bar{v}_{S'} \in \mathcal{N}_S$ as in Lemma~\ref{lem:equivariant proj}, then there exists $g' \in G$ so that $\|g' - \Id\| \ll \|\bar{v}_{S'} - \bar{v}_{S}\|$ and 
    \begin{align*}
        \mathbf{S'} \subseteq g'N_{\mathbf{G}}(\mathbf{S})(g')^{-1}.
    \end{align*}
\end{proposition}

\begin{proof}
    Note that since the radical of $\mathbf{S'}$ is unipotent, for all $s' \in S'$, $s'.v_{S'} = v_{S'}$. By Lemma~\ref{lem:equivariant proj}, for all $s' \in B_1^{S'}$ we have $s'.\Pi(v_{S'}) = \Pi(v_{S'}) = g'.v_{S}$. Since $(S')^\circ$ is generated by $B_1^{S'}$, for all $s' \in (S')^\circ$ we have
    \begin{align*}
        s' \in N_{G}(g' S (g')^{-1}) = g'N_{G}(S)(g')^{-1}.
    \end{align*}
    Therefore, $\mathbf{S'} \cap N_{\mathbf{G}}(\mathbf{S})$ is a variety of same dimension as $\mathbf{S'}$. By Zariski connectedness of $\mathbf{S'}$, $\mathbf{S'} \leq N_{\mathbf{G}}(\mathbf{S})$. 
\end{proof}

\subsection{Effective {\L}ojasiewiecz's inequality and its consequence}
The proof of Lemma~\ref{lem:Closing lemma Linear Algebra} uses an effective version of {\L}ojasiewiecz's inequality \cite{Loj59}. It asserts that the distance between a point to zero locus of an analytic function can be controlled by the value of that function. We use an effective version of this statement for polynomials proved in \cite{Sol91}, see also \cite[Theorem 3.2]{LMMSW24}. The height of a polynomial in $\Z[x_1, \ldots, x_n]$ is defined to be the maximum of its coefficients in absolute value. 
\begin{theorem}[Solern\'{o} \cite{Sol91}]\label{thm:Lojasiewiecz}
    For any $d \in \N$, there exists $C(d) > 1$ with the following property. 

    Let $h > 1$ and let $f_1, \ldots, f_r \in \Z[x_1, \ldots, x_n]$ have degree at most $d$ and height at most $h$. Let $\mathcal{V} \subseteq \R^n$ be the zero locus of $f_1, \ldots, f_r$. Then for $w \in \R^n$
    \begin{align*}
        \min\{1, d(w, \mathcal{V})\} \ll_d (1 + \|w\|)^{C(d)} h^{C(d)} \max_{1 \leq i \leq r} |f_i(w)|^{\frac{1}{C(d)}}
    \end{align*}
    where $d(w, \mathcal{V}) = \inf_{v \in \mathcal{V}} d(w, v)$. 
\end{theorem}

The effective {\L}ojasiewiecz's inequality allow us to approximate $\LieM$ by a subalgebra $\mathfrak{w}$ which is also a $\LieH$-module. We record this consequence in the following lemma. It holds in the following more general setting. 

Suppose $\mathbf{G} \subseteq \SL_N$ is an algebraic group over $\Q$ and $\mathbf{H} \leq \mathbf{G}$ is a semisimple $\R$-subgroup. Let $\LieG $ and $\LieH$ be their Lie algebras respectively. Suppose $\mathbf{H}$ has no compact simple factor. Let $\mathbf{a}$ be a non-zero semisimple element whose projection to each factor is nonzero and $\ad(\mathbf{a})$ is diagonalizable over $\R$. Let $a_t = \exp(t\mathbf{a})$ and let
\begin{align*}
    U = \{u \in \mathbf{H}(\R):a_{-t} ua_t \to e \text{ as } t \to +\infty\}.
\end{align*}
\begin{lemma}\label{lem:coro of Loj in linear algebra}
    Suppose $(\LieG, \LieH)$ satisfy the normalizer $\mathfrak{n}_{\LieG}(\LieH) = \LieH$. 
    
    There exists an absolute constant $\constc\label{c:linear algebra Loj} > 1$ so that the following holds for all $\tilde{\eta} \in (0, \frac{1}{100})$, $\tilde{R} \gg \tilde{\eta}^{-2}$, and $\tilde{t} \geq \ref{c:linear algebra Loj}(\ref{c:linear algebra Loj} + 1)\log R$. 
    
    Suppose there exist a non-trivial connected proper $\R$-subgroup $\mathbf{M}$ of $\mathbf{G}$. Let $M = \mathbf{M}(\R)$ and let $v_M$ be a non-zero vector in the line corresponding to $\LieM = \mathrm{Lie}(M)$ in $\wedge^{\dim M} \LieG$. Suppose
    \begin{align}
        \|v_M\| \geq \tilde{\eta},
    \end{align}
    \begin{align}
        \sup_{u \in \mathsf{B}_1^{U}}\|a_{\tilde{t}} u v_M\| \leq \tilde{R}.
    \end{align}
    Then we have
    \begin{align*}
        \|v_M\| \ll \tilde{R}.
    \end{align*}
    Moreover, there exists a subalgebra $\mathfrak{w} < \LieG$ with $\dim \mathfrak{w} = \dim \LieM$ satisfies the following. 
    \begin{enumerate}
        \item The subalgebra $\mathfrak{w}$ is a $\LieH$-module. 
        \item Let $v_{\mathfrak{w}}$ be a vector in the line $\wedge^{\dim M} \mathfrak{w}$. We have
        \begin{align*}
            \|x \wedge v_{\mathfrak{w}}\| \ll e^{-\frac{\tilde{t}}{\ref{c:linear algebra Loj}} } \tilde{R}^{\ref{c:linear algebra Loj}}\|x\|\|v_{\mathfrak{w}}\|, \quad \forall x \in \LieM.
        \end{align*}
    \end{enumerate}
\end{lemma}

To apply Theorem~\ref{thm:Lojasiewiecz}, we need a $\Q$-structure on $\LieH$. When $\LieH$ is $\R$-split, this is the classical result by Chevalley and Serre. In general, we recall that from \cite{Kam14} (recorded in Proposition~\ref{pro:General Chevalley basis}), there exists a basis of $\{x_i\}_{i = 1}^h$ of $\LieH$ with structure constants $c_{i, j}^k \in \Q$ where $[x_i, x_j] = \sum_{k = 1}^h c_{ij}^k x_k$. 

\begin{proof}[Proof of Lemma~\ref{lem:coro of Loj in linear algebra}]
    Similar to Section~\ref{sec:Global prelim}, we set $\lambda_{\min}$ be the least positive expanding rate of $\Ad(a_t)$ on $\LieH$. Let $m = \dim M$. Let $V = \wedge^m \LieG$. This is a representation of $H$ with the following decomposition
    \begin{align*}
        V = V^H \oplus V^{\mathrm{non}}
    \end{align*}
    where $V^H$ consists of fixed vectors of $H$ and $V^{\mathrm{non}}$ is the direct sum of non-trivial sub-representations. This decomposition is unique. We write $v_M = v_0 + v^{\mathrm{non}}$ according to this decomposition. 
    Since
    \begin{align*}
        \sup_{u \in \mathsf{B}_1^U}\|a_{\tilde{t}} u v_M\| \asymp \|v_0\| + \sup_{u \in \mathsf{B}_1^U}\|a_{\tilde{t}} u.v^{\mathrm{non}}\| &{}\leq \tilde{R},
    \end{align*}
    we have 
    \begin{align*}
        \|v^{\mathrm{non}}\| \ll \tilde{R}e^{-\lambda_{\min}\tilde{t}}, \qquad \|v_0\| \leq \tilde{R}.
    \end{align*}
    where the implied constants depend only on the representation $V$, cf. \cite[Section 5]{Sh96}. Therefore, we have $\|v_M\| \ll \tilde{R}$, which proves the first assertion. 

    Let $\mathcal{V}$ be the variety in $\LieG^{h + m}$ consisting of tuples $(x_1, \ldots, x_h, w_1, \ldots, w_m)$ satisfying the following polynomial equations:
    \begin{align}
    \begin{aligned}
        &[x_i, x_j] = \sum_{k = 1}^h c_{ij}^k x_k, \quad i, j = 1, \ldots, h,\\
        &\sum_{j = 1}^m w_1 \wedge [x_i, w_j] \wedge \cdots \wedge w_m = 0, \quad i = 1, \ldots, h, \\
        &[w_k, w_l] \wedge w_1 \wedge \cdots \wedge w_m = 0, \quad k, l = 1, \ldots, m.
    \end{aligned}
    \end{align}
    Note that since $\LieG$ is a $\Q$-Lie algebra in $\mathfrak{sl}_N$ and $c_{ij}^k \in \Q$, the above polynomials have rational coefficients with height bounded by constant depending explicitly on $\LieG$ and $\LieH$. Let $C$ be the constant from Theorem~\ref{thm:Lojasiewiecz} applying to the above polynomials. 

    Since $\|v_{M}\| \ll \tilde{R}$, there exists basis $\{w_j^{(0)}\}_{j = 1}^m$ of $\LieM$ with size $\|w_j^{(0)}\| \ll \tilde{R}$ for all $j = 1, \ldots, m$. Let $x_i^{(0)}$ be a basis of $\LieH$ from Proposition~\ref{pro:General Chevalley basis}. We have
    \begin{align*}
        d(((x_i^{(0)})_{i = 1}^h, (w_j^{(0)})_{j = 1}^m), \mathcal{V}) \ll \tilde{R}^{C} e^{-\frac{\lambda_{\min}}{C} \tilde{t}}.
    \end{align*}

    Let $((x_i^{(1)})_{i = 1}^h, (w_j^{(1)})_{j = 1}^m) \in \mathcal{V}$ so that 
    \begin{align*}
        \|x_i^{(0)} - x_i^{(1)}\| \ll \tilde{R}^{C} e^{-\frac{\lambda_{\min}}{C} \tilde{t}}, \|w_j^{(0)} - w_j^{(1)}\| \ll \tilde{R}^{C} e^{-\frac{\lambda_{\min}}{C} \tilde{t}}.
    \end{align*}
    If $\tilde{t} \gg_{\LieG} \log \eta$ and $\tilde{R} \gg \tilde{\eta}^{-2}$, we have 
    \begin{align*}
        \|x_1^{(1)} \wedge \cdots \wedge x_h^{(1)}\| \gg \|x_1^{(0)} \wedge \cdots \wedge x_h^{(0)}\| > 0, \|w_1^{(1)} \wedge \cdots \wedge w_m^{(1)}\| \gg \tilde{\eta} > 0
    \end{align*}
    and hence 
    \begin{align*}
        \dim \span\{x_1^{(1)} , \ldots, x_h^{(1)}\} = h, \quad \dim \span\{w_1^{(1)}, \ldots,  w_m^{(1)}\} = m.
    \end{align*}

    Let $\LieH^{(1)} = \Span \{x_1^{(1)}, \ldots, x_h^{(1)}\}$. It is a subalgebra of $\LieG$ isomorphic to $\LieH$. By Proposition~\ref{pro:conjugate of ss group}, there exists $\|g^{(1)} - \Id\| \ll \tilde{R}^{C} e^{-\frac{\lambda_{\min}}{C} \tilde{t}}$ so that $\LieH^{(1)} \subseteq \Ad(g^{(1)}) \mathfrak{n}_{\LieG}(\LieH)$. Since $\mathfrak{n}_{\LieG}(\LieH) = \LieH$, we have $\LieH^{(1)} = \Ad(g^{(1)})\LieH$ by dimension counting. 

    Let $\mathfrak{w}^{(1)} = \Span \{w_1^{(1)}, \ldots, w_m^{(1)}\}$. It is both a subalgebra of $\LieG$ with $\dim \mathfrak{w} = \dim \LieM$ and a $\LieH^{(1)}$-module. Let $\mathfrak{w} = \Ad(g^{(1)})^{-1} \mathfrak{w}^{(1)}$. It is both a subalgebra of $\LieG$ and a $\LieH$-module. Property~(2) follows directly from the construction if we set $\ref{c:linear algebra Loj} = 2C \max\{\frac{1}{\lambda_{\min}}, 1\}$.     
\end{proof}

\subsection{Almost invariant subalgebras of semisimple Lie algebra}
We recall the following lemma from \cite{LMMSW24}. For any subspace (not necessarily $\Q$-subspace) $\mathfrak{s} \subseteq \LieG$, we define $\hat{v}_{\mathfrak{s}}$ to be the corresponding point in $\mathbb{P}(\wedge^{\dim \mathfrak{s}} \LieG)$. For any $0 < r \leq \dim \LieG$, we equip $\mathbb{P}(\wedge^{r} \LieG)$ with the Fubini-Study metric $\mathrm{d}$ where $\mathrm{d}(\hat{v}, \hat{w})$ is the angle between the corresponding lines in $\mathbb{P}(\wedge^{r} \LieG)$. 

\begin{lemma}[\text{\cite[Lemma 3.3]{LMMSW24}}]\label{lem:isolation to ideals with same dimension}
    Let $\mathfrak{s}$ be a semisimple Lie algebra and let $\|\cdot\|$ be a norm on it. There exists $c_2$ depending only on $\mathfrak{s}$ and $\|\cdot\|$ so that the following holds. Suppose $\LieH_1$ is a Lie subalgebra of $\mathfrak{s}$ for which there exists a Lie ideal $\LieH_2 \triangleleft \mathfrak{s}$ with $\dim \LieH_1 = \dim \LieH_2$ and 
    \begin{align*}
        \mathrm{d}(\hat{\mathpzc{v}}_{\LieH_1}, \hat{\mathpzc{v}}_{\LieH_2}) \leq c_2.
    \end{align*}
    Then $\LieH_1 = \LieH_2$. In particular, $\LieH_1$ is a Lie ideal. 
\end{lemma}

\begin{remark}
    We remark that in \cite{LMMSW24} they assume $\mathfrak{s}$ to be a $\Q$-Lie algebra so that $c_2$ depends only on the height of $\mathfrak{s}$. In this paper, we will apply this lemma only to $\LieG$ (when $\rad(\LieG) = \{0\}$), $\LieG/\rad(\LieG)$, and $\LieH$. Both the Lie algebra $\LieG$ and $\LieG/\rad(\LieG)$ has a natural $\Q$-structure from the setting. We remark that by \cite[Proposition 3.1]{MSGT23}, there exists a Levi factor of $\LieG$ over $\Q$ with height bounded by power of height of $\LieG$. For the Lie algebra $\LieH$, we utilize the $\Q$-structure from the pre-set Chevalley basis from Proposition~\ref{pro:General Chevalley basis}. 
\end{remark}

\subsection{A linear algebra lemma}
The following lemma is the main result of this section. 
\begin{lemma}\label{lem:Closing lemma Linear Algebra}
    Suppose the data $(G, \Gamma, H, \mathfrak{r}, \mathbf{a}, U)$ satisfy {}\Irrep{}. Then there exists a constant $\constc\label{c:linear algebra} > 0$ so that the following holds for all $\tilde{\eta} \in (0, 1)$, $\tilde{R} \gg \tilde{\eta}^{-2}$, and $\tilde{t} \geq \ref{c:linear algebra}(\ref{c:linear algebra} + 1)\log R$. 
    
    Suppose there exist a non-trivial connected proper $\R$-subgroup $\mathbf{M}$ of $\mathbf{G}$ with the following properties. Suppose the radical of $\mathbf{M}$ is unipotent. Let $M = \mathbf{M}(\R)$ and let $v_M$ be a non-zero vector in the line corresponding to $\LieM = \mathrm{Lie}(M)$ in $\wedge^{\dim M} \LieG$. Suppose
    \begin{align}
        \|v_M\| \geq \tilde{\eta},
    \end{align}
    \begin{align}
        \sup_{u \in \mathsf{B}_1^{U}}\|a_{\tilde{t}} u v_M\| \leq \tilde{R}.
    \end{align}
    Then we have
    \begin{align*}
        \|v_M\| \ll \tilde{R}.
    \end{align*}
    
    Moreover, for all $A \geq 1$, if $\tilde{t} \geq A\ref{c:linear algebra}(\ref{c:linear algebra} + 1) \log \tilde{R}$ and 
    \begin{align}\label{eqn:linear algebra lemma closed to unipotent}
        \sup_{\mathpzc{z} \in B_1^{\LieU}, u \in \mathsf{B}_1^U} \|\mathpzc{z} \wedge (a_{\tilde{t}} u v_M)\| \leq e^{-\frac{\tilde{t}}{A}}\tilde{R},
    \end{align}
    then at least one of the following holds. 
    \begin{enumerate}
        \item We have $\{0\} \neq \rad(\LieG) \not\subseteq \LieM$ and 
        \begin{align*}
            \sup_{\mathpzc{r} \in B_1^{\rad(\LieG)}}\|\mathpzc{r} \wedge v_M\| \leq \tilde{R}^{\ref{c:linear algebra}} e^{-\frac{\tilde{t}}{\ref{c:linear algebra}} } \|v_M\|.
        \end{align*}
        \item There exists $g' \in G$ with $\|g' - I\| \leq \tilde{R}^{\ref{c:linear algebra}} e^{-\frac{\tilde{t}}{\ref{c:linear algebra}}}$ so that 
        \begin{align*}
             g' N_G(M) (g')^{-1} = H.
        \end{align*}
    \end{enumerate}
\end{lemma}

\begin{proof}[Proof of Lemma~\ref{lem:Closing lemma Linear Algebra}]
    Let $\ref{c:linear algebra} = \ref{c:linear algebra Loj} + 1$. Applying Lemma~\ref{lem:coro of Loj in linear algebra} to $(\LieG, \LieH)$, we have that $\|v_M\| \ll \tilde{R}$. This proves the first assertion. From now on we assume \cref{eqn:linear algebra lemma closed to unipotent} holds for the given $\tilde{t}, \tilde{R}$ and $A$. 
    
    We start by showing that $\LieM$ is not a proper ideal of $\LieG$ containing $\rad(\LieG)$ nor an ideal of $\LieH$. Indeed, if so, by \cref{eqn:linear algebra lemma closed to unipotent}, we have
    \begin{align*}
        \sup_{\mathpzc{z} \in B_1^{\LieU}} \|\mathpzc{z} \wedge v_M\| = \sup_{\mathpzc{z} \in B_1^{\LieU}, u \in \mathsf{B}_1^U} \|\mathpzc{z} \wedge (a_{\tilde{t}} u v_M)\| \leq e^{-\frac{\tilde{t}}{A}}\tilde{R}\leq \tilde{R}^{-1}.
    \end{align*}
    We get a contradiction if $\tilde{R} \gg \tilde{\eta}^{-2}$ and the implied constant is large enough depending only on $c_1$ in Lemma~\ref{lem:U isolate from proper ideal}. 

    By Lemma~\ref{lem:coro of Loj in linear algebra} (applying to $(\LieG, \LieH)$), there exists a subalgebra $\mathfrak{w} < \LieG$ with $\dim \mathfrak{w} = \dim \LieM$ so that it is also a $\LieH$-module and
    \begin{align*}
            \|x \wedge v_{\mathfrak{w}}\| \leq e^{-\frac{\tilde{t}}{\ref{c:linear algebra}} } \tilde{R}^{\ref{c:linear algebra}}\|x\|\|v_{\mathfrak{w}}\|, \quad \forall x \in \LieM
    \end{align*}
    where $v_{\mathfrak{w}}$ is any vector in the line $\wedge^m \mathfrak{w}$. Here we take $\tilde{R}$ be large enough to take care of implied constant in Lemma~\ref{lem:coro of Loj in linear algebra}. 

    \medskip
    \noindent\textit{Case 1.} The subalgebra $\mathfrak{w} \not\subseteq \LieH$, we will show property~(1) holds in this case. We have $\LieH + \mathfrak{w} = \LieG$ due to irreducibility of $\mathfrak{r}$. Since $\mathfrak{w}$ is both a subalgebra and a $\LieH$-module, it is an ideal of $\LieG$. Note that $\LieH$ surjects to $\LieG/\mathfrak{w}$, $\LieG/\mathfrak{w}$ is semisimple and $\rad(\LieG) \subseteq \mathfrak{w}$. In particular, if $\rad(\LieG) \neq \{0\}$, this implies that 
    \begin{align*}
        \sup_{\mathpzc{r} \in B_1^{\rad(\LieG)}} \|\mathpzc{r} \wedge v_M\| \leq e^{-\frac{\tilde{t}}{\ref{c:linear algebra}} } \tilde{R}^{\ref{c:linear algebra}}\|v_{M}\|.
    \end{align*}
    It suffices to show that in this case we have $\{0\} \neq \rad(\LieG) \not\subseteq \LieM$. 

    Suppose $\rad(\LieG) = \{0\}$, then $\LieG$ is semisimple. By Lemma~\ref{lem:isolation to ideals with same dimension} applying to $\LieG$, if $R$ is large enough, we have $\LieM = \mathfrak{w} \triangleleft \LieG$ and we get a contradiction since $\LieM$ is not a proper ideal of $\LieG$. 
    
    Suppose $\rad(\LieG) \subseteq \LieM$, then let $\overline{\LieM}$ and $\overline{\mathfrak{w}}$ be the image of $\LieM$ and $\mathfrak{w}$ in $\LieG /\rad(\LieG)$. Note that $\dim\overline{\LieM} = \dim\overline{\mathfrak{w}}$. By Lemma~\ref{lem:isolation to ideals with same dimension} applying to $\LieG/\rad(\LieG)$, if $R$ is large enough, we have $\LieM = \mathfrak{w} \triangleleft \LieG$. Similar argument shows this also contradict to Lemma~\ref{lem:U isolate from proper ideal}. 

    \medskip
    \noindent\textit{Case 2.} The subalgebra $\mathfrak{w} \subseteq \LieH$, we will show property~(2) holds in this case. Since $\mathfrak{w}$ is $\LieH$-module, it is an ideal of $\LieH$. Let $\pi_{\LieH}$ be the projection from $\LieG$ to $\LieH$ according to the decomposition $\LieG = \LieH \oplus \mathfrak{r}$. 
    
    We claim that $\pi_{\LieH}|_{\LieM}$ is injective if $\tilde{R}$ is large enough. Indeed, for all $x \in \LieM$ with $\|x\| = 1$, there exists $x' \in \mathfrak{w}$ so that $\|x - x'\| \ll \tilde{R}^{-1}$. Therefore, $\|\pi_{\LieH}(x) - \pi_{\LieH}(x')\| = \|\pi_{\LieH}(x) - x'\|\ll \tilde{R}^{-1}$. This shows the injectivity if $\tilde{R}$ is large enough. Moreover, by dimension counting, $\pi_{\LieH}(\LieM) = \mathfrak{w}$. 

    We now show that $\mathfrak{w}$ is not an ideal of $\LieG$. Suppose not, $\mathfrak{w} \triangleleft \LieG$. Since $\mathfrak{r}$ is a non-trivial representation of $\LieH$, $\mathfrak{w} \subsetneq \LieH$. If $\rad(\LieG) = \{0\}$, $\LieG$ is semisimple. Lemma~\ref{lem:isolation to ideals with same dimension} applying to $\LieG$ directly implies that $\LieM = \mathfrak{w} \triangleleft \LieG$ if $R$ is large enough. If $\rad(\LieG) \neq \{0\}$, we have $[\mathfrak{w}, \rad(\LieG)] = 0$ and therefore $\mathfrak{w} = [\mathfrak{w}, \mathfrak{w}] = [\pi_{\LieH}(\LieM), \pi_{\LieH}(\LieM)] \subseteq \LieM$. The last inclusion follows from Lemma~\ref{lem:structure under irrep condition} that $\rad(\LieG)$ is abelian. A dimension comparison shows that $\mathfrak{w} = \LieM \triangleleft \LieG$. In both situation, $\mathfrak{w} = \LieM \triangleleft \LieG$ and we get a contradiction since $\LieM$ is not a proper ideal of $\LieG$. 
    
    By irreducibility of $\mathfrak{r}$, we have $\mathfrak{n}_{\LieG}(\mathfrak{w}) = \LieH$. Applying Lemma~\ref{pro:conjugate of ss group} to $\mathfrak{s} = \mathfrak{w}$ and $\mathfrak{s}' = \LieM$, there exists $g'$ with $\|g' - \Id\| \leq e^{-\frac{\tilde{t}}{\ref{c:linear algebra}} } \tilde{R}^{\ref{c:linear algebra}}$ so that 
    \begin{align*}
        \LieM \subseteq \Ad(g')^{-1} \mathfrak{n}_{\LieG}(\mathfrak{w}) = \Ad(g')^{-1} \LieH.
    \end{align*}
    Applying Lemma~\ref{lem:isolation to ideals with same dimension} to $\Ad(g') \LieM$ and $\mathfrak{w}$ in $\LieH$, we have $\Ad(g')^{-1} \LieM = \mathfrak{w}$ and 
    \begin{align*}
        g' N_G(M) (g')^{-1} = H.
    \end{align*}
\end{proof}

\section{Applying closing lemma for large balls in unipotent orbits}\label{sec:apply large ball unipotent}
This section is devoted to prove the following single scale version of Theorem~\ref{thm:Closing lemma many scale}. 
\begin{lemma}\label{lem:Closing lemma one scale}
    Suppose the data $(G, \Gamma, H, \mathfrak{r}, \mathbf{a}, U)$ satisfy {}\Irrep{}. Then there exist constants $\consta\label{a:closing lemma frostman one scale}, \constc\label{c:closing lemma frostman one scale}, \constm\label{m:closing lemma frostman one scale} > 1$, $\constE\label{e:closing lemma frostman one scale}> \ref{a:closing lemma frostman one scale} > 1$ and $\epsilon_2 > 0$ so that the following holds. For all $D > 0$, parameters $R \gg 1 $ and $\eta \ll \eta_0$ satisfying $R \gg \eta^{-\ref{e:closing lemma frostman one scale}}$, and $x_1 \in X_\eta$, let $M = \ref{m:closing lemma frostman one scale} + \ref{c:closing lemma frostman one scale} D$, $t = M\log R$, $\mu_t = \nu_{t} \ast \delta_{x_1}$ and $\delta = R^{-\frac{1}{\ref{a:closing lemma frostman one scale}}}$. 

    Suppose that for all periodic orbit $H.x'$ with $\vol(H.x') \leq R$, we have
    \begin{align*}
        d(x_1, x') > R^{-D}.
    \end{align*}
    Then for all $y \in X_{3\eta}$ and all $r_H \leq 10^5C_0\eta$, we have
    \begin{align*}
        \mu_t((\mathsf{B}_{r_H}^H)^{\pm 1}\exp(B_{\delta}^{\mathfrak{r}}).y) \leq \delta^{\epsilon_2}.
    \end{align*}
\end{lemma}
As indicated in the introduction, the proof of the above lemma relies on the effective closing lemma for orbit of large balls in unipotent groups proved by Lindenstrauss--Margulis--Mohammadi--Shah--Wieser \cite{LMMSW24}. Recall that For any $\Q$-subgroup $\mathbf{M}$ of $\mathbf{G}$ with Lie algebra $\LieM$, we define $\mathpzc{v}_M \in \wedge^{\dim \LieM} \LieG$ to be one of the primitive integral vector in the line $\wedge^{\dim \LieM} \LieM$. For any $0 < r \leq \dim \LieG$, we set $\mathrm{d}$ to be the Fubini-Study metric on $\mathbb{P}(\wedge^{r} \LieG)$. 

\begin{theorem}[Lindenstrauss--Margulis--Mohammadi--Shah--Wieser]\label{thm:closing lemma long unipotent}
    
    There exist constants $\consta\label{a:long unipotent1}, \consta\label{a:long unipotent2} > 1$ and $\constE\label{e:long unipotent} > 1$ depending on $\mathbf{G}$ and $\mathbf{a}$ so that the following holds. Let $\tau \in (0, 1)$ and $e^t > S \geq \ref{e:long unipotent} \tau^{-\ref{a:long unipotent1}}$. Let $x = g\Gamma \in X_\tau$ be a point. 

    Suppose there exists $\mathcal{E} \subseteq \mathsf{B}_1^U$ with the following properties.
    \begin{enumerate}
        \item $m_U(\mathcal{E}) > S^{-\frac{1}{\ref{a:long unipotent1}}}$.
        \item For any $u, u' \in \mathcal{E}$, there exists $\gamma \in \Gamma$ with \begin{align*}
            \|a_t u a_{-t} g \gamma g^{-1} a_t (u')^{-1} a_{-t} \| \leq S^{\frac{1}{\ref{a:long unipotent1}}},\\
            \mathrm{d}(a_t u a_{-t} g \gamma g^{-1} a_t (u')^{-1} a_{-t}.\hat{v}_{\mathfrak{h}}, \hat{v}_{\mathfrak{h}}) \leq S^{-1}.
        \end{align*}
    \end{enumerate}
    Then one of the following is true. 
    
    \begin{enumerate}
        \item There exists a non-trivial proper $\Q$-subgroup $\mathbf{M}$ with radical to be unipotent so that 
        \begin{align*}
        \sup_{u \in \mathsf{B}_1^U} \|a_t u a_{-t} g.\mathpzc{v}_M\| {}&\leq S^{\ref{a:long unipotent2}}, \\
        \sup_{\mathpzc{z} \in B_1^{\LieU}, u \in \mathsf{B}_1^U} \|\mathpzc{z} \wedge (a_t u a_{-t} g.\mathpzc{v}_M)\| {}&\leq e^{-\frac{t}{\ref{a:long unipotent2}}}S^{\ref{a:long unipotent2}}.
        \end{align*}
        \item There exists a non-trivial proper normal $\Q$-subgroup $\mathbf{M}$ with 
        \begin{align*}
            \sup_{\mathpzc{z}\in B_1^{\LieU}} \|\mathpzc{z} \wedge \mathpzc{v}_{M}\| \leq S^{-\frac{1}{\ref{a:long unipotent2}}}.
        \end{align*}
    \end{enumerate}
\end{theorem}

\begin{proof}[Proof of Lemma~\ref{lem:Closing lemma one scale}]
We prove the lemma for $\mathsf{B}_{r_H}^H$. The proof for $(\mathsf{B}_{r_H}^H)^{-1}$ is similar. In this proof we write $|g| = \max\{\|g\|, \|g\|^{-1}\}$ for simplicity. 

Let $\ref{c:linear algebra}$ be the constant coming from Lemma~\ref{lem:Closing lemma Linear Algebra} and $c_0$ be the constant coming from the comparison between volume and height in Proposition~\ref{pro:Volume and Height and normalizer}. Let $\ref{a:long unipotent1} > 1$, $\ref{a:long unipotent2} > 1$ and $\ref{e:long unipotent}$ be as in Theorem~\ref{thm:closing lemma long unipotent}. Let $\ref{m:closing lemma frostman one scale} = \ref{a:long unipotent2}(\ref{c:linear algebra}(\ref{c:linear algebra} + 1) + 1)/c_0$, $\ref{a:closing lemma frostman one scale} = c_0\ref{a:long unipotent2}$, $\ref{c:closing lemma frostman one scale} = 2\ref{c:linear algebra}$, and $\epsilon_2 = \frac{1}{4\ref{a:long unipotent1}}$. 

For initial point $x_1 \in X_{\eta}$, by reduction theory, we can write $x_1 = g_1 \Gamma$ where $|g_1| \leq \eta^{-A}$. The constant $A$ depends only on the height of $\mathbf{G}$. 
Let $R$ be large enough depending polynomially on $\eta$ which will be explicated later. Let $\delta = R^{-\frac{1}{\ref{a:closing lemma frostman one scale}}}$ and let $D > 0$, $M = \ref{m:closing lemma frostman one scale} + \ref{c:closing lemma frostman one scale}D$, and $t = M \log R$. Let $\tilde{R} = \eta^{A} R^{\frac{1}{2c_0}}$ and $S = \tilde{R}^{\frac{1}{\ref{a:long unipotent2}}}$. Here we take $R$ large enough depending polynomially on $\eta$ so that $\delta < \eta$, $S > \ref{e:long unipotent}\eta^{-\ref{a:long unipotent1}}$, and $\delta^{\epsilon_2} > S^{-\frac{1}{\ref{a:long unipotent1}}}$. 

Suppose there exists $y \in X_{3\eta}$ and $r_H \leq 10^5C_0\eta$ so that 
\begin{align*}
\mu_t(\mathsf{B}_{r_H}^H\exp(B_{\delta}^{\mathfrak{r}}). y) = \nu_t \ast \delta_{x_1} (\mathsf{B}_{r_H}^H\exp(B_{\delta}^{\mathfrak{r}}). y) > \delta^{\epsilon_2}.
\end{align*}
Let
\begin{align*}
    \mathcal{E}_y = \{u \in \mathsf{B}_1^U: a_t u.x_1 \in \mathsf{B}_{r_H}^H \exp(B_{\delta}^{\mathfrak{r}}). y\},
\end{align*}
we have 
\begin{align*}
    m_U(\mathcal{E}_y) > \delta^{\epsilon_2} > S^{\frac{1}{\ref{a:long unipotent1}}}.
\end{align*}
Fix a $u_1 \in \mathcal{E}$ and let $\mathcal{E} = \mathcal{E}_y u_1^{-1} \subseteq \mathsf{B}_{O(1)}^U$. For all $u, u' \in \mathcal{E}$, we have
\begin{align*}
    a_t uu_1.x_1 = h\exp(w).a_{t} u'u_1.x_1
\end{align*}
where $h \in \mathsf{B}_{2r_H}^H$ and $w \in B_{2\delta}^{\mathfrak{r}}$. 
Re-centering at $x_2 = a_{t}u_1.x_1$, we have
\begin{align*}
    a_{t} ua_{-t} x_2 = h\exp(w).a_{t} u'a_{-t} x_2.
\end{align*}
Note that $x_2 \in \mathsf{B}^H_{r_H}\exp(B_{\delta}^{\mathfrak{r}}).y \subseteq X_{2\eta}$. Write $x_2 = g_2 \Gamma$ where $|g_2| \leq \eta^{-A}$. There exists $\gamma_{u, u'} \in \Gamma$ so that 
\begin{align*}
    a_{t} ua_{-t} g_2 \gamma_{u, u'} = h\exp(w)a_{t} u'a_{-t} g_2
\end{align*}
with $h \in \mathsf{B}_{O(r_H)}^H$ and $w \in B_{O(\delta)}^{\mathfrak{r}}$. 
Therefore, 
\begin{align*}
    a_{t} ua_{-t} g_2 \gamma_{u, u'} g_2^{-1}a_{t} (u')^{-1}a_{-t} = h\exp(w).
\end{align*}

In summary, we have
\begin{enumerate}
    \item $m_U(\mathcal{E}) > \delta^{\epsilon_2} \geq S^{\frac{1}{\ref{a:long unipotent1}}}$.
    \item For all $u, u' \in \mathcal{E}$, there exists $\gamma \in \Gamma$ so that 
    \begin{align*}
        \|a_{t} ua_{-t} g_2 \gamma_{u, u'} g_2^{-1}a_{t} (u')^{-1}a_{-t}\| = \|h\exp(w)\| \ll 1,\\
        \mathrm{d}(a_{t} ua_{-t} g_2 \gamma_{u, u'} g_2^{-1}a_{t} (u')^{-1}a_{-t}.\hat{\mathpzc{v}}_{\mathfrak{h}}, \hat{\mathpzc{v}}_\mathfrak{h}) \ll \delta = R^{-\frac{1}{c_0\ref{a:long unipotent2}}} \leq S^{-1}.
    \end{align*}
\end{enumerate}

We apply Theorem~\ref{thm:closing lemma long unipotent} with $e^t > S > \ref{e:long unipotent}\eta^{-\ref{a:long unipotent1}}$ and the base point $x_2 = g_2\Gamma \in X_{2\eta}$. Note that by Lemma~\ref{lem:U isolate from proper ideal}, case~(2) in Theorem~\ref{thm:closing lemma long unipotent} cannot hold if $R$ is large enough so that $S^{-\frac{1}{\ref{a:long unipotent2}}} \leq c_1$ where $c_1$ is from Lemma~\ref{lem:U isolate from proper ideal}. Therefore, there exists a non-trivial proper $\Q$-subgroup $\mathbf{M}$ so that 
\begin{align*}
    \sup_{u \in \mathsf{B}_{O(1)}^U} \|a_{t} u a_{-t} g_2.\mathpzc{v}_M\| {}&\leq S^{\ref{a:long unipotent2}} = \eta^{A} R^{\frac{1}{2c_0}},\\
    \sup_{\mathpzc{z} \in B_1^{\LieU}, u \in \mathsf{B}_{O(1)}^U} \|\mathpzc{z} \wedge (a_{t} u a_{-t} g_2.\mathpzc{v}_M)\| {}&\leq e^{-\frac{t}{\ref{a:long unipotent2}}}S^{\ref{a:long unipotent2}} = e^{-\frac{t}{\ref{a:long unipotent2}}}\eta^{A} R^{\frac{1}{2c_0}}.
\end{align*}

Since $x_2 = g_2\Gamma = a_tu_1g_1\Gamma$, there exists $\gamma \in \Gamma$ so that $g_2 \gamma = a_t u_1 g_1$. Note that $u_1 \in \mathsf{B}_1^U$, we have
\begin{align*}
    \sup_{u \in \mathsf{B}_1^U} \|a_{t} u g_1 \gamma.\mathpzc{v}_M\| \leq S^{\ref{a:long unipotent2}} = \tilde{R}.
\end{align*}
Similarly, we have
\begin{align*}
    \sup_{\mathpzc{z} \in B_1^{\LieU}, u \in \mathsf{B}_1^U} \|\mathpzc{z} \wedge (a_t u g_1 \gamma.\mathpzc{v}_M)\| \leq e^{-\frac{t}{\ref{a:long unipotent2}}} S^{\ref{a:long unipotent2}} = e^{-\frac{t}{\ref{a:long unipotent2}}} \tilde{R}.
\end{align*}

Note that $\mathbf{M}$ is a $\Q$-group and $\mathpzc{v}_M$ is a primitive non-zero integer vector in $\LieG_{\Z}$, we have $\|\mathpzc{v}_M\| \geq 1$. Since $|g_1| \leq \eta^{-A}$, we have
\begin{align*}
    \|g_1\gamma.\mathpzc{v}_M\| \geq \eta^{\star A}.
\end{align*}

We now apply Lemma~\ref{lem:Closing lemma Linear Algebra} with $\tilde{t} = t$, $\tilde{R} = \eta^{A}R^{\frac{1}{c_0}}$, $\ref{a:long unipotent2}$, and $g_1\gamma.\mathpzc{v}_{M} \in \wedge^{\dim \mathbf{M}} \LieG$. Note that the condition of the lemma is satisfied if $R$ is large enough depending polynomially on $\eta$. The first conclusion of the Lemma implies that
\begin{align*}
    \|g_1 \gamma.\mathpzc{v}_M\| \ll \tilde{R}
\end{align*}
and as a consequence $\|\gamma.\mathpzc{v}_M\| \ll \eta^A\tilde{R} = R^{\frac{1}{2c_0}}$. 

We claim that case~(1) in second conclusion of Lemma~\ref{lem:Closing lemma Linear Algebra} cannot happen. We utilize the fact that $\mathbf{M}$ is in fact a $\Q$-subgroup. Suppose not, since $\{0\} \neq \rad(\LieG) \not \subseteq \Ad(g_1 \gamma)\LieM$, $\rad(\LieG) \not \subseteq \Ad(\gamma)\LieM$ and moreover, there exists a primitive integer vector $\mathpzc{r} \in \rad(\LieG) \setminus \Ad(\gamma)\LieM$ with size $\|\mathpzc{r}\| \ll \height(\rad(\LieG)) \ll \height(\mathbf{G})^\star$. The last inequality follows from \cite[Proposition 3.1]{MSGT23}. We have
\begin{align*}
    \eta^{\star A} \ll \|g_1(\mathpzc{r} \wedge \gamma.\mathpzc{v}_M)\| \ll e^{-\frac{\tilde{t}}{\ref{c:linear algebra}}} \tilde{R}^{\ref{c:linear algebra}}\|g_1.\mathpzc{r}\| \ll \eta^{-A}  e^{-\frac{\tilde{t}}{\ref{c:linear algebra}}} \tilde{R}^{\ref{c:linear algebra}},
\end{align*}
which leads to a contradiction. 

Therefore, there exists $g' \in G$ with $\|g' - \Id\|\leq \tilde{R}^{\ref{c:linear algebra}}e^{-\frac{1}{\ref{c:linear algebra}}\tilde{t}}$ so that 
\begin{align*}
    \mathbf{H} = g'g_1\gamma N_{\mathbf{G}}(\mathbf{M}) \gamma^{-1} g_1^{-1}(g')^{-1}.
\end{align*}

Since $\mathbf{M}$ is a $\Q$-subgroup of $\mathbf{G}$, its normalizer $N_{\mathbf{G}}(\mathbf{M})$ is also a $\Q$-subgroup of $\mathbf{G}$. Therefore, the orbit $Hg'g_1\Gamma$ is periodic. Moreover, 
\begin{align*}
    \|g' - \Id\| \leq \tilde{R}^{\ref{c:linear algebra}}e^{-\frac{1}{\ref{c:linear algebra}}\tilde{t}} \leq R^{-D}.
\end{align*}
By Proposition~\ref{pro:Volume and Height and normalizer}, We have
\begin{align*}
    \vol(Hg'g_1 \Gamma) = \vol(Hg'g_1\gamma \Gamma) \ll \height(\gamma\mathbf{M}\gamma^{-1})^{c_0} \leq R^{\frac{1}{2}}.
\end{align*}
We get a contradiction to the initial Diophantine condition if $R$ is large enough. This proves the lemma. 
\end{proof}

\section{Dimension estimate in many scales}\label{sec: Improve many scales}
This section is devoted to prove Theorem~\ref{thm:Closing lemma many scale}. We improve Lemma~\ref{lem:Closing lemma one scale} to obtain information for larger scales. The key ingredient is the following avoidance principle, \cref{pro:avoidance}. 

\subsection{Avoiding periodic orbits}
The proof of Theorem~\ref{thm:Closing lemma many scale} relies on the following avoidance principle obtained in \cite{SS24}. It asserts that the trajectory $a_t \mathsf{B}_1^U.x_0$ is away from periodic orbits most of the time. 
\begin{proposition}\label{pro:avoidance}
    There exist $\mathsf{m}$, $s_0$, $\consta\label{a:avoidance}$, $\constc\label{c:avoidance}$, and $\constd\label{d:avoidance}$ depending only on $(G, H, \Gamma)$, so that the following holds. Let $R_1, R_2 \geq 1$. Suppose $x_0 \in X$ is so that
    \begin{align*}
        d_X(x_0, x) \geq(\log R_1)^{\ref{d:avoidance}} R_1^{-1}
    \end{align*}
    for all $x$ with $\vol(H.x) \leq R_2$. Then for all $s \geq \ref{a:avoidance}\max\{\log R_1, |\log \inj(x_0)|\} + s_0$ and all $\eta \in (0, 1]$, we have
    \begin{align*}
        m_U\Biggl(\Biggl\{u \in \mathsf{B}_1^U: \begin{aligned}{}&\inj(a_su.x_0) \leq \eta \text{ or }\exists x \text { with }\vol(H.x) \leq R_2\\ {}&\text { and } d_X(a_s u. x_0, x) \leq \ref{c:avoidance}^{-1} R_2^{-\ref{d:avoidance}}\end{aligned}\Biggr\}\Biggr) \leq \ref{c:avoidance} (R_2^{-1} + \eta^{\frac{1}{\mathsf{m}}}).
    \end{align*}
\end{proposition}
\begin{proof}
    See \cite[Proposition 4.2, 4.4]{LMWY25} and \cite[Theorem 2]{SS24}. See also \cite[Corollary 7.2]{LMMS}. 
\end{proof}

\subsection{\texorpdfstring{F{\o}lner property for $U$}{F{\o}lner property for U}}
The following lemma allow us to view $\mu_{t_2 + t_1}$ as a $2$-step random walk. It is a standard application of F{\o}lner property.   
\begin{lemma}\label{lem:Folner1}
There exists $c_3 > 0$ so that for all $A \subseteq X$, we have
\begin{align*}
|\mu_{t_2 + t_1}(A) - (\nu_{t_2} \ast \mu_{t_1})(A)| \ll e^{-c_3 t_1}.
\end{align*}
\end{lemma}

\subsection{Proof of Theorem~\ref{thm:Closing lemma many scale}}

\begin{proof}[Proof of Theorem~\ref{thm:Closing lemma many scale}]
    We will prove the theorem for $\mathsf{B}_{r_H}^H$. The proof for $(\mathsf{B}_{r_H}^H)^{-1}$ is exactly the same. 
    
    We write the condition $R \gg \eta^{-\ref{e:closing lemma frostman one scale}}$ in Lemma~\ref{lem:Closing lemma one scale} as $R \geq R_0\eta^{-\ref{e:closing lemma frostman one scale}}$. Let $\ref{a:closing lemma frostman one scale}$, $\ref{c:closing lemma frostman one scale}$, $\ref{e:closing lemma frostman one scale}$, $\ref{m:closing lemma frostman one scale}$, and $\epsilon_2$ be as in Lemma~\ref{lem:Closing lemma one scale}. Let $\mathsf{m}$, $\ref{a:avoidance}$, $\ref{d:avoidance}$, $\ref{c:avoidance}$, and $s_0$ be as in \cref{pro:avoidance}. Let $\epsilon_1 = \min\{\frac{1}{\ref{a:closing lemma frostman one scale}\mathsf{m} \ref{e:closing lemma frostman one scale}}, c_3, \epsilon_2\}$. Let $\ref{m:closinglemma frostman} = 2\ref{a:avoidance} + \ref{m:closing lemma frostman one scale} + 1$, $\ref{c:closinglemma frostman} = 2\ref{a:avoidance} + \ref{c:closing lemma frostman one scale}$, $\ref{d:closinglemma frostman} = \ref{d:avoidance} + 1$, $\ref{a:closinglemma frostman} = \ref{a:closing lemma frostman one scale}$. Let $\ref{e:closinglemma frostman} = 10^{10}\ref{e:closing lemma frostman one scale}^2\ref{a:closing lemma frostman one scale}^2$. Suppose $R \geq \ref{c:avoidance}e^{s_0}R_0\eta^{-\ref{e:closinglemma frostman}}$. We have $\log R \geq 2|\log \eta| + s_0 \geq 2|\log \inj(x_1) | + s_0$. Suppose $R \gg_{\ref{d:avoidance}} 1$ so that $R^D \geq \ref{d:avoidance} \log(2D + 1)$ and $R^{\ref{d:avoidance}} \geq (\log R)^{\ref{d:avoidance}}$. Let $d_0 = \frac{100\ref{e:closing lemma frostman one scale}}{\ref{a:closing lemma frostman one scale}}$, note that we have $R^{-\frac{1}{A_7}} \leq \eta^{100d_0^2}$ and $\eta^{d_0} \leq \eta^{\frac{\ref{e:closing lemma frostman one scale}}{\ref{a:closing lemma frostman one scale}}}$. 
    
    Let $\delta_0 := R^{-\frac{1}{\ref{a:closing lemma frostman one scale}}}$. For all $\delta_0 \leq r \leq \eta^{d_0}$, let $R_2 = r^{-\ref{a:closing lemma frostman one scale}}$ and $R_1 = R^{2D + 1}$. Then for all $x$ with $H.x$ periodic and $\vol(H.x) \leq R_2 \leq R$, we have
    \begin{align*}
        d(x_1, x) \geq R^{-D} \geq (\log R_1)^{\ref{d:avoidance}} R_1^{-1}.
    \end{align*}

    For all $D \geq \ref{d:closinglemma frostman} + 1$, let $M = \ref{m:closinglemma frostman} + \ref{c:closinglemma frostman}D$ and $t \geq M \log R$. Let $t_2 = (\ref{m:closing lemma frostman one scale} + \ref{c:closing lemma frostman one scale}D)\log R_2$ and $t_1 = t - t_2$. We have
    \begin{align*}
        t_1 = t - t_2 \geq\ref{a:avoidance}\max\{\log R_1, \log |\inj(x_1)|\} + s_0.
    \end{align*}
    
    Let $\tilde{\eta} = R_0^{-\frac{1}{\ref{e:closing lemma frostman one scale}}}R_2^{-\frac{1}{\ref{e:closing lemma frostman one scale}}} = R_0^{-\frac{1}{\ref{e:closing lemma frostman one scale}}} r^{\frac{\ref{a:closing lemma frostman one scale}}{\ref{e:closing lemma frostman one scale}}} \leq \eta$ and let
    \begin{align*}
        X_1 = \{x \in X_{\tilde{\eta}}: \forall x' \text{ with } H.x \text{ periodic and }\vol(H.x) \leq R_2, d(x, x') > R_2^{-D}\}
    \end{align*}
    and $X_2 = X \setminus X_1$. Recall $\epsilon_1 = \min\{\frac{1}{\ref{a:closing lemma frostman one scale}\mathsf{m} \ref{e:closing lemma frostman one scale}}, c_3, \epsilon_2\}$. Apply \cref{pro:avoidance} to $x_1$, $R_1$, $R_2$, $\tilde{\eta}$ and $t_1$, we have
    \begin{align}\label{eqn:manyscale bad point}
        \mu_{t_1}(X_2) \leq \ref{c:avoidance}(R_2^{-1} + \tilde{\eta}^{\frac{1}{\mathsf{m}}}) \ll r^{\epsilon_1}.
    \end{align}
    Apply Lemma~\ref{lem:Closing lemma one scale} to $x \in X_1$, $R_2$ and $t_2$ and note that $R_2^{-\frac{1}{\ref{a:closing lemma frostman one scale}}} = r$, for all $y \in X_{3\eta}$ and $r_H \leq 10^5C_0\eta$, we have
    \begin{align}\label{eqn:many scale good point}
        (\nu_{t_2} \ast \delta_{x})(\mathsf{B}_{r_H}^H \exp(B_{r}^{\mathfrak{r}}).y) \ll r^{\epsilon_2}.
    \end{align}
    
    Combine \cref{eqn:many scale good point,eqn:manyscale bad point}, for all $y \in X_{3\eta}$, $r_H \leq \eta$ we have
    \begin{align*}
        (\nu_{t_2} \ast \mu_{t_1})(\mathsf{B}_{r_H}^H\exp(B_{r}^{\mathfrak{r}}).y) 
        \ll r^{\epsilon_2} + \mu_{t_1}(X_2) \leq r^{\epsilon_2} + r^{\epsilon_1} \ll r^{\epsilon_1}.
    \end{align*}
    Since $t_1 \geq \log R$, we have
    \begin{align*}
        \mu_{t}(\mathsf{B}_{r_H}^H\exp(B_{r}^{\mathfrak{r}}).y) = \mu_{t_2 + t_1}(\mathsf{B}_{r_H}^H\exp(B_{r}^{\mathfrak{r}}).y) \ll r^{\epsilon_1} + e^{-c_3 t_1} \ll r^{\epsilon_1}.
    \end{align*}
    This proves the theorem. 
\end{proof}

\part{Dimension improvement in the transverse complement}\label{part:projection}

The main result of this part is Theorem~\ref{thm:energy Improvement}. It is a linear dimension improvement result in the representation $\mathfrak{r}$ of $H$. Let us first fix some notations. 

As in the introduction, $\mathbf{H}$ is a connected semisimple $\R$-group and $H = \mathbf{H}(\R)$ is the group of its $\R$-points. Let $\LieH$ be the Lie algebra of $H$ and let $\LieH = \oplus_{i = 1}^k \LieH_i$ be the unique (up to permutation) decomposition of $\LieH$ into simple ideals. In this part, we use $V$ to denote a regular representation of $\mathbf{H}$. Let $\mathbf{a} \in \LieH$ be a semisimple element so that its projection to each $\LieH_i$ is non-trivial. Suppose $\ad\mathbf{a}$ is diagonalizable over $\R$. Let $a_t = \exp(t\mathbf{a})$ and let 
\begin{align*}
    U = \{u \in H: a_{-t} u a_{t} \to e \text{ as } t \to +\infty\}.
\end{align*}
Recall that
\begin{align*}
    \mathsf{B}^U_1 = \exp(\mathsf{B}_1^{\LieU}(0))
\end{align*}
and $m_U$ is the Haar measure on $U$ so that $m_U(\mathsf{B}^U_1) = 1$. 

\medskip

Let $V$ be a finite dimensional non-trivial representation of $H$. We set $n = \dim V$. 

\newcommand{\GIrrep}{\nameref{hyp:GIrrep}\xspace}
\begin{namedhypothesis}[GIrrep]\label{hyp:GIrrep}
The representation $V$ is an irreducible representation of $H$. 
\end{namedhypothesis}

\newcommand{\GSchu}{\nameref{hyp:GProx}\xspace}
\begin{namedhypothesis}[GSchu]\label{hyp:GProx}
Let $\Fix(U)$ be the fixed point of $U$ on $V$. The orbit $H.\Fix(U)$ does not lie in any proper Schubert variety of $\Gr_{\dim \Fix (U)}(V)$. 
\end{namedhypothesis}

Note that if the data $(G, \Gamma, H, \mathfrak{r}, \mathbf{a}, U)$ satisfy {}\Irrep{} and {}\Schu{}, the complement $\mathfrak{r}$ satisfies the two conditions. 

Note that since $V$ is a regular representation of $\mathbf{H}$ and $\mathbf{a}$ is a semisimple element in $\LieH$, $V$ can be decomposed into eigenspaces of $\mathbf{a}$. Let 
\begin{align*}
    V = \bigoplus_\lambda V_{\lambda}
\end{align*}
 be such decomposition. Note that there is a \emph{total order} on $\{\lambda\}$ (since we only focus on the element $\mathbf{a}$) and we denote
\begin{align*}
    V^{(\mu)} = \bigoplus_{\lambda \geq \mu} V_{\lambda}.
\end{align*}
Let $\mu_1 < \cdots < \mu_m = \mu_{\max}$ be all the eigenvalues of $\mathbf{a}$ and let $k_1, \ldots, k_m$ be the dimension of the corresponding eigenspaces. Since $V$ is a nontrivial representation, $\mu_1 < 0 < \mu_m$. 

By Mostow's simultaneous Cartan decomposition theorem \cite{Mos55}, there exist an inner product on $V$ and an associated orthonormal basis so that matrix representations of elements in $A_H$ are diagonal. Moreover, the matrix transpose is the adjoint operation under this inner product. We identify $V \cong \R^n$ under this basis. 

We remark that the norm induced by this inner product is \emph{not} the norm from the max-norm on $\LieG \subseteq \mathfrak{sl}_N$. On the other hand, such two norm differs by a constant that can be computed explicitly. \emph{In this part, the balls in $V$ will be defined using the norm from Mostow's theorem.}

Let $|\cdot|_\delta$ denote the $\delta$-covering number according to the inner product fixed by Mostow's theorem. For a finite set $F$, let $\mu_F$ be the uniform probability measure on $F$. For all $\alpha \in (0, \dim V)$ and scale $\delta \in (0, 1)$, we defined the following (modified) $\alpha$-energy of the set $F$ in Subsection~\ref{subsec:Margulis function}:
\begin{align*}
    \mathcal{G}^{(\alpha)}_{F, \delta}(w) = \sum_{w' \in F, w' \neq w} \max\{\|w' - w\|, \delta\}^{-\alpha}.
\end{align*}

The following is the main result of this part. It is an analog of \cite[Theorem 6.1]{LMWY25}. 
\begin{theorem}\label{thm:energy Improvement}
Suppose the representation $V$ of $H$ satisfies {}\GIrrep{} and \GSchu{}. Let $\alpha \in (0, n)$. There exists $\epsilon_1 = \epsilon_1(\alpha) > 0$ and $\delta_1 = \delta_1(\alpha) > 0$ so that the following holds for all $\epsilon \in (0, \epsilon_1)$ and $\delta \in (0, \delta_1)$. Suppose there exists a finite set $F \subset B_1^{V}(0)$ with $\#F \gg_\epsilon 1$ satisfying
\begin{align*}
    \mathcal{G}_{F, \delta}^{(\alpha)}(w) \leq \Upsilon \quad \forall w \in F.
\end{align*}

Then for all $\ell \gg_\epsilon 1$, there exists $J \subset \mathsf{B}_1^U$ with $m_U(\mathsf{B}_1^U \setminus J) \ll_\epsilon |\log \delta|e^{-\epsilon(\mu_{m} - \mu_{m - 1}) \ell}$ so that the following holds. For all $u \in J$ there exists $F_{u} \subseteq F$ with $\#(F \setminus F_u) \ll_\epsilon |\log \delta|e^{-\epsilon(\mu_{m} - \mu_{m - 1}) \ell}\#F$ so that for all $w \in F_u$
\begin{align*}
    \mathcal{G}^{(\alpha)}_{F_{u}(w), \delta'}(a_\ell u.w) \ll_\epsilon e^{-\epsilon(\mu_{m} - \mu_{m - 1})\ell} \Upsilon
\end{align*}
where the new scale $\delta' = e^{\mu_{m}\ell}\max\{\delta, \#F^{-\frac{1}{\alpha}}\}$ and the set
\begin{align*}
    F_{u}(w) = \{a_{\ell} u.w': w' \in F_{u}, \|a_{\ell} u.w' - a_{\ell} u.w\|\leq e^{\mu_{1}\ell}\}.
\end{align*}
Moreover, the function $\epsilon_1$ and $\delta_1$ has a uniform lower bound when $\alpha$ varies in a compact subset of $(0, n)$. 
\end{theorem}

Theorem~\ref{thm:energy Improvement} follows from the following theorem providing Frostman-type estimate. 
\begin{theorem}\label{thm:ImprovementMain}
Suppose the representation $V$ of $H$ satisfies {}\GIrrep{} and \GSchu{}. Let $\alpha \in (0, n)$. There exists $\epsilon_1 = \epsilon_1(\alpha) > 0$ and $\delta_1 = \delta_1(\alpha) > 0$ so that the following holds for all $\epsilon \in (0, \epsilon_1)$ and $\delta_0 \in (0, \delta_1)$. 

Let $F \subset B^{V}_1(0)$ be a finite set satisfying
\begin{align*}
    \mu_F(B^{V}_\delta(x)) \leq C\delta^\alpha \quad \forall x \in V
\end{align*}
for some $C \geq 1$ and all $\delta \geq \delta_0$. 

For all $\ell \gg_{\epsilon} 1$ and $\delta \in [e^{\mu_{m}\ell}\delta_0, e^{\mu_{1}\ell}]$, there exists $J_{\ell, \delta} \subseteq \mathsf{B}_1^U$ with $m_U(\mathsf{B}_1^U \setminus J_{\ell, \delta}) \ll_\epsilon e^{-\epsilon(\mu_{m} - \mu_{m - 1})\ell}$ so that the following holds. Let $u \in J_{\ell, \delta}$, there exists $F_{\ell, \delta, u} \subseteq F$ with 
\begin{align*}
    \mu_F(F \setminus F_{\ell, \delta, u}) \ll_\epsilon e^{-\epsilon(\mu_{m} - \mu_{m - 1})\ell}
\end{align*}
such that for all $w \in F_{\ell, \delta, u}$ we have
\begin{align*}
    \mu_F(\{w' \in F_{\ell, \delta, u}: \|a_\ell u.w' - a_\ell u.w\| \leq \delta\}) \ll_\epsilon Ce^{-\epsilon(\mu_{m} - \mu_{m - 1})\ell}\delta^{\alpha}.
\end{align*}
\end{theorem}
\cref{thm:ImprovementMain} in turns follows from the following theorem on covering numbers. 

\begin{theorem}\label{thm:ImprovementSlabExpansion}
Suppose the representation $V$ of $H$ satisfies {}\GIrrep{} and \GSchu{}. Let $\alpha \in (0, n)$. There exists $\epsilon_1 = \epsilon_1(\alpha) > 0$ and $\delta_1 = \delta_1(\alpha) > 0$ so that the following holds for all $\epsilon \in (0, \epsilon_1)$ and $\delta_0 \in (0, \delta_1)$. 

Let $F \subset B^{V}_1(0)$ be a finite set satisfying
\begin{align*}
    \mu_F(B^{V}_\delta(x)) \leq C\delta^\alpha \quad \forall x \in V
\end{align*}
for some $C \geq 1$ and all $\delta \geq \delta_0$. 

There exists $C_{\epsilon} > 1$ so that the following holds. 
For all $\ell \gg_{\epsilon} 1$ and $\delta \in [e^{\mu_{m}\ell}\delta_0, e^{\mu_{1}\ell}]$, we define the exceptional set $\mathcal{E}(F)$ to be
\begin{align*}
    \mathcal{E}(F) = \{u \in \mathsf{B}_1^U: \exists F' \subseteq F \text{ with }& \mu_F(F') \geq e^{-\epsilon(\mu_m - \mu_{m - 1})\ell}\\
    \text{ and }&|a_\ell u.F'|_{\delta} < C_{\epsilon}^{-1}C^{-1} e^{\epsilon(\mu_m - \mu_{m - 1})\ell} \delta^{-\alpha}\}.
\end{align*}

We have
\begin{align*}
    m_U(\mathcal{E}(F)) \leq C_{\epsilon}e^{-\epsilon(\mu_m - \mu_{m - 1})\ell}.
\end{align*}
\end{theorem}

A key step in the proof of the above theorem is an estimate of covering number using certain anisotropic tubes explicated later, see Theorem~\ref{thm:Multislicing}. Before we state it, let us introduce some notations. 

For a subspace $W \subseteq V$, we always use $\pi_{W}$ to denote the orthogonal projection to $W$. For $V^{(\mu)}$, we simplify the notation by setting $\pi^{(\mu)} = \pi_{V^{(\mu)}}$. We use $\pi_{u}^{(\mu)}$ to denote the projections $\pi^{(\mu)} \circ u$. 

We adapt the notations in \cite{BH24} for partitions using anisotropic tubes associated to the flag $\{V^{(\mu)}\}_{\mu}$. Let $\mathcal{D}_\delta(\mu)$ be the partition of $V^{(\mu)}$ via $\delta$-cubes. For a $m$-tuple $\mathbf{r} = (\mathsf{r_1}, \ldots, \mathsf{r_m})$ satisfying $0 \leq \mathsf{r_1} \leq \cdots \leq \mathsf{r_m} = 1$, we define
\begin{align*}
    \mathcal{D}_{\delta}^{\mathbf{r}} = \bigvee_{i = 1}^m (\pi^{(\mu_i)})^{-1} \mathcal{D}_{\delta^{\mathsf{r_i}}}(\mu_i)
\end{align*}
to be the partition consisting of (possibly anisotropic) tubes. We will use $T$ to denote an atom in $\mathcal{D}_{\delta}^{\mathbf{r}}$. Roughly, $T$ is a tube of size
\begin{align*}
    \prod_{j = 1}^{k_1}\delta^{\mathsf{r_1}} \times \cdots \times 
    \prod_{j = 1}^{k_m}\delta^{\mathsf{r_m}} 
\end{align*}
with edges parallel to an orthogonal basis compatible with the weight space decomposition $V = \bigoplus_{i} V_{\mu_i}$. Its volume satisfies
\begin{align*}
    \vol(T) \sim \delta^{\sum_{i = 1}^m\mathsf{k_i} \mathsf{r_i}}.
\end{align*}

\begin{theorem}\label{thm:Multislicing}
Suppose the representation $V$ of $H$ satisfies {}\GIrrep{} and \GSchu{}. Let $\alpha \in (0, n)$. There exists $\epsilon_1 = \epsilon_1(\alpha) > 0$ and $\delta_1 = \delta_1(\alpha) > 0$ so that the following holds for all $\epsilon \in (0, \epsilon_1)$ and $\delta_0 \in (0, \delta_1)$. 

Let $F \subset B^{V}_1(0)$ be a finite set satisfying
\begin{align*}
    \mu_F(B^{V}_\delta(x)) \leq C\delta^\alpha \quad \forall x \in V
\end{align*}
for some $C \geq 1$ and all $\delta \geq \delta_0$. 

Fix a $m$-tuple $\mathbf{r}$. Then there exists $C_{\epsilon, \mathbf{r}}$ so that the following holds. For all $\delta_0 \leq \delta < \delta_1$, we define the exceptional set $\mathcal{E}(F)$ to be
    \begin{align*}
        \mathcal{E}(F) = \{u \in \mathsf{B}_1^U: &\exists F' \subseteq F \text{ with }\mu_F(F') \geq \delta^\epsilon \text{ and }\\
        &|u.F'|_{\mathcal{D}_{\delta}^{\mathbf{r}}} < C_{\epsilon, \mathbf{r}}^{-1}C^{-1}\vol(T)^{-\frac{1}{n}\alpha - (\mathsf{r_m} - \mathsf{r_{m - 1}}) \epsilon}\}.
    \end{align*}

    We have
    \begin{align*}
        m_U(\mathcal{E}(F)) \leq C_{\epsilon, \mathbf{r}}\delta^{\epsilon}.
    \end{align*}
\end{theorem}

The proof of Theorem~\ref{thm:ImprovementSlabExpansion} assuming Theorem~\ref{thm:Multislicing} is similar to \cite[Section 8]{Lin25}. We now discuss the proof of Theorem~\ref{thm:Multislicing}. As indicated in the introduction, the proof is based on Bourgain's discretized projection theorem to the highest weight space and subcritical bound for other $V^{(\mu)}$'s. This argument utilizes a sub-modularity inequality, similar arguments have appeared in special cases in an early draft by the author \cite[Section 14]{Lin25} and in \cite{BHZ25}. The arguments there works in general. 

The following theorem provides subcritical estimate for projections in irreducible representation of $H$. 

\begin{theorem}\label{thm:subcritical entropy general irrep}
    Suppose the representation $V$ of $H$ satisfies only {}\GIrrep{}. There exists $M > 1$ depending only on $U$ and $V$ so that for all $0 < \epsilon < \frac{1}{100}$ and $0 < \delta \ll_{\epsilon, V} 1$ the following holds for all eigenvalue $\mu$ of $\mathbf{a}$. 

    Let $A \subseteq B_1^V$, we set the exceptional set of projection to be 
    \begin{align*}
        \mathcal{E}(A) = \{u \in B_1^U: |\pi^{(\mu)}_u(A)|_{\delta} < \delta^{M\epsilon} |A|_{\delta}^{\frac{\dim V^{(\mu)}}{\dim V}}\}.
    \end{align*}
    Then we have
    \begin{align*}
        m_U(\mathcal{E}(A)) \leq \delta^\epsilon.
    \end{align*}
\end{theorem}

As a bi-product, we obtain the following result (Theorem~\ref{thm:dimension bound one subspace}) regarding intersection of orbit of a subspace in irreducible representations with certain Schubert variety. Let $F$ be a field and let $\tilde{H}$ be a connected linear algebraic group over $F$. Let $V$ be a regular irreducible representation of $\tilde{H}$ defined over $F$. 

\begin{theorem*}
    For all subspaces $W, W' \subseteq V$, there exists $\tilde{h} \in \tilde{H}$ so that 
    \begin{align*}
        \dim [(\tilde{h}.W) \cap W'] \leq \frac{\dim W}{\dim V} \dim W'.
    \end{align*}
    Note that the set of such $\tilde{h}$ forms a Zariski open subset of $\tilde{H}$.
\end{theorem*}

Part~\ref{part:projection} is organized as the following. In Section~\ref{sec:prepare rep and proj}, we collect some basic properties of projections in irreducible representation of semisimple Lie groups. As indicated in the introduction, dimension improvement for the action of $\{\Ad(a_\ell u)\}_{u \in \mathsf{B_1^U}}$ on $\mathfrak{r}$ comes from gain in $\mathfrak{r}^{(\mu_{\max})}$ and 'no-loss' in other in $\mathfrak{r}^{(\mu)}$. In section~\ref{sec:improvement}, we prove an $\epsilon$-improvement in $\mathfrak{r}^{(\mu_{\max})}$ from a version of Bourgain's discretized projection theorem. Then we deduce \cref{thm:energy Improvement,thm:ImprovementMain,thm:ImprovementSlabExpansion,thm:Multislicing} from Theorem~\ref{thm:subcritical entropy general irrep} and and Theorem~\ref{thm:supcritical} in Section~\ref{sec:anisotropic}. The rest of this part is devoted to the proof of the 'no-loss' part, i.e.,  Theorem~\ref{thm:subcritical entropy general irrep} and we refer to Section~\ref{sec:outline subcritical} for a more detailed outline. Section~\ref{sec:outline subcritical} to \ref{sec:translate to entropy} can be read independently to the rest of the paper. 

\section{Preparation II: representations of semisimple Lie groups}\label{sec:prepare rep and proj}

\subsection{\texorpdfstring{An inner product on $V$}{An inner product on V}}

For any representation $V$ of $H$ in this paper, we fix an inner product $\langle\cdot, \cdot \rangle$ from Mostow's theorem \cite{Mos55}. It satisfies the following properties. 
\begin{enumerate}
    \item The inner product $\langle\cdot, \cdot \rangle$ is $K_H$-invariant. 
    \item For all $a \in A_H$, the action of $a$ on $V$ is self-adjoint under this $\langle\cdot, \cdot \rangle$. As a consequence, different eigenspaces of $\mathbf{a}$ are orthogonal. 
\end{enumerate}
We fix an orthonormal basis of $V$ with repect to this inner product. This gives an isomorphism $V \cong \R^n$. For a subspace $W \subseteq V$, we set $\pi_W$ to be the orthogonal projection to $W$ with respect to this inner product. For a linear operator $A$ on $V$, we use $(A)^t$ to denote the adjoint operator of $A$ under this inner product. 

Recall that in the introduction in Part~\ref{part:projection}, we define the projections $\pi_{u}^{(\lambda)} = \pi^{(\lambda)} \circ u$. 
There are also the following closely related orthogonal projections $\pi_{u^t.V^{(\lambda)}}$. The latter is the orthogonal projection to the subspace $u^t.V^{(\lambda)}$. More generally, we have two similar projections, $\pi_W \circ h$ and $\pi_{h^t.W}$

The following linear algebra lemma allows us to roughly identify $\pi_{W} \circ h$ with $\pi_{h^t.W}$. 
\begin{lemma}\label{lem:equivalence of two type of proj}
    For all $h \in H$ and subspace $W \subseteq V$, there exists a linear isomorphism $\varphi_h: W \to h^t.W$ so that
    \begin{align*}
        \varphi_h \circ \pi_W \circ h = \pi_{h^t.W}
    \end{align*}
    and
    \begin{align*}
        \max\{\Lip(\varphi_h), \Lip(\varphi_h^{-1})\} \ll \|h\|^\star.
    \end{align*}
\end{lemma}
Note that when we pick $u \in B_1^U$, the Lipschitz constants depend only on the ambient representation. Therefore, the estimate on covering numbers for projections $\pi^{(\lambda)}_u$ and $\pi_{u^t.V^{(\lambda)}}$ are equivalent up to an absolute constant. 

\subsection{Non-degenerate measures on Grassmannians}\label{subsec:nondegeneracydef}
For two subspaces $U$ and $W$ of $V$, we define
\begin{align*}
    d_{\measuredangle}(U, W) = \|u_1 \wedge \cdots \wedge u_k \wedge w_1 \wedge \cdots \wedge w_l\|
\end{align*}
where $\{u_i\}_{i = 1}^k$ and $\{w_j\}_{j = 1}^l$ are orthonormal basis of $U$ and $W$ respectively. This is independent to the choice of $\{u_i\}_{i = 1}^k$ and $\{w_j\}_{j = 1}^l$. Similarly, for subspaces $V_1, \ldots, V_q$ of $V$, we define
\begin{align*}
    d_{\measuredangle}(V_1, \ldots, V_q) = \|\mathbf{v}_1 \wedge \cdots \wedge \mathbf{v}_q\|
\end{align*}
where $\mathbf{v}_i$'s are wedge of an orthonormal basis of $V_i$. This is independent to the choice of $\{\mathbf{v}_i\}_{i = 1}^q$. 

Let $W \in \Gr_{n - k}(\R^n)$, we define
\begin{align*}
    \mathcal{V}(W, \rho) = \{U \in \Gr_k(\R^n): d_{\measuredangle}(U, W) \leq \rho\}.
\end{align*}
If $\rho = 0$, $\mathcal{V}(W, 0)$ is the collection of $k$-dimensional subspaces intersecting $W$ non-trivially. It belongs
to the class of algebraic subvarieties of the grassmannian known as Schubert varieties. 

\begin{definition}[$(C, \kappa)$-non-degeneracy]
    
For a probability measure $\sigma$ on $\Gr_m(\R^n)$, we say it satisfies $(C, \kappa)$-non-degeneracy condition at scales larger than $\delta$ if the following holds. 

There exist constants $C \geq 1$, $\kappa > 0$ such that for all $\rho \geq \delta$ and all $W \in \Gr_{m}(\R^n)$, one has
\begin{align}\label{eqn:Non-degenerate}
    \sigma(\mathcal{V}(W, \rho)) \leq C\rho^{\kappa}.
\end{align}
\end{definition}
\begin{remark}
Most literature use the terminology non-concentration condition. We use the terminology non-degeneracy here to distinguish it from the non-concentration condition on the set or the measure \emph{in} the representation space $V$. Also, due to the polynomial nature of unipotent flow, this condition corresponds to non-degeneracy for some polynomials. 
\end{remark}

In practice, we always allow $C = O(\delta^{-O(\epsilon)})$. 
We will say a family of subspaces satisfies the non-degeneracy condition if the measure and scale are clear in the context. 

\section{Dimension improvement from highest weight space}\label{sec:improvement}
This section is devoted to the dimension gain from $\mu_{\max}$. The main ingredient is Bourgain's discretized projection theorem \cite{Bou10,He20,Shm23}. \emph{This is the only place in this paper where we need {}\GSchu{}. }

\begin{theorem}\label{thm:supcritical}
Suppose $V$ is a representation of $H$ satisfying {}\GIrrep{} and \GSchu{}. Let $\mu_{\max}$ be the largest eigenvalue of $\mathbf{a}$ on $V$. 

For all $0 < \alpha < n$ and $\kappa > 0$, there exists $\epsilon = \epsilon(\alpha, \kappa) > 0$ so that the following holds for all $0 < \delta \ll 1$. 

Suppose $A \subseteq B_1^V(0)$ be a set satisfying the following conditions:
\begin{gather*}
    |A|_{\delta} \geq \delta^{-\alpha + \epsilon}, \\
    |A \cap B_r^V(x)|_{\delta} \leq \delta^{-\epsilon} r^{\kappa} |A|_{\delta}, \quad\forall r \geq \delta, x \in V.
\end{gather*}
Then there exists $\mathcal{E} \subseteq \mathsf{B}_1^U$ with $m_{U}(\mathcal{E}) < \delta^{\epsilon}$ so that for all $u \notin \mathcal{E}$ and all $A' \subseteq A$ with $|A'|_{\delta} \geq \delta^{\epsilon}|A|_{\delta}$, we have
\begin{align*}
    |\pi^{(\mu_{\max})}_{u}(A')|_{\delta} \geq \delta^{-\frac{k_m\alpha}{n} - \epsilon}.
\end{align*}

\end{theorem}

The main ingredient of the proof is \cite[Theorem 1]{He20} recorded in the follwing theorem. 

\begin{theorem}[{\cite[Theorem 1]{He20}}]\label{thm:He supcritical}
For all $0 < \alpha < n$ and $\kappa > 0$, there exists $\epsilon = \epsilon(\alpha, \kappa) > 0$ so that the following holds for all $0 < \delta \ll 1$. 

Suppose $A \subseteq B_1^V(0)$ be a set satisfying the following conditions:
\begin{gather*}
    |A|_{\delta} \geq \delta^{-\alpha + \epsilon}, \\
    |A \cap B_r^V(x)|_{\delta} \leq \delta^{-\epsilon} r^{\kappa} |A|_{\delta}, \quad\forall r \geq \delta, x \in V.
\end{gather*}

Suppose $\sigma$ is a probability measure on $\Gr(k, \R^n)$ that is $(\delta^{-\epsilon}, \kappa)$-non-degenerate. 

Then there exists $\mathcal{E} \subseteq \Gr_k(\R^n)$ with $\sigma(\mathcal{E}) < \delta^{\epsilon}$ so that for all $V \notin \mathcal{E}$ and all $A' \subseteq A$ with $|A'|_{\delta} \geq \delta^{\epsilon}|A|_{\delta}$, we have
\begin{align*}
    |\pi_{V}(A')|_{\delta} \geq \delta^{-\frac{k\alpha}{n} - \epsilon}.
\end{align*}
\end{theorem}

We now prove of Theorem~\ref{thm:supcritical}. Fix an orthonormal basis $v_1, \ldots, v_{k_m}$ of $\Fix(U)$. Write $\mathbf{v} = v_1 \wedge \cdots \wedge v_{k_m}$. For any subspace $W$ with $\dim W = n - \dim \Fix(U)$, write $\mathbf{w} = w_1 \wedge \cdots \wedge w_{n - k_m}$ where $w_1, \ldots, w_{n - k_m}$ is an orthonormal basis of $W$. Note that
\begin{align*}
    d_{\measuredangle}(u^t.\Fix(U), W) = \frac{\|(u^t.\mathbf{v}) \wedge \mathbf{w}\|}{\|u^t.\mathbf{v}\|} \asymp \|(u^t.\mathbf{v}) \wedge \mathbf{w}\|
\end{align*}
when $u \in \mathsf{B}_1^U$. 

Due to the polynomial nature of actions of unipotent groups, we need an estimate on the size of the set where
the (vector-valued) polynomial function $\left(u \mapsto (u^t.\mathbf{v}) \wedge \mathbf{w}\right)$ is small. This is known as Remez's inequality and is used by Kleinbock and Margulis and later Kleinbock and Tomanov in \cite{KM98,KT07} to verify the '$(C, \alpha)$-good' property. We record the form we need in the following lemma. 

\begin{lemma}[\text{\cite[Lemma 3.4]{KT07}}]\label{lem:RemezKT}
    For all $d, k \in \N$, there exists a constant $C = C_{d, k} > 0$ so that the following holds. Let $P \in \R[x_1, \ldots, x_d]$ be a polynomial with degree at most $k$. For all ball $B \subset \R^d$ and $\epsilon > 0$, we have
    \begin{align*}
        \Leb\{x \in B:|P(x)| < \epsilon\} \leq C\biggl(\frac{\epsilon}{\|P\|_{L^\infty(B)}}\biggr)^{\frac{1}{dk}}\Leb(B).
    \end{align*}
\end{lemma}

By Remez's inequality, it suffices to estimate the supremum of the polynomial $\left(u \mapsto (u^t.\mathbf{v}) \wedge \mathbf{w}\right)$. By condition \GSchu{} and the fact that $\Fix(U)$ is $\LieH_0 \oplus \LieU$-invariant, the polynomials $\left(u \mapsto (u^t.\mathbf{v}) \wedge \mathbf{w}\right)$ are non-zero for all $W \in \Gr_{k_m}(\R^n)$. Also, their degree is bounded uniformly depending only on $V$ and $U$. By the compactness of $\Gr_{n - k_m}(\R^n)$, we have
\begin{align*}
    \sup_{u \in \mathsf{B}_1^U} \|(u^t.\mathbf{v}) \wedge \mathbf{w}\| \gg 1
\end{align*}
where the implied constant is independent to $W$. We remark that such constant can be computed explicitly in an explicit representation $V$. 

\begin{remark}
    We remark that for the example where $G = \SL_d(\R) \ltimes \R^d$, $H = \SL_d(\R)$ and \begin{align*}
    a_t = \begin{pmatrix}
        e^{nt} \Id_n & \\
         & e^{-mt}\Id_m
    \end{pmatrix},
    \end{align*}
    the orbit $U^t.\Fix(U)$ parametrizes an open set of $\Gr_{n}(\R^d)$. One can apply Marstrand's projection theorem instead.     
\end{remark}
\section{\texorpdfstring{Proof of \cref{thm:energy Improvement,thm:ImprovementMain,thm:ImprovementSlabExpansion} assuming Theorem~\ref{thm:subcritical entropy general irrep} and \ref{thm:supcritical}}{Proof of Theorem~\nameref{thm:energy Improvement,thm:ImprovementMain,thm:ImprovementSlabExpansion} assuming Theorem~\ref{thm:subcritical entropy general irrep}and \ref{thm:supcritical}}}\label{sec:anisotropic}

\subsection{\texorpdfstring{Proof of \cref{thm:energy Improvement,thm:ImprovementMain,thm:ImprovementSlabExpansion} assuming Theorem~\ref{thm:Multislicing}}{Proof of Theorem~\nameref{thm:energy Improvement,thm:ImprovementMain,thm:ImprovementSlabExpansion} assuming Theorem~\ref{thm:Multislicing}}}

\begin{proof}[Proof of Theorem~\ref{thm:energy Improvement} assuming Theorem~\ref{thm:ImprovementMain}]
This is a standard transition from Frostman-type estimate to energy estimate. The statement can be proved e.g. by following the proof of \cite[Theorem 6.1]{LMWY25} step-by-step and replacing \cite[Lemma 6.2]{LMWY25} by Theorem~\ref{thm:ImprovementMain}. 
\end{proof}

We now deduce \cref{thm:ImprovementMain} from \cref{thm:ImprovementSlabExpansion}. This procedure is well-known. We reproduce it here for completeness. 
\begin{proof}[Proof of \cref{thm:ImprovementMain} assuming \cref{thm:ImprovementSlabExpansion}]
Applying Theorem~\ref{thm:ImprovementSlabExpansion}, there exists $\mathcal{E} \subset \mathsf{B}^U_1$ with $m_U(\mathcal{E}) \ll_\epsilon e^{-\epsilon (\mu_{m} - \mu_{m - 1})\ell}$ such that for all $u \notin \mathcal{E}$ and all $F'$ with $\mu_F(F') \geq e^{-\epsilon(\mu_{m} - \mu_{m - 1})\ell}$, we have
\begin{align*}
    |a_{\ell}u.F'|_{\delta} \geq C_{\epsilon}^{-1} C^{-1} e^{\epsilon(\mu_{m} - \mu_{m - 1})\ell} \delta^{-\alpha}.
\end{align*}

We define
\begin{align*}
    \mathcal{D}_{\delta, \text{bad}}^{\ell, u} = \{Q \in \mathcal{D}_{\delta}: (a_{\ell}u)_{*} \mu_F (Q) > C_{\epsilon} C e^{-\epsilon(\mu_{m} - \mu_{m - 1})\ell} \delta^{\alpha}\}.
\end{align*}
Let
\begin{align*}
    F_{\ell, \delta, u}' = (a_{\ell} u)^{-1} \bigcup_{Q \in \mathcal{D}_{\delta, \text{bad}}^{\ell, u}} ((a_{\ell} u.F) \cap Q).
\end{align*}
Since $(a_{\ell} u)_{*} \mu_F$ is a probability measure, we have
\begin{align*}
    \# \mathcal{D}_{\delta, \text{bad}}^{\ell, u} < C_{\epsilon}^{-1} C^{-1} e^{\epsilon(\mu_{m} - \mu_{m - 1})\ell} \delta^{-\alpha}, 
\end{align*}
which is equivalent to
\begin{align*}
    |a_{\ell} u F_{\ell, \delta, u}'|_{\delta} < C_{\epsilon}^{-1} C^{-1} e^{\epsilon(\mu_{m} - \mu_{m - 1})\ell} \delta^{-\alpha}.
\end{align*}
Therefore, we have
\begin{align*}
    \mu_F(F_{\ell, \delta, u}') \leq e^{-\epsilon(\mu_{m} - \mu_{m - 1})\ell}.
\end{align*}

Let $F_{\ell, \delta, u} = F \backslash F_{\ell, \delta, u}'$. For all $\delta$-(dyadic) cube $Q$, we have
\begin{align*}
    (a_{\ell} u)_{*}(\mu_F|_{F_{\ell, \delta, u}}) (Q) \leq{}& C_{\epsilon} C e^{-\epsilon(\mu_{m} - \mu_{m - 1})\ell} \delta^{\alpha}
\end{align*}
which proves the theorem. 
\end{proof}

\begin{proof}[Proof of Theorem~\ref{thm:ImprovementSlabExpansion} assuming Theorem~\ref{thm:Multislicing}]
Recall that we have the following decomposition of $V$ into eigenspaces of $\mathbf{a}$:
\begin{align*}
    V = \bigoplus_{i} V_{\mu_i}
\end{align*}
where $\mu_1 < \cdots < \mu_m = \mu_{\max}$. Recall that $k_i = \dim V_{\mu_i}$. Since $V$ is a nontrivial representation, $\mu_1 < 0 < \mu_m$. 

In the rest of the proof, we set
\begin{align*}
    \mathbf{r} = (\mathsf{r_1}, \ldots, \mathsf{r_m}) \text{ with } \mathsf{r_i} = \frac{\mu_1 - \mu_i}{\mu_1 - \mu_m}.
\end{align*}
Also for simplicity, let $\tilde{\delta} = e^{-\mu_{1}\ell} \delta$ and $\rho = e^{(\mu_1-\mu_{m})\ell}$. For all $u \in \mathsf{B}^U_1$ all subset $F' \subseteq F$ with $\mu_F(F') \geq e^{-\epsilon(\mu_{m} - \mu_{m - 1})\ell}$, we have
\begin{align*}
    |a_{\ell}u. F'|_{\delta} = |u. F'|_{\tilde{\delta}\mathcal{D}_{\rho}^{\mathbf{r}}}
    \gg \sum_{Q \in \mathcal{D}_{\tilde{\delta}}} |u. F'_Q|_{\tilde{\delta}\mathcal{D}_{\rho}^{\mathbf{r}}}
    = \sum_{Q \in \mathcal{D}_{\tilde{\delta}}} |u. \Hom_Q F'_Q|_{\mathcal{D}_{\rho}^{\mathbf{r}}}.
    \end{align*}

    We use $F^Q$ to denote $\Hom_Q F_Q$. We note that $F^Q$ satisfies the following Frostman-type condition:
    \begin{align*}
        \mu_{F^Q} (B^{\mathfrak{r}}_{\rho'}(x)) = \frac{1}{\mu_F(Q)} \mu_F(B^{\mathfrak{r}}_{\tilde{\delta}\rho'}(x'))
        \leq \frac{C(\tilde{\delta})^\alpha}{\mu_F(Q)} (\rho')^{\alpha}
    \end{align*}
    for all $\rho' \geq (\tilde{\delta})^{-1} \delta_0$. 

    Note that by our restriction to $\delta$, $\rho \geq (\tilde{\delta})^{-1} \delta_0$. Therefore, for all $Q \in \mathcal{D}_{\tilde{\delta}}$ so that $\mu_{F_Q}(F'_Q) \geq \rho^{2\epsilon}$, applying \cref{thm:Multislicing} to $\mu_{F^Q}$, there exists $\mathcal{E}_Q \subset \mathsf{B}^U_1$ for all $Q$ with $m_U(\mathcal{E}_Q) \leq C_{\epsilon} \rho^{2\epsilon}$, and for all $u \notin \mathcal{E}_Q$, we have  \begin{align}\label{eqn:SlabExpansionLocal}
    \begin{aligned}
        |u. \Hom_Q F'_Q|_{\mathcal{D}_{s}^{\mathbf{r}}} \geq{}& C_\epsilon^{-1}\frac{\mu_F(Q)}{C(\tilde{\delta})^{\alpha}} \rho^{-\frac{\mu_1}{\mu_1 - \mu_m}\alpha - 2\frac{\mu_{m - 1} - \mu_m}{\mu_1 - \mu_m}\epsilon}\\
        ={}& C_\epsilon^{-1}\mu_F(Q)C^{-1}e^{2(\mu_m - \mu_{m - 1})\epsilon\ell}\delta^{-\alpha}.
    \end{aligned}
    \end{align}
    Let
    \begin{align*}
        \mathcal{D}_{\tilde{\delta}}(u) = \{Q \in \mathcal{D}_{\tilde{\delta}}(F): u \in \mathcal{E}_Q\}
    \end{align*}
    and let
    \begin{align*}
        \mathcal{D}_{\tilde{\delta}}^{\mathrm{large}}(F') = \{Q \in \mathcal{D}_{\tilde{\delta}}: \mu_{F_Q}(F_Q') \geq \rho^{2\epsilon}\}.
    \end{align*}
    Since $\mu_F(F') \geq e^{-\epsilon(\mu_{m} - \mu_{m - 1})\ell} = \rho^{\frac{\mu_{m - 1} - \mu_m}{\mu_1 - \mu_m}\epsilon}$, we have
    \begin{align*}
        \sum_{Q \in \mathcal{D}_{\tilde{\delta}}^{\mathrm{large}}(F')} \mu_F(Q) \geq \rho^{\frac{\mu_{m - 1} - \mu_m}{\mu_1 - \mu_m}\epsilon} - \rho^{2\epsilon} \geq \frac{1}{2}\rho^{\frac{\mu_{m - 1} - \mu_m}{\mu_1 - \mu_m}\epsilon}.
    \end{align*}
    By Fubini's theorem, there exists $\mathcal{E} \subseteq \mathsf{B}^U_1$ with $m_U(\mathcal{E}) \ll_\epsilon \rho^{\frac{\mu_{m - 1} - \mu_m}{\mu_1 - \mu_m}\epsilon}$ so that for all $u \notin \mathcal{E}$, we have
    \begin{align*}
        \sum_{Q \in \mathcal{D}_{\tilde{\delta}}(u)} \mu_F(Q) \leq \frac{1}{4}\rho^{(2 - \frac{\mu_{m - 1} - \mu_m}{\mu_1 - \mu_m})\epsilon}.
    \end{align*}
    Therefore, we have
    \begin{align*}
        |a_\ell u.F'|_{\delta} \gg{}& \Biggl(\sum_{Q \in \mathcal{D}_{\tilde{\delta}}^{\mathrm{large}}(F') \setminus \mathcal{D}_{\tilde{\delta}}(u)} \mu_F(Q)\Biggr) C_\epsilon^{-1}C^{-1} e^{2(\mu_m - \mu_{m - 1})\epsilon\ell}\delta^{-\alpha}\\
        \gg{}& (\frac{1}{2}\rho^{\frac{\mu_{m - 1} - \mu_m}{\mu_1 - \mu_m}\epsilon} - \frac{1}{4}\rho^{(2 - \frac{\mu_{m - 1} - \mu_m}{\mu_1 - \mu_m})\epsilon}) C_\epsilon^{-1}C^{-1} e^{2(\mu_m - \mu_{m - 1})\epsilon\ell}\delta^{-\alpha}\\
        \gg{}& C_\epsilon^{-1}C^{-1} e^{(\mu_m - \mu_{m - 1})\epsilon\ell}\delta^{-\alpha}.
    \end{align*}
    This completes the proof of the theorem. 
\end{proof}

\subsection{Proof of Theorem~\ref{thm:Multislicing}}
The process of combining projection theorems for different directions to an estimate for covering number using an-isotropic tubes is shared by many recent works (e.g. \cite{LMWY25,BH24,BHZ25,BH25b,Lin25,OL25}). In \cite{BH25b}, they provide a multi-slicing machinery that provides a general framework to obtain dimension improvement result from super-critical estimates and subcritical estimates. Theorem~\ref{thm:Multislicing} follows directly from Theorem~\ref{thm:supcritical}, \ref{thm:subcritical entropy general irrep} and \cite[Theorem 3.4]{BH25b} (or more precisely Lemma A.8 there). Note that since Theorem~\ref{thm:supcritical} is a directly supercritcal estimate (instead of an alternative), the condition~$(S^+A)$ there is automatically satisfied and we can take $(\mathbf{j}, \mathbf{r}) = (\mathbf{k}, \mathbf{s})$ and $(\mathcal{V}_{Q, \theta})_{\theta} = (\mathcal{W}_{Q, \theta})_{\theta}$ there. 

\section{Submodularity inequality in irreducible representations}\label{sec:outline subcritical}
The goal of the rest of the paper is to prove Theorem~\ref{thm:subcritical entropy general irrep} (recorded below) and it can be read independently to the other parts of this paper. 

\begin{theorem*}
    Suppose the representation $V$ of $H$ satisfies only {}\GIrrep{}. There exists $M > 1$ depending only on $U$ and $V$ so that for all $0 < \epsilon < \frac{1}{100}$ and $0 < \delta \ll_{\epsilon, V} 1$ the following holds for all eigenvalue $\mu$ of $\mathbf{a}$. 

    Let $A \subseteq B_1^V$, we set the exceptional set of projection to be 
    \begin{align*}
        \mathcal{E}(A) = \{u \in B_1^U: |\pi^{(\mu)}_u(A)|_{\delta} < \delta^{M\epsilon} |A|_{\delta}^{\frac{\dim V^{(\mu)}}{\dim V}}\}.
    \end{align*}
    Then we have
    \begin{align*}
        m_U(\mathcal{E}(A)) \leq \delta^\epsilon.
    \end{align*}
\end{theorem*}

Before we discuss the proof of the above theorem, we record two interesting consequences of the above theorem (and its proof). The proof of them are standard and are omitted in this paper since we do not need them. The notation $\dim_{\mathrm{H}}$ stands for Hausdorff dimension. 
\begin{theorem}
    Suppose $V$ is an irreducible representation of $H$. Then for all Borel set $A \subseteq B_1^V$, we have
    \begin{align*}
        \dim_{\mathrm{H}} \pi_{u}^{(\mu)}(A) \geq \frac{\dim V^{(\mu)}}{\dim V}\dim_{\mathrm{H}} A
    \end{align*}
    for $m_{\mathsf{B}_1^U}$-almost every $u$. 
\end{theorem}
Similarly, we have the following theorem for $H$-action. 
\begin{theorem}\label{thm:Hausdorff H action}
    Suppose $V$ is an irreducible representation of $H$. Then for all Borel set $A \subseteq B_1^V$ and all subspace $W \subseteq V$ we have
    \begin{align*}
        \dim_{\mathrm{H}} \pi_{h.W}(A) \geq \frac{\dim W}{\dim V}\dim_{\mathrm{H}} A
    \end{align*}
    for $m_{H}$-almost every $h$. 
\end{theorem}

We now discuss the proof of Theorem~\ref{thm:subcritical entropy general irrep}. An important ingredient and also an interesting bi-product is to establish the following estimate for dimension of projections of linear subspaces. 

\begin{theorem}\label{thm:subcritical linear general irrep proj}
    For any subspace $W \subseteq V$ of dimension $k$, there exists $m \ll_{n, k} 1$ and a Zariski open dense subset $\mathcal{O}_W \subseteq H^m$ so that for all $(h_1, \ldots, h_m) \in \mathcal{O}_W$ and all subspace $W' \subseteq V$
    \begin{align}\label{eqn:linear critical average proj}
        \dim \pi_{h_1.W} (W') + \cdots + \dim \pi_{h_m.W} (W') \geq \frac{mk}{n} \dim W'.
    \end{align}
    The integer $m$ and the subset $\mathcal{O}_W \subseteq H^m$ does \emph{not} depend on $W'$. 
\end{theorem}

Theorem~\ref{thm:subcritical linear general irrep proj} is equivalent to the following. 

\begin{theorem}\label{thm:subcritical linear general irrep intersection}
    For all subspace $W \subseteq V$ of dimension $k$, there exist positive integer $m \ll_{n, k} 1$ and a Zariski open dense subset $\mathcal{O}_W \subseteq H^m$ so that for all $(h_1, \ldots, h_m) \in \mathcal{O}_W$ and all subspace $W' \subseteq V$
    \begin{align}\label{eqn:linear critical intersection}
        \dim [h_1.W \cap W'] + \cdots + \dim [h_m.W \cap W'] \leq \frac{mk}{n} \dim W'.
    \end{align}
    The integer $m$ and the subset $\mathcal{O}_W \subseteq H^m$ does \emph{not} depend on $W'$. 
\end{theorem}

In fact, we will prove Theorem~\ref{thm:subcritical linear general irrep intersection} in a more general setting. Let $F$ be a field and let $\tilde{H}$ be a connected linear algebraic group over $F$. Let $V$ be a regular irreducible representation of $\tilde{H}$ defined over $F$. By irreducible we mean there is no non-trivial sub-representation. For simplicity, we still use $n$ to denote the dimension of $V$. 

\begin{theorem}\label{thm:subcritical linear general irrep general field}
    For all subspace $W \subseteq V$ of dimension $k$, there exists $m \ll_{n, k} 1$ and a Zariski open dense subset $\mathcal{O}_W \subseteq \tilde{H}^m$ so that for all $(\tilde{h}_1, \ldots, \tilde{h}_m) \in \mathcal{O}_W$ and all subspace $W' \subseteq V$
    \begin{align}\label{eqn:linear critical intersection general}
        \dim [\tilde{h}_1.W \cap W'] + \cdots + \dim [\tilde{h}_m.W \cap W'] \leq \frac{mk}{n} \dim W'.
    \end{align}
    The integer $m$ and the subset $\mathcal{O}_W \subseteq \tilde{H}^m$ does \emph{not} depend on $W'$. 
\end{theorem}

As a consequence, we have the following. This is Theorem~\ref{thm:dimension bound one subspace}. 

\begin{corollary}\label{cor:dimension bound one subspace}
    For all subspaces $W, W' \subseteq V$, there exists $\tilde{h} \in \tilde{H}$ so that 
    \begin{align}
        \dim [(\tilde{h}.W) \cap W'] \leq \frac{\dim W}{\dim V} \dim W'.
    \end{align}
    Note that the set of such $\tilde{h}$ therefore forms a Zariski open dense subset of $\tilde{H}$.
\end{corollary}

\begin{remark}\label{rem:history}
We give an overview of related works here. 

The case of Corollary~\ref{cor:dimension bound one subspace} where $\tilde{H}$ is a real reductive group defined over $\Q$ and the representation \emph{and the subspace $W$} are both defined over $\Q$ was proved by Eskin--Mozes--Shah \cite[Corollary 1.4]{EMS97}. Their proof is based on a non-divergence estimate for translates of algebraic measure on homogeneous spaces. Some explicit cases of Corollary~\ref{cor:dimension bound one subspace} were proved in \cite{Lin25} and \cite{BHZ25}. In an earlier version of the work by B{\'e}nard--He \cite{BH25b}, they proved Corollary~\ref{cor:dimension bound one subspace} for adjoint representation of complex simple Lie groups by a case-by-case study. The earlier version of this paper provided a general proof for Corollary~\ref{cor:dimension bound one subspace} with the condition where $F$ is of characteristic $0$, but the proof only uses $F$ having infinite cardinality. Note that when $F$ is finite the only connected linear algebraic group is the trivial group and the theorem holds automatically. Afterwards, an independent proof of Corollary~\ref{cor:dimension bound one subspace} is given in the second version of \cite{BH25b}. 

The proof we give here is a generalization of argument in \cite{Lin25} which dealt with special cases where $(G, H) = (\SL_4(\R), \SO(2,2))$ and $(G, H) = (\SL_4(\R), \SO(3,1))$. Our proof is purely combinatorial and is based on the idea from convexity property of entropy. It utilizes only irreducibility and Zariski connectedness. The upshot of our proof is that it automatically gives the stronger verion, Theorem~\ref{thm:subcritical linear general irrep general field} where the estimate is on average but the exceptional set does \emph{not} depend on $W'$. 
\end{remark}

We now sketch the idea of the proof of Theorem~\ref{thm:subcritical entropy general irrep} and simultaneously Theorem~\ref{thm:Hausdorff H action} and \ref{thm:subcritical linear general irrep general field}. Let us recall the following sub-modularity inequality for entropy. It follows directly from convexity of conditional entropy. 
\begin{lemma}
    For all probability measure $\mu$ and all partition $\mathcal{P}, \mathcal{Q}, \mathcal{R}$ and $\mathcal{S}$ satisfying
    \begin{align*}
        \mathcal{P} \vee \mathcal{Q} = \mathcal{R}
    \end{align*}
    and $\mathcal{P}, \mathcal{Q}$ being refinement of $\mathcal{S}$, 
    we have
    \begin{align*}
        H(\mu, \mathcal{P}) + H(\mu, \mathcal{Q}) \geq H(\mu, \mathcal{R}) + H(\mu, \mathcal{S}).
    \end{align*}
\end{lemma}

For simplicity, let us first get back to the setting where $H$ is a $\R$-group. 

We now outline the proof of Theorem~\ref{thm:subcritical entropy general irrep}. Without loss of generality, $\delta$ is a dyadic scale. Let $\mu$ be a probability measure supported on $B_1^V$. For Theorem~\ref{thm:subcritical entropy general irrep}, one can take $\mu = \mu_{A^{(\delta)}}$ to be the renormalized Lebesgue measure on the $\delta$-neighborhood $A^{(\delta)}$ of $A$. For Theorem~\ref{thm:Hausdorff H action}, one can apply Frostman's lemma to get a measure $\mu$ with certain dimension condition. 

For all subspace $W \subseteq V$, let $\mathcal{D}_{\delta}(W)$ be the partition of $W$ using $\delta$-dyadic cubes. For simplicity, we will denote it by $\mathcal{D}_{\delta}$ if the subspace $W$ is clear from the context. In the rest of this section, we always assume $W \neq \{0\}, V$. 

We will outline a proof of $H(\mu, \pi_{h.W}^{-1}\mathcal{D}_{\delta}) \geq \frac{\dim W}{\dim V}H(\mu, \mathcal{D}_{\delta}) - O(1)$ for generic $h$ in the following. This implies Theorem~\ref{thm:Hausdorff H action}. The proof can be adapted to covering number. The transition between $H$-action and $U$-action follows directly from the restricted root space decomposition. 

By irreducibility of $V$, there exists a minimal positive integer $q \leq n$ so that for generic $(h_1, \ldots, h_q)$, 
\begin{align*}
    h_1.W + \cdots + h_q.W = V.
\end{align*}

Note that 
\begin{align}\label{eqn:partitionfromprojectionandrefinement}
    \pi_{W_1 + W_2}^{-1} \mathcal{D}_{\delta} \sim \pi_{W_1}^{-1} \mathcal{D}_{\delta} \vee \pi_{W_2}^{-1} \mathcal{D}_{\delta}
\end{align}
where the constant depends on $W_1$ and $W_2$. 
Therefore, iterating the sub-modularity inequality for entropy, we have
\begin{align*}
    {}&H(\mu, \pi_{h_1.W}^{-1} \mathcal{D}_{\delta}) + \cdots + H(\mu, \pi_{h_q.W}^{-1} \mathcal{D}_{\delta})\\
    \geq{}& H(\mu, \pi_{\sum_{j = 1}^m h_j.W}^{-1} \mathcal{D}_{\delta}) + \sum_{k = 1}^{q - 1} H(\mu, \pi_{(\sum_{j = 1}^k h_j.W) \cap h_{k + 1}.W}^{-1} \mathcal{D}_{\delta}) - O(1)\\
    ={}& H(\mu, \mathcal{D}_{\delta}) + \sum_{k = 1}^{q - 1} H(\mu, \pi_{(\sum_{j = 1}^k h_j.W) \cap h_{k + 1}.W}^{-1} \mathcal{D}_{\delta}) - O(1)
\end{align*}
where the implied constant depends on $(h_1, \ldots, h_q)$ (in fact polynomially). Note that by minimality of $q$, for all $k = 1, \ldots, q - 1$ we have
\begin{align*}
    \dim \left(\sum_{j = 1}^k h_j.W\right) \cap h_{k + 1}.W < \dim W.
\end{align*}
Suppose we can run certain induction argument, then we have
\begin{align*}
    {}&H(\mu, \pi_{h_1.W}^{-1} \mathcal{D}_{\delta}) + \cdots + H(\mu, \pi_{h_q.W}^{-1} \mathcal{D}_{\delta})\\
    \geq{}& H(\mu, \mathcal{D}_{\delta}) + \sum_{k = 1}^{q - 1} \frac{\dim (\sum_{j = 1}^k h_j.W) \cap h_{k + 1}.W}{\dim V}H(\mu,  \mathcal{D}_{\delta}) - O(1)\\
    ={}& \frac{H(\mu,  \mathcal{D}_{\delta})}{\dim V} \left(\dim V + \sum_{k = 1}^{q - 1}\left[\dim \left(\sum_{j = 1}^k h_j.W \right) + \dim h_{k+1}.W - \dim \left(\sum_{j = 1}^{k + 1} h_j.W\right)\right]\right)\\
    {}&- O(1)\\
    ={}& \frac{\dim W}{\dim V} H(\mu, \mathcal{D}_{\delta})  - O(1).
\end{align*}
This shows the entropy estimate in Theorem~\ref{thm:subcritical entropy general irrep} in average. Moreover, the choice of $(h_j)_j$ does \emph{not} depend on the measure $\mu$.  

To run the above induction argument, we need to study the families of subspaces of the form $\left(\sum_{j = 1}^k h_j.W\right) \cap h_{k + 1}.W$. They are no longer in the same Grassmannian nor form a single $H$-orbit. However, for generic choice of $(h_1, \ldots, h_q)$, it lies in a same Grassmannian. This resolves the first problem. Note that the families of subspaces of the form $\left(\sum_{j = 1}^k h_j.W\right) \cap h_{k + 1}.W$ are stratified by $H$-orbit. A Fubini-type argument allows one to run an induction argument sketched as above. 

We remark that such argument utilizes only a sub-modularity type inequality and certain Fubini argument. The first has an analogue for dimension of linear space over any field, see Lemma~\ref{lem:submodularity linear}. The latter can be adapt to general field utilizing the fact that $H$ acts on the Grassmannian regularly with orbit being constructible sets. Hence the same idea can be used to prove Theorem~\ref{thm:linear subcritical irrep strong}. 

The above outline can be translated into a rigorous proof. However, the details are quite tedious, especially in keeping track on constants from \cref{eqn:partitionfromprojectionandrefinement}, the comparison between $\pi_{W_1 + W_2}^{-1} \mathcal{D}_{\delta}$ and $\pi_{W_1}^{-1} \mathcal{D}_{\delta} \vee \pi_{W_2}^{-1} \mathcal{D}_{\delta}$. 

The proof given here follows a slightly different outline and make use of recent developments of H{\"o}lder--Brascamp--Lieb inequality. Roughly speaking, such inequality provides a linear criterion to the entropy estimate in Theorem~\ref{thm:subcritical entropy general irrep}, see e.g. \cite[Theorem 2.1]{BCCT}. Therefore, we can prove the linear version of Theorem~\ref{thm:subcritical entropy general irrep}, Theorem~\ref{thm:subcritical linear general irrep proj} using the above outline without tedious tracking constant procedure and apply recent development of H{\"o}lder--Brascamp--Lieb inequality in \cite{Gre21} to prove Theorem~\ref{thm:subcritical entropy general irrep}. 

We remark that very recently, B{\'e}nard--He established an effective Brascamp--Lieb inequality \cite{BH25a}. It will be interesting to compare the estimate of Brascamp--Lieb constant there with the estimate in this paper which builds on Gressman's work \cite{Gre21}. The latter in turns, utilizes geometric invariant theory. 

The rest of the paper is organized as the following. Section~\ref{sec:Linear case} is devoted to the proof of Theorem~\ref{thm:subcritical linear general irrep general field}. Then we collect results from recent developments of H{\"o}lder--Brascamp--Lieb inequality from \cite{Gre21} in Section~\ref{sec:Brascamp Lieb}. In Section~\ref{sec:translate to entropy}, we combine all the ingredients to prove Theorem~\ref{thm:subcritical entropy general irrep}. 

\section{Linear submodularity inequality in irreducible representations}\label{sec:Linear case}
The goal of this section is to prove Theorem~\ref{thm:subcritical linear general irrep general field}. In fact, we will prove a more general version, Theorem~\ref{thm:linear subcritical irrep strong}. Before stating the result, let us introduce the following notion. 

Without loss of generality, in this section, $F$ is an infinite field\footnote{Note that over finite field the only connected algebraic group is the trivial group, the theorem holds automatically. } and $H$ is a connected algebraic group defined over $F$. Let $V$ be an irreducible representation of $H$ of dimension $n$. For all $0 \leq k \leq n$, we use $\Gr_k(V)$ to denote the Grassmannian of $k$-dimensional subspaces in $V$. The actions of $h \in H$ on $H^\ell$ and $\prod_{i = 1}^q \Gr_{k_i}(V)$ in this section are referred to the left actions by diagonally embedded $h$ in $H^\ell$ and $H^q$ respectively. 

\begin{definition}
    We say $\mathcal{W} \subseteq \Gr_k(V)$ is $H$-stratified with data $(\ell, \mathcal{O}, \Psi)$ if there exists $\ell \geq 1$, a $H$-left invariant Zariski open dense subset $\mathcal{O} \subseteq H^\ell$, and an left $H$-equivariant regular map $\Psi:\mathcal{O} \to \Gr_k(V)$ so that $\Psi(\mathcal{O}) = \mathcal{W}$.
\end{definition}

\begin{theorem}\label{thm:linear subcritical irrep strong}
    Suppose $\mathcal{W}$ is a $H$-stratified subset of $\Gr_k(V)$ with data $(\ell, \mathcal{O}, \Psi)$, then there exist a positive integer $m \ll_n 1$ and a Zariski open dense subset $\mathcal{O}_{\mathcal{W}} \subseteq \mathcal{O}^{m} \subseteq (H^{\ell})^m$ so that for all subspace $W' \subseteq V$ and all $(\mathbf{h}_1, \ldots, \mathbf{h}_m) \in \mathcal{O}$, we have
    \begin{align}
        \dim (\Psi(\mathbf{h}_1) \cap W') + \cdots + \dim (\Psi(\mathbf{h}_m) \cap W') \leq \frac{mk}{n} \dim W'.
    \end{align}
    The bound of integer $m$ does \emph{not} depend on $\mathcal{W}$ and its corresponding data $(\ell, \mathcal{O}, \Psi)$. The subset $\mathcal{O}_{\mathcal{W}}$ does \emph{not} depend on $W'$. 
\end{theorem}

\begin{proof}[Proof of Theorem~\ref{thm:subcritical linear general irrep general field} assuming Theorem~\ref{thm:linear subcritical irrep strong}]
    Let $\mathcal{W} = H.W$, then it follows directly from Theorem~\ref{thm:linear subcritical irrep strong}.
\end{proof}

To study families of subspaces of form $\left(\sum_{j = 1}^k h_j.W\right) \cap h_{k + 1}.W$, we introduce the following notation. 
\begin{definition}\label{def: tree operation}
    Fix $h \geq 1$ and a connected finite (rooted) tree $\mathcal{T}$ with height $h$ satisfying the following conditions.
    \begin{enumerate}
        \item Let $o$ be the root of $\mathcal{T}$. If $h \geq 2$, $\deg (o) \geq 2$.
        \item Any vertex either has $\geq 2$ descendants or has no descendant. The set of the vertices with no descendant is denoted by $\mathcal{L}(\mathcal{T})$, the leaves of $\mathcal{T}$. 
    \end{enumerate}
    Fix a function
    \begin{align*}
        \omega: \{v \in \mathcal{T}: v\text{ has $\geq 2$ descendants} \} \to \left\{\sum, \bigcap\right\},
    \end{align*}
    for all $k$, we define a map
    \begin{align*}
        \Phi^{(\mathcal{T})}_{\omega}: \Gr_k(V)^{\mathcal{L}(\mathcal{T})} \to \bigsqcup_{i = 0}^n \Gr_i(V)
    \end{align*}
    in the following inductive way. 
    \begin{enumerate}
        \item If $\mathcal{T}$ has height $1$, then the function $\omega$ is a function with domain being empty set, we define $\Phi^{(\mathcal{T})}_{\omega}$ to be the identity map on $\Gr_k(V)$.
        \item If $\mathcal{T}$ has height $h \geq 2$, its root $o$ has $\geq 2$ descendant. Let $\{\mathcal{T}_i\}_{i = 1}^{d}$ be the sub-trees of $\mathcal{T}$ rooted at descendants of $o$ and let $\omega_i$ be the restriction of $\omega$ to $\mathcal{T}_i$. We can identify
        \begin{align*}
            \Gr_k(V)^{\mathcal{L}(\mathcal{T})} \cong \prod_{i = 1}^d \Gr_k(V)^{\mathcal{L}(\mathcal{T}_i)}
        \end{align*}By inductive hypothesis, $\Phi^{(\mathcal{T}_i)}_{\omega_i}$ has been defined. We define
        \begin{align*}
            \Phi^{(\mathcal{T})}_{\omega} = \begin{cases}
                \sum_{i = 1}^d \Phi^{(\mathcal{T}_i)}_{\omega_i}, & \text{ if } \omega(o) = \sum,\\
                \bigcap_{i = 1}^d \Phi^{(\mathcal{T}_i)}_{\omega_i}, & \text{ if } \omega(o) = \bigcap.
            \end{cases}
        \end{align*}
    \end{enumerate}
\end{definition}

We have the following direct consequence of the $H$-equivariance of $\Phi^{(\mathcal{T})}_\omega$. 
\begin{lemma}
    For all $H$-invariant subset $\mathcal{W} \subseteq \Gr_k(V)$, all finite connected rooted tree $\mathcal{T}$, function $\omega$ and $i = 0, \ldots, n$, the set 
    \begin{align*}
        \Phi^{(\mathcal{T})}_\omega (\mathcal{W}) \cap \Gr_i(V)
    \end{align*}
    is $H$-invariant. 
\end{lemma}

The following lemma records that from the operation $\Phi^{(\mathcal{T})}_{\omega}$, we can form an $H$-stratified subset of certain Grassmannian. It is a direct consequence of the construction. 

\begin{lemma}\label{lem:generic dimension}
    Let $\mathcal{W}$ be a $H$-stratified set with data $(\ell, \mathcal{O}, \Psi)$. Then for all finite connected rooted tree $\mathcal{T}$ and function $\omega$ satisfying the conditions in Definition~\ref{def: tree operation}, there exist an integer $k_{\mathcal{T}, \omega}$ and an $H$-invariant Zariski open dense subset $\mathcal{O}_{\mathcal{T}, \omega} \subseteq \mathcal{O}^{\mathcal{L}(\mathcal{T})} \subseteq (H^{\ell})^{\mathcal{L}(\mathcal{T})}$ so that the following holds. 
    \begin{enumerate}
        \item The image $\Phi^{(\mathcal{T})}_{\omega}[(\Psi(\mathcal{O}_{\mathcal{T}, \omega}))^{\mathcal{L}(\mathcal{T})}]$ lies in $\Gr_{k_{\mathcal{T}, \omega}}(V)$.
        \item The map $\Phi^{(\mathcal{T})}_{\omega} \circ (\Psi^{\mathcal{L}(\mathcal{T})})$ is regular on $\mathcal{O}_{\mathcal{T}, \omega}$. 
    \end{enumerate}
\end{lemma}

The following lemma utilizes the irreducibility of $V$. 

\begin{lemma}\label{lem:move out orbit}
    For all non-trivial subspaces $W, W' \subseteq V$, there exists a Zariski open dense subset $\mathcal{O}_{W, U} \subseteq H$ so that for all $h \in \mathcal{O}_{W, W'}$, we have
    \begin{align*}
        \dim [(h.W) \cap W'] < \min\{\dim W, \dim W'\}.
    \end{align*}
\end{lemma}
\begin{proof}
    Note that the condition $\dim [(h.W) \cap W'] < \min\{\dim W, \dim W'\}$ is an open condition. Suppose this does not hold, then for all $h \in H$, we have
    \begin{align*}
        \dim W \geq \dim [(h.W) \cap W'] \geq \dim W.
    \end{align*}
    This implies that $h.W \subseteq W'$ for all $h \in H$ and hence $W'$ contains a nontrivial sub-representation, contradicting to the fact that $V$ is irreducible. 
\end{proof}

\begin{lemma}\label{lem:move out family}
    Let $\mathcal{W}$ be a $H$-stratified set in $\Gr_k(V)$ with data $(\ell, \mathcal{O}, \Psi)$ where $0 < k < n$. 
    For all non-trivial subspace $W' \subseteq V$, there exists a Zariski open dense subset $\mathcal{O}_{\mathcal{W}, W'} \subseteq \mathcal{O}$ so that for all $\mathbf{h} \in \mathcal{O}_{\mathcal{W}, W'}$, we have
    \begin{align*}
        \dim [\Psi(\mathbf{h}) \cap W'] < \min\{k, \dim W'\}.
    \end{align*}
\end{lemma}
\begin{proof}
    Note that by regularity of $\Psi$, the condition $\dim [\Psi(\mathbf{h}) \cap W'] < \min\{k, \dim W'\}$ is an open condition. It suffices to show that there exists $\mathbf{h} \in \mathcal{O}$ so that the inequality holds. The existence of such $\mathbf{h}$ follows from Lemma~\ref{lem:move out orbit}, $H$-invariance of $\mathcal{O}$ and the condition that $\Psi$ is $H$-equivariant. 
\end{proof}

The following lemma is a direct consequence of Lemma~\ref{lem:generic dimension} and Lemma~\ref{lem:move out family}. 
\begin{lemma}\label{lem:sum is whole space}
    Let $\mathcal{W}$ be a $H$-stratified set in $\Gr_k(V)$ with data $(\ell, \mathcal{O}, \Psi)$ where $0 < k < n$. There exist a minimal positive integer $1 \leq q \leq n$ and a $H$-left invariant Zariski open dense subset $\tilde{\mathcal{O}} \subseteq \mathcal{O}^q \subseteq (H^\ell)^q$ so that for all $(\mathbf{h}_i)_{i = 1, \ldots, q} \in \tilde{\mathcal{O}}$, we have
    \begin{align*}
        \Psi(\mathbf{h}_1) + \cdots + \Psi(\mathbf{h}_q) = V.
    \end{align*}
    Moreover, for all $2 \leq q' \leq q$, there exists $k_{q'} < k$ so that the following hold. For all $(\mathbf{h}_i)_{i = 1}^{q}$, we have $\dim \left(\sum_{i = 1}^{q' - 1}\Psi(\mathbf{h}_i)\right) \cap \Psi(\mathbf{h}_{q'}) = k_{q'}$ and the map
              \begin{align*}
               \Psi^{(q')}: \tilde{\mathcal{O}} {}&\to \Gr_{k_{q'}}(V)\\
               (\mathbf{h}_i)_{i = 1}^{q} {}&\mapsto \left(\sum_{i = 1}^{q' - 1}\Psi(\mathbf{h}_i)\right) \cap \Psi(\mathbf{h}_{q'}).
         \end{align*}
          is regular. The $(\ell q, \tilde{\mathcal{O}}, \Psi^{(q')})$ forms a data of a $H$-stratified set in $\Gr_{k_{q'}}(V)$. In particular, for all $(\mathbf{h}_i)_{i = 1}^{q} \in \tilde{\mathcal{O}}$ and all $1 < q' \leq q$, 
    \begin{align*}
        \dim \left(\sum_{i = 1}^{q' - 1}\Psi(\mathbf{h}_i)\right) \cap \Psi(\mathbf{h}_{q'}) < k.
    \end{align*}
\end{lemma}

The following standard linear algebra lemma can be viewed as a (linear) dual version of sub-modularity inequality for entropy. 
\begin{lemma}\label{lem:submodularity linear}
    For all subspaces $W', W_1, W_2 \subseteq V$, we have
    \begin{align*}
        \dim W' \cap W_1 \cap W_2 + \dim W'\cap (W_1 + W_2) \geq \dim W' \cap W_1 + \dim W' \cap W_2.
    \end{align*}
\end{lemma}

\begin{proof}[Proof of Theorem~\ref{thm:linear subcritical irrep strong}]
    The idea of the proof is recorded in the previous section. We recommend the reader to read the proof accompanied with previous section to avoid the confusion from the heavy notations. Without loss of generality, we assume $0 < k < n$ and $W' \neq \{0\}$ and $W' \neq V'$. 

    We prove the theorem by induction on $k$. For $k = 1$, let
    \begin{align*}
        \mathcal{U} = \{(\mathbf{h}_1, \ldots, \mathbf{h}_n) \in \mathcal{O}^n: \Psi(\mathbf{h}_1) + \cdots + \Psi(\mathbf{h}_n) = V\}.
    \end{align*}
    By Lemma~\ref{lem:move out family}, the set $\mathcal{U} \neq \emptyset$. Moreover, $k = 1$ implies that for all $(\mathbf{h}_1, \ldots, \mathbf{h}_n) \in \mathcal{U}$ the sum above is a direct sum. 
    Since $\dim\left(\sum_{i = 1}^n\Psi(\mathbf{h}_i)\right) = n$ is an open condition, the set $\mathcal{U}$ is a Zariski open dense subset. Let $m = n$ and $\mathcal{O}_{\mathcal{W}} = \mathcal{U}$, for all $(\mathbf{h}_i)_{i = 1}^n \in \mathcal{O}_{\mathcal{W}} = \mathcal{U}$ and subspace $W' \subseteq V$, we have
    \begin{align*}
        \sum_{i = 1}^n \dim \Psi(\mathbf{h}_i) \cap W' \leq \dim \left(\sum_{i = 1}^n\Psi(\mathbf{h}_i)\right) \cap W' = \dim W' = n\frac{1}{n}\dim W'.
    \end{align*}
    The proof of the base case is complete. 

    We now prove the inductive step. Suppose the theorem holds for all $H$-stratified subsets in $\Gr_{k'}(V)$ for all $k' \leq k - 1$. By Lemma~\ref{lem:sum is whole space}, there exist integer $1 \leq q \leq n$ and $H$-left invariant Zariski open dense subset $\tilde{\mathcal{O}} \subseteq \mathcal{O}^q \subseteq (H^\ell)^q$ so that for all $(\mathbf{h}_i)_{i = 1, \ldots, q} \in \tilde{\mathcal{O}}$, we have
    \begin{align*}
        \Psi(\mathbf{h}_1) + \cdots + \Psi(\mathbf{h}_q) = V.
    \end{align*}
    Moreover, for all $2 \leq q' \leq q$ there exists $k_{q'} < k$ so that the map
        \begin{align*}
            \Psi^{(q')}: \tilde{\mathcal{O}} {}&\to \Gr_{k_{q'}}(V)\\
            (\mathbf{h}_i)_{i = 1}^{q} {}&\mapsto \left(\sum_{i = 1}^{q' - 1}\Psi(\mathbf{h}_i)\right) \cap \Psi(\mathbf{h}_{q'}).
         \end{align*}
    is well-defined and regular. 

    Applying the inductive hypothesis to the $H$-stratified data $(\ell q, \tilde{\mathcal{O}}, \Psi^{(q')})$ in $\Gr_{k_{q'}}(V)$ for all $2 \leq q' \leq q$, there exists $m_{q'} \ll_n 1$ and $\tilde{\mathcal{O}}^{(q')} \subseteq \tilde{\mathcal{O}}^{m_{q'}}$ so that for all $[(\mathbf{h}_{i}^{(j)})_{i = 1}^{q}]_{j = 1}^{m_{q'}} \in \tilde{\mathcal{O}}^{(q')}$ and all subspace $W' \subseteq V$, we have
    \begin{align}\label{eqn:inductive hypothesis}
    \begin{aligned}
        {}&\sum_{j = 1}^{m_{q'}} \dim \left(\left(\sum_{i = 1}^{q' - 1} \Psi(\mathbf{h}_{i}^{(j)})\right) \cap \Psi(\mathbf{h}_{q'}^{(j)}) \cap W'\right)\\
        ={}& \sum_{j = 1}^{m_{q'}} \dim \left(\Psi^{(q')}((\mathbf{h}_{i}^{(j)})_{i = 1}^{q}) \cap W'\right)
        \leq m_{q'}\frac{k_{q'}}{n} \dim W'.
    \end{aligned}
    \end{align}
    
    Let $m = n! \prod_{i = 1}^{q - 1} m_i$. Since $q \leq n$ and by inductive hypothesis $m_i \ll_n 1$, we have $m \ll_n 1$. Let
    \begin{align*}
        \mathcal{O}_{\mathcal{W}} = \tilde{\mathcal{O}}^{\frac{m}{q}} \cap \bigcap_{2 \leq q' \leq q} \left(\tilde{\mathcal{O}}^{(q')}\right)^{\frac{m}{m_{q'}}} \subseteq \mathcal{O}^m.
    \end{align*}
    Since $H$ is Zariski connected, $\mathcal{O}_{\mathcal{W}}$ is a Zariski open dense subset. For all $[(\mathbf{h}_{i}^{(j)})_{i = 1}^{q}]_{j = 1}^{\frac{m}{q}} \in \tilde{\mathcal{O}}_{\mathcal{W}}$, by Lemma~\ref{lem:submodularity linear} we have
    \begin{align*}
        {}&\sum_{j = 1}^{\frac{m}{q}} \sum_{i = 1}^q \dim \left(W'\cap \Psi(\mathbf{h}_{i}^{(j)})\right)\\\leq{}& \sum_{j = 1}^{\frac{m}{q}} \left[\dim W' \cap\left(\sum_{i = 1}^q \Psi(\mathbf{h}_{i}^{(j)})\right)  + \sum_{q' = 2}^q \dim W' \cap\left(\sum_{i = 1}^{q' - 1} \Psi(\mathbf{h}_{i}^{(j)})\right) \cap \Psi(\mathbf{h}_{q'}^{(j)})\right]\\
        ={}& \frac{m}{q} \dim W' + \sum_{q' = 2}^q \sum_{p = 1}^{\frac{m}{qm_{q'}}} \sum_{j = 1}^{m_{q'}}\dim W' \cap \Psi^{(q')} ((\mathbf{h}_{i}^{(j + (p-1)m_{q'})})_{i = 1}^q).
    \end{align*}
    Applying \cref{eqn:inductive hypothesis} (from inductive hypothesis), we have
    \begin{align}\label{eqn:apply inductive hypothesis}
    \begin{aligned}
        {}&\sum_{j = 1}^{\frac{m}{q}} \sum_{i = 1}^q \dim \left(W'\cap \Psi(\mathbf{h}_{i}^{(j)})\right)\\
        \leq{}& \frac{m}{q} \dim W' + \sum_{q' = 2}^q \sum_{p = 1}^{\frac{m}{qm_{q'}}} m_{q'}\frac{k_{q'}}{n} \dim W'
        = \frac{m}{q} \left(1 + \sum_{q' = 2}^q \frac{k_{q'}}{n}\right) \dim W'.
    \end{aligned}
    \end{align}
    We now calculate $\sum_{q' = 2}^q k_{q'}$. Note that for all $(\mathbf{h}_i)_{i = 1}^{q} \in \tilde{\mathcal{O}}$, we have
    \begin{align*}
        k_{q'} = \dim \left(\sum_{i = 1}^{q' - 1}\Psi(\mathbf{h}_i)\right) \cap \Psi(\mathbf{h}_{q'}) = \dim \left(\sum_{i = 1}^{q' - 1}\Psi(\mathbf{h}_i)\right) + k - \dim \left(\sum_{i = 1}^{q'}\Psi(\mathbf{h}_i)\right).
    \end{align*}
    Therefore, 
    \begin{align*}
        \sum_{q' = 2}^q k_{q'} = \dim \Psi(\mathbf{h}_1) + (q - 1)k - \dim \sum_{i = 1}^q\Psi(\mathbf{h}_i) = qk - n.
    \end{align*}
    Combining this with \cref{eqn:apply inductive hypothesis}, we have
    \begin{align*}
        \sum_{j = 1}^{\frac{m}{q}} \sum_{i = 1}^q \dim \left(W'\cap \Psi(\mathbf{h}_{i}^{(j)})\right) \leq m \frac{k}{n} \dim W'
    \end{align*}
    for all $[(\mathbf{h}_{i}^{(j)})_{i = 1}^{q}]_{j = 1}^{\frac{m}{q}} \in \tilde{\mathcal{O}}_{\mathcal{W}}$. This completes the proof of the inductive step and hence the theorem. 
\end{proof}

\section{The H{\"o}lder--Brascamp--Lieb inequality}\label{sec:Brascamp Lieb}

The H{\"o}lder--Brascamp--Lieb inequality allow us to translate the linear dimension estimate in Theorem~\ref{thm:subcritical linear general irrep proj} to the entropy estimate in Theorem~\ref{thm:subcritical entropy general irrep}. For reader's convenience, we recall related notions and results in this section. The main result is Proposition~\ref{prop:BL constant for unipotent}. 

In this section, we assume $H$ is a semisimple real algebraic group with Lie algebra $\LieH$ and $V$ is a faithful irreducible representation of $H$ for simplicity. Recall that we fixed an inner product and an orthonormal basis from Mostow's theorem \cite{Mos55} so that the matrix transpose is the adjoint operation under this inner product and $h^t \in H$ for all $h \in H$. We identify $V \cong \R^n$ using this basis. 

For all $j = 1, \ldots, m$, $\{n_j\}_{j = 1}^m \subseteq \N$, linear maps $\pi_j: \R^n \to \R^{n_j}$ and $p_j \in [0, 1]$, we define the Brascamp--Lieb constant $\mathrm{BL}(\{\pi_j, p_j\}_{j = 1}^m)$ to be the smallest non-negative real number so that
\begin{align}
    \int_{\R^n} \prod_{j = 1}^m f_j(\pi_j(x))^{p_j} \,\mathrm{d}x \leq \mathrm{BL}(\{\pi_j, p_j\}_{j = 1}^m) \prod_{j = 1}^m \left(\int_{\R^{n_j}} f_j\right)^{p_j}
\end{align}
for all non-negative measurable functions $f_j \in L^1(\R^{n_j})$. If such real number does not exist, we say $\mathrm{BL}(\{\pi_j, p_j\}_{j = 1}^m) = \infty$ or equivalently $[\mathrm{BL}(\{\pi_j, p_j\}_{j = 1}^m)]^{-1} = 0$. 

For simplicity, we identify $\pi_j$ with its $n_j \times n$ matrix representation under the standard basis of $\R^n$ and $\R^{n_j}$ in this section. Following \cite{Gre21}, we use $\{\pi_j\}_{j = 1}^m$ to denote the \emph{$m$-tuple} $(\pi_1, \ldots, \pi_m)$ for simplicity. 

The following theorem is proved by Bennett--Carbery--Christ--Tao. It provides a linear criterion for finiteness of the Brascamp--Lieb constant $\mathrm{BL}(\{\pi_j, p_j\}_{j = 1}^m)$. 

\begin{theorem}[\text{\cite[Theorem 2.1]{BCCT}}]\label{thm:BCCT}
    Suppose $\pi_j$'s are surjective. Then the Brascamp--Lieb constant $\mathrm{BL}(\{\pi_j, p_j\}_{j = 1}^m)$ is finite if and only if
    \begin{align}\label{eqn:BCCT exponent}
        n = \sum_{j = 1}^m p_j n_j
    \end{align}
    and 
    \begin{align}\label{eqn:BCCT linear}
        \dim (U) \leq \sum_{j = 1}^m p_j \dim (\pi_j(U)) \text{ for all subspace $U \subseteq V$.}
    \end{align}
\end{theorem}

\begin{proof}
    This is exactly \cite[Theorem 2.1]{BCCT}. We remark that $p_j$'s here are the reciprocals of the exponents in \cite{BCCT}. 
\end{proof}

The following propositions are direct consequences. 

\begin{proposition}
    Let $W$ be a subspace of $V$ of dimension $k$ and let $m$ and $\mathcal{O}_W$ be as in Theorem~\ref{thm:subcritical linear general irrep proj}. Then for all $(h_1, \ldots, h_m) \in \mathcal{O}_W$, the Brascamp--Lieb constant 
    \begin{align*}
        \mathrm{BL}\left(\left\{\pi_{h_j. W}, \frac{n}{km}\right\}_{j = 1}^m\right) < +\infty.
    \end{align*}
\end{proposition}

\begin{proof}
   Let $p_j = \frac{n}{km}$ and $n_j = k$ for all $j = 1, \ldots, m$. The theorem now follows from Theorem~\ref{thm:subcritical linear general irrep proj} and Theorem~\ref{thm:BCCT}. 
\end{proof}

\begin{proposition}\label{prop:BL finite for H action}
    Let $W$ be a subspace of $V$ of dimension $k$ and let $m$ and $\mathcal{O}_W$ be as in Theorem~\ref{thm:subcritical linear general irrep proj}. Then for all $(h_1, \ldots, h_m) \in (\mathcal{O}_W)^t$, the Brascamp--Lieb constant 
    \begin{align*}
        \mathrm{BL}\left(\left\{\pi_{W} \circ h_j, \frac{n}{km}\right\}_{j = 1}^m\right) < +\infty.
    \end{align*}
\end{proposition}

\begin{proof}
   Let $p_j = \frac{n}{km}$ and $n_j = k$ for all $j = 1, \ldots, m$. The theorem now follows from Lemma~\ref{lem:equivalence of two type of proj}, Theorem~\ref{thm:subcritical linear general irrep proj} and Theorem~\ref{thm:BCCT}.  
\end{proof}

Using the Bruhat decomposition, we have the following proposition. 

\begin{proposition}
    For all $\mu$, there exist positive integer $m \ll_{n} 1$ and a Zariski open dense subset $\mathcal{O}_{\mu}' \subseteq (U^+)^{m}$ so that for all $(u_1, \ldots, u_m) \in \mathcal{O}_{\mu}'$, the Brascamp--Lieb constant
    \begin{align*}
        \mathrm{BL}\left(\left\{\pi_{u_j^t.V^{(\mu)}}, \frac{n}{m \dim V^{(\mu)}}\right\}_{j = 1}^m\right) < +\infty.
    \end{align*}
\end{proposition}

\begin{proof}
    Let $P = P(a_t) = \{h \in H: a_{-t} h a_t \text{ is bounded when } t \to +\infty\}$. This is a parabolic subgroup of $H$. There exists a Zariski open dense subset of $H$ in $U^- P$.
    
    Applying Theorem~\ref{thm:subcritical linear general irrep proj} to subspace $V^{(\mu)}$, there exists a positive integer $m \ll_n 1$ and a Zariski open dense subset $\tilde{\mathcal{O}}_{\mu} \subseteq H^m$ so that for all $(h_1, \ldots, h_m) \in \tilde{\mathcal{O}}_{\mu}$ and all subspaces $W' \subseteq V$ we have
    \begin{align*}
        \dim W' \leq \sum_{j = 1}^m \frac{n}{m\dim V^{(\mu)}} \dim \pi_{h_j.V^{(\mu)}}(W').
    \end{align*}
    Therefore, there exists a Zariski open dense subset $\tilde{\mathcal{O}}_{\mu}' \subseteq (U^-P)^m$ so that for all $(u_j^t p_j)_{j = 1}^m \in \tilde{\mathcal{O}}_{\mu}'$ and all subspaces $W' \subseteq V$ we have
    \begin{align*}
        \dim W' \leq \sum_{j = 1}^m \frac{n}{m\dim V^{(\mu)}} \dim \pi_{u_j^t.V^{(\mu)}}(W').
    \end{align*}
    Pullback to $U \times P$ and project to the first factor, we obtain the Zariski open dense subset $\mathcal{O}_{\mu}' \subseteq (U^+)^{m}$ satisfying the conditions in the proposition. 
\end{proof}

By Lemma~\ref{lem:equivalence of two type of proj}, we have the following proposition. 

\begin{proposition}\label{prop:BL finite for U}
    For all $\mu$, there exist positive integer $m \ll_{n} 1$ and a Zariski open dense subset $\mathcal{O}_{\mu}' \subseteq (U^+)^{m}$ so that for all $(u_1, \ldots, u_m) \in \mathcal{O}_{W}'$, the Brascamp--Lieb constant
    \begin{align*}
        \mathrm{BL}\left(\left\{\pi^{(\mu)}_{u_j}, \frac{n}{m\dim V^{(\mu)}}\right\}_{j = 1}^m\right) < +\infty.
    \end{align*}
\end{proposition}

\begin{proof}
    The proof is exactly the same as the proof of Proposition~\ref{prop:BL finite for H action} using  Lemma~\ref{lem:equivalence of two type of proj} and the previous proposition.  
\end{proof}

We also need an explicit estimate on the the Brascamp--Lieb constant $\mathrm{BL}(\{\pi_j, p_j\}_{j = 1}^m)$. It was shown by Lieb \cite{Lie90} (see also \cite{BL76}) that any Brascamp--Lieb inequality has an extremizing sequence of Gaussian. Using this, Gressman obtained the following explicit expression of Brascamp--Lieb constant in \cite{Gre21}. We remark that this expression is naturally related to geometric invariant theory. 

\begin{lemma}[\text{\cite[Lemma 1]{Gre21}}]\label{lem:Lieb constant}
    Suppose the exponents $\{p_j\}_{j = 1}^m$ and dimensions $\{n_j\}_{j = 1}^m$ satisfies
    \begin{align}
        \sum_{j = 1}^m \frac{p_j n_j}{n} = 1.
    \end{align}
    Then the Brascamp--Lieb constant $\mathrm{BL}(\{\pi_j, p_j\}_{j = 1}^m)$ satisfies
    \begin{align}
        [\mathrm{BL}(\{\pi_j, p_j\}_{j = 1}^m)]^{-1} = \inf_{\substack{A \in \SL_n(\R), \mathstrut \\A_j \in \SL_{n_j}(\R), j = 1, \ldots, m}} \prod_{j = 1}^m n_j^{-\frac{p_j n_j}{2}} \| A_j \pi_j A^t\|_{\mathrm{HS}}^{p_j n_j},
    \end{align}
    where $\| \cdot \|_{\mathrm{HS}}$ denotes the Hilbert--Schmidt norm computed with respect to the standard basis of $\R^n$ and $\R^{n_j}$'s and $A^t$ is the transpose of $A$. 
\end{lemma}

Using the above lemma, Gressman obtained a polynomial bound (on entries of $\{\pi_j\}_{j = 1}^m$) for Brascamp--Lieb constant $\mathrm{BL}(\{\pi_j\}_{j = 1}^m, \{p_j\}_{j = 1}^m)$ using ingredients from geometric invariant theory. As a consequence of the results and arguments in \cite[Section 3]{Gre21}, we have the following propositions. 

\begin{proposition}\label{prop:BL constant for H action}
    For all subspace $W \subseteq V$ of dimension $k$, there exist a positive integer $m \ll_{n} 1$, a Zariski open dense subset $\mathcal{O}_{W}' \subseteq (H)^{m}$, a non-zero polynomial $\Phi$ on $H^{m}$ with degree $\ll_{V} 1$ and a positive integer $s_{\Phi} \ll_{V} 1$ so that for all $(h_1, \ldots, h_m) \in \mathcal{O}_{W}'$, the Brascamp--Lieb constant
    \begin{align*}
        \mathrm{BL}\left(\left\{\pi_{W} \circ h_j, \frac{n}{km}\right\}_{j = 1}^m\right) \ll_{n} \|\Phi\|_{\infty}^{\frac{1}{s_{\Phi}}} |\Phi(h_1, \ldots, h_m)|^{-\frac{1}{s_{\Phi}}}.
    \end{align*}
    The norm $\|\Phi\|_{\infty}$ is the maximum of absolute value of coefficient in $\Phi$. 
\end{proposition}

\begin{proposition}\label{prop:BL constant for unipotent}
    For all $\mu$, there exist positive integer $m \ll_{n} 1$, a Zariski open dense subset $\mathcal{O}_{\mu}' \subseteq (U^+)^{m}$, a non-zero polynomial $\Phi$ on $(U^+)^{m}$ with degree $\ll_{V} 1$ and a positive integer $s_{\Phi} \ll_{n} 1$ so that for all $(u_1, \ldots, u_m) \in \mathcal{O}_{\mu}'$, the Brascamp--Lieb constant
    \begin{align*}
        \mathrm{BL}\left(\left\{\pi^{(\mu)}_{u_j}, \frac{n}{m\dim V^{(\mu)}}\right\}_{j = 1}^m\right) \ll_{n} \|\Phi\|_{\infty}^{\frac{1}{s_{\Phi}}} |\Phi(u_1, \ldots, u_m)|^{-\frac{1}{s_{\Phi}}}.
    \end{align*}
    The norm $\|\Phi\|_{\infty}$ is the maximum of absolute value of coefficient in $\Phi$. 
\end{proposition}

Proposition~\ref{prop:BL constant for H action} and Proposition~\ref{prop:BL constant for unipotent} follows from \cite[Section 3]{Gre21} (more precisely  \cite[Lemma 2, Proposition 3, Lemma 3]{Gre21}) and Remez's inequality. The only implicit part in \cite{Gre21} is the bound for $\deg\Phi$ and $s_{\Phi}$. For reader's convenience, we provide detailed proofs in the Appendix~\ref{app:BL}. 

\section{Proof of Theorem~\ref{thm:subcritical entropy general irrep}}\label{sec:translate to entropy}
With all the ingredients developed in previous sections, we now proceed the detail of proof of Theorem~\ref{thm:subcritical entropy general irrep} in this section. Again, due to the polynomial nature of actions of unipotent groups, we need an estimate on the size of the set where
certain polynomial function is small. This is known as Remez's inequality and is used by Kleinbock and Margulis and later Kleinbock and Tomanov in \cite{KM98,KT07} to verify the '$(C, \alpha)$-good' property. The precise version we need has been recorded in Lemma~\ref{lem:RemezKT}. 

\begin{proof}[Proof of Theorem~\ref{thm:subcritical entropy general irrep}]
   For simplicity, let $k = \dim V^{(\mu)}$. Recall that we use $A^{(\delta)}$ to denote the $\delta$-neighborhood of $A$. 
   
   Let $m \ll_n 1$, $\mathcal{O}_{\lambda}' \subseteq (U^+)^{m}$, $\Phi$ and positive integer $s_{\Phi} \ll_n 1$ be as in Proposition~\ref{prop:BL constant for unipotent}. Recall that we have $\deg \Phi \ll_V 1$. Let $p_j = \frac{n}{mk}$ for all $j = 1, \ldots, m$. 
   
   Let $M_1$ be a large positive integer so that $M_1 > m(m + 2) \deg \Phi$. Since $m \ll_n 1$ and $\deg \Phi \ll_V 1$, the integer $M_1$ can be chosen to depend only on $V$. Let $M = 2M_1$. 
   Let $\epsilon \in (0, \frac{1}{100})$ and $0 < \delta \ll_{\epsilon, V} 1$ be a small enough number depending only on $\epsilon$ and the representation $V$. We will explicate this dependence later in the proof. 
   
   Suppose the theorem does not hold for $M$, $\epsilon$ and $\delta$ as above. Then there exists $A \subseteq B_1^V$ so that $m_U(\mathcal{E}(A)) > \delta^\epsilon$. 
   
   Let 
   \begin{align*}
       \mathcal{E}_{\mathrm{BL}} = \{(u_1, \ldots, u_m) \in (B_1^U)^m: |\Phi(u_1, \ldots, u_m)| < \|\Phi\|_{\infty} \delta^{M_1\epsilon}\}.
   \end{align*}
   Since $m \ll_n 1$ and $\deg \Phi \ll_{n} 1$, by Remez's inequality (Lemma~\ref{lem:RemezKT}), there exists constant $C$ depending only on $\dim U$ and $n$ so that
   \begin{align}
       m_U^{\otimes m}(\mathcal{E}_{\mathrm{BL}}) \leq C\delta^{\frac{M_1\epsilon}{m\deg \Phi}} \leq C\delta^{(m + 2)\epsilon} \leq \delta^{(m + 1)\epsilon}
   \end{align}
   if $\delta$ is small enough depending on $\epsilon$ with $\delta^{-\epsilon} \geq C$. 
   
   Let $\tilde{\mathcal{E}} = \mathcal{E}(A)^m \setminus \mathcal{E}_{\mathrm{BL}}$, it has measure $\geq \delta^{m\epsilon} -  \delta^{(m + 1)\epsilon}$. If $\delta$ is small enough depending only on $\epsilon$, we have $\tilde{\mathcal{E}} \neq \emptyset$. For all $(u_1, \ldots, u_m) \in \tilde{\mathcal{E}}$, we have
    \begin{align*}
        \mathrm{BL}(\{\pi_{u_j}^{(\lambda)}, p_j\}_{j = 1}^m) \ll_n \left(\|\Phi\|_{\infty} |\Phi(u_1, \ldots, u_m)|^{-1}\right)^{\frac{1}{s_{\Phi}}} \leq \delta^{-\frac{M_1}{s_{\Phi}}\epsilon} \leq \delta^{-M_1\epsilon}.
    \end{align*}
    
    Let $f_j = \mathds{1}_{\pi^{(\lambda)}_{u_j}(A^{(\delta)})} \in L^1(V^{(\mu)})$, note that
    \begin{align*}
        \int_{\R^n} \prod_{j = 1}^m f_j(\pi_{u_j}^{(\lambda)}(x))^{p_j} \,\mathrm{d}x \geq \Leb(A^{(\delta)}) \gg_n \delta^n |A|_{\delta}.
    \end{align*}
    By the definition of the Brascamp--Lieb constant, we have
    \begin{align*}
        \delta^n |A|_\delta \ll \int_{\R^n} \prod_{j = 1}^m f_j(\pi_{u_j}^{(\lambda)}(x))^{p_j} \,\mathrm{d}x \leq{}& \mathrm{BL}(\{\pi_j, p_j\}_{j = 1}^m) \prod_{j = 1}^m \left(\int_{V^{(\mu)}} f_j\right)^{p_j}\\
        \ll_n{}& \delta^{-M_1 \epsilon} \prod_{j = 1}^m \left(\delta^{k}|\pi_{u_j}^{(\lambda)}(A)|_{\delta}\right)^{\frac{n}{km}}.
    \end{align*}
    Therefore, 
    \begin{align}\label{eqn:average from BL}
        \left(\prod_{j = 1}^m|\pi_{u_j}^{(\lambda)}(A)|_{\delta}^{\frac{n}{k}}\right)^{\frac{1}{m}} \gg_n \delta^{M_1\epsilon} |A|_\delta.
    \end{align}
    For all $(u_1, \ldots, u_m) \in \tilde{\mathcal{E}} = \mathcal{E}(A)^m \setminus \mathcal{E}_{\mathrm{BL}}$, we have from definition of $\mathcal{E}(A)$ that
    \begin{align}\label{eqn:in exceptional set}
        |\pi_{u_j}^{(\lambda)}(A)|_{\delta} < \delta^{M\epsilon}|A|_{\delta}^{\frac{k}{n}}
    \end{align}
    for all $j = 1, \ldots, m$. 
    Combining \cref{eqn:average from BL} and \cref{eqn:in exceptional set}, we have
    \begin{align*}
        \delta^{M_1\epsilon} \ll_n \delta^{M m\epsilon} = \delta^{2M_1\epsilon},
    \end{align*}
    which leads to a contradiction if $\delta$ is small enough depending on $\epsilon$ and $n$. The proof of the theorem is complete. 
\end{proof}

\part{Proof of polynomially effective equidistribution theorem}\label{part:phase three plus proof of main}
As indicated in the introduction, this part is devoted to phase~(3) (from large dimension to polynomially effective equidistribution) and a sketch of an adaption for the framework of \cite{LMWY25}. Some details are left to Appendix~\ref{app:decomposition initial}. We collect needed spectral input and discuss phase~(3) in Section~\ref{sec:Venkatesh}. A preparation on sheeted set, admissible measure and Margulis function is done in Section~\ref{sec:boxes}. In Section~\ref{sec:proof of main}, we combine results in Part~\ref{part:closinglemma} and \ref{part:projection} to prove the main theorem. 

\section{Mixing and effective equidistribution}\label{sec:Venkatesh}
This section is devoted to phase~(3) as indicated in the introduction. We utilize spectral gap property of the space $X = G/\Gamma$. The main result is Lemma~\ref{lem:Venkatesh}. 

\subsection{Decay of matrix coefficients}
We show the following results on decay of matrix coefficient on $X$ in this subsection. 

Let $d_i$ be the Riemannian metric on each simple factor of $H$ defined using the Killing form and the Cartan involution defined in Section~\ref{sec:Global prelim}. There exists $\kappa_0 \in (0, 1)$ so that for all $h \in H$
\begin{align}\label{eqn:decay of matrix coefficient}
    \Biggl|\int_X \varphi(h.x)\psi(x) \,\mathrm{d}\mu_X(x) - \int_X \varphi \,\mathrm{d}\mu_X \int_X \psi \,\mathrm{d}\mu_X \Biggr| \ll \mathcal{S}(\varphi) \mathcal{S}(\psi) e^{-\kappa_0 \min_{i} d_i(e, h)}
\end{align}
for all $\varphi, \psi \in \mathrm{C}^\infty_c(X) + \C \mathds{1}_X$. Here $\mathcal{S}(\cdot)$ is certain Sobolev norm on $\mathrm{C}_c^\infty(X) + \C \mathds{1}_X$ so that it dominates $\|\cdot\|_{\infty}$ and the Lipschitz norm $\|\cdot\|_{\Lip}$.

\subsubsection{The ambient group $\mathbf{G}$ is semisimple}
In the case where $\mathbf{G}$ is semisimple, such result is well-known, see e.g. \cite[Section 2.4]{KM96}. Using Mostow's theorem \cite{Mos55}, it suffices to show the following result. 

\begin{lemma}
    Suppose $\mathbf{G}$ is semisimple and the complement $\mathfrak{r}$ of $\LieH$ is a nontrivial irreducible representation of $\LieH$. Suppose $\mathbf{x} \in \LieH$ is not contained in any proper ideal of $\LieH$, then it is not contained in any proper ideal of $\LieG$
\end{lemma}

\begin{proof}
    Let $\mathfrak{I} \triangleleft \LieG$ be a proper ideal of $\LieG$. If $\mathfrak{I} \subseteq \LieH$, it is a proper ideal of $\LieH$ and the lemma follows directly. If not, note that $\mathfrak{I}$ is also an $\LieH$-module and $\LieG = \LieH + \mathfrak{I}$ due to irreducibility of $\mathfrak{r}$. 
    Since $\mathbf{x}$ acts non-trivially on $\mathfrak{r}$, $[\mathbf{x}, \mathfrak{I}] \neq \{0\}$ and this proves the lemma. 
\end{proof}

\subsubsection{The ambient group $\mathbf{G}$ is isogenous to semi-direct product}
In this case we reduce the estimate on decay of matrix coefficients to its semisimple base and corresponding orthogonal complement. 

We first separate the part of $H$ acting trivially on $\mathfrak{r} = \rad(\LieG)$. Note that there exists a Levi factor $\mathbf{H}'$ over $\Q$ so that $\mathbf{G}$ is isogenous to $\mathbf{H}' \ltimes \mathfrak{r}$ over $\Q$. Moreover, there exists $g \in \mathbf{G}(\R)$ so that $\mathbf{H}(\R) = g \mathbf{H}'(\R) g^{-1}$. Let $\mathbf{H}_1'$ be the normal $\Q$-subgroup of $\mathbf{H}'$ consisting of elements acting trivially on $\mathfrak{r}$. There exists a normal $\Q$-subgroup $\mathbf{H}_2' \triangleleft \mathbf{H}'$ so that $\mathbf{H}' = \mathbf{H}_1' \mathbf{H}_2'$ and $\mathbf{H}_1' \cap \mathbf{H}_2'$ is finite. The estimate of correlation in the $\mathbf{H}_1'$ factor follows from \cite{KM96}. Let $H_2 = g^{-1} \mathbf{H}_2'(\R) g$, we have $H_2 \ltimes \mathfrak{r} = \mathbf{H}_2'(\R) \ltimes \mathfrak{r}$ and $H_2$ acting non-trivially on $\mathfrak{r}$. It suffices to estimate the decay of matrix coefficients for $L_0^2(H_2 \ltimes \mathfrak{r}/(H_2 \ltimes \mathfrak{r} \cap \Gamma))$ (the orthogonal complement of constant functions) as representation of $H_2$. 

We decompose $L_0^2(H_2 \ltimes \mathfrak{r}/((H_2 \ltimes \mathfrak{r}) \cap \Gamma))$ as $\mathfrak{r}$-fixed vectors and its orthogonal complement. The first representation is
isomorphic to the regular representation of $\mathbf{H}_2'(\R)$ on $L^2_0(\mathbf{H}_2'(\R)/(\mathbf{H}_2'(\R) \cap \SL_N(\Z)))$, which is isolated
from the trivial representation. The estimate of the matrix coefficient in the latter follows from \cite[Theorem 1.4]{Wang14}, see also \cite[Theorem 4.5]{KW19} and the references there for certain cases. \cref{eqn:decay of matrix coefficient} now follows from combining the above estimates. 

\subsection{From large dimension to effective equidistribution}
\begin{lemma}\label{lem:Venkatesh}
Suppose the data $(G, \Gamma, H, \mathfrak{r}, \mathbf{a}, U)$ satisfy {}\Irrep{}. Then there exists absolute $\varrho_0 \in (0, 1)$ so that the following holds. Let $\delta_0 \in (0, 1)$. Let $\ell_1, \ell_2 > 0$ with $\kappa_1\ell_2 \geq \max\{\ell_1, |\log \eta|\}$ and $\ell_2 \ll |\log \delta_0|$, and let $\varrho \in (0, \varrho_0]$. Let $\mu$ be a probability measure on $B^{\mathfrak{r}}_{\varrho}(0)$ satisfying
    \begin{align*}
        \mu(B_\delta^{\mathfrak{r}}(w)) \leq \Upsilon \delta^{\dim\mathfrak{r}} \quad \forall w \in \mathfrak{r}, \delta \geq \delta_0.
    \end{align*}
    Then for all $\phi \in \mathrm{C}^\infty_c(X) + \C\mathds{1}_X$ and all $x \in X_\eta$, we have
    \begin{align*}
        &\int_{\mathsf{B}_1^U} \int_{\mathsf{B}_1^U}\int_{\mathfrak{r}} \phi(a_{\ell_1}u_1 a_{\ell_2} u_2 \exp(w).x)\,\mathrm{d}\mu(w)\,\mathrm{d}m_U(u_2)\,\mathrm{d}m_U(u_1)\\
        ={}& \int_X \phi \,\mathrm{d}\mu_X + O(\mathcal{S}(\phi)(\varrho^\star + \eta + \Upsilon^{\star}\varrho^{-\star}e^{-\kappa_1\ell_1})).
    \end{align*}
\end{lemma}

\begin{proof}
    The statement can be proved by following the proof of \cite[Lemma 5.2]{LMWY25} step-by-step, see also \cite[Theorem 11.1]{OS25} for a variant when $G$ is simple. 
\end{proof}

\section{Preparation III: Sheeted set and Margulis function}\label{sec:boxes}
We recall notions and results on sheeted set, admissible measures and Margulis function in \cite{LMW22,LMWY25}. 

\subsection{Sheeted set and admissible measure}\label{subsec:sheeted set admissible measure}
Suppose $\eta \leq \frac{1}{10^5C_0}\eta_0$ is a small parameter where $\eta_0$ and $C_0$ are from Lemma~\ref{lem:BCH}. Let $\beta$ be a another small parameter satisfying $\beta \leq \eta^{d_0}$ where $d_0$ is from Theorem~\ref{thm:Closing lemma many scale}. We enlarge $d_0$ if needed so that $d_0 \geq \dim G$. We set
\begin{align*}
    \mathsf{E}_{\beta, \eta} = \mathsf{B}_{\beta}^{U^-} \mathsf{B}_\beta^{0} \mathsf{B}_\eta^{U^+}.
\end{align*}

A subset $\mathcal{E} \subseteq X$ is called a $(\beta, \eta)$-\emph{sheeted set} if there exists a base point $y \in X_\eta$ and a finite set of transverse cross-section $F \subset B_{\eta}^{\mathfrak{r}}$ so that the map $(h, w) \mapsto h\exp(w).y$ is injective on $\mathsf{E}_{\beta, \eta} \times B_{\eta}^{\mathfrak{r}}$ and 
\begin{align*}
    \mathcal{E} = \bigsqcup_{w \in F} \mathsf{E}_{\beta, \eta}\exp(w).y.
\end{align*}

We now define admissible measure. Informally, an admissible measure $\mu_{\mathcal{E}}$ associated to a sheeted set $\mathcal{E}$ is a probability measure on $\mathcal{E}$ that is equivalent to Haar measure of $H$ on each sheet. Moreover, each sheet is assigned with roughly equal weight. 

A probability measure $\mu_{\mathcal{E}}$ on a $(\beta, \eta)$-sheeted set $\mathcal{E}$ is called $(\mathsf{K}, \Lambda)$-admissible if 
\begin{align*}
    \mu_\mathcal{E} = \frac{1}{\sum_{w \in F}\mu_w(X)} \sum_{w \in F} \mu_w
\end{align*}
where $\mu_w$ are measures on $\mathsf{E}_{\beta, \eta}\exp(w).y$ with the following properties. For all $w \in F$, there exists a function $\varrho_w$ defined on $\mathsf{E}$ with $\frac{1}{\mathsf{K}} \leq \varrho_w \leq \mathsf{K}$ so that for all $\mathsf{E}' \subseteq \mathsf{E}_{\beta, \eta}$, we have
\begin{align*}
    \mu_w(\mathsf{E}'\exp(w).y) = \int_{\mathsf{E}'} \varrho_w(\mathsf{h}) \,\mathrm{d}m_H(\mathsf{h}).
\end{align*}
Moreover, there exists $\mathsf{E}_w = \cup_{i = 1}^{\mathsf{K}}\mathsf{E}_{w, i} \subseteq \mathsf{E}$ so that 
\begin{enumerate}
    \item $\mu_w\bigl((\mathsf{E}\setminus\mathsf{E}_{w})\exp(w).y\bigr) \leq \mathsf{K}\eta \mu_w(X)$,
    \item the complexity (see \ref{app:decomp boxes complexity} or \cite[Section 7.2]{LMW22}) of $\mathsf{E}_{w, i}$ is bounded by $\mathsf{K}$ for all $i$, and
    \item $\Lip(\varrho_w|_{\mathsf{E}_{w, i}}) \leq \Lambda$.
\end{enumerate}
This is a slightly more refined version of the notion of admissible measure in \cite{LMWY25}, for which we separate the bounds. In later application, $\mathsf{K}$ will always be an absolute constant. 

We record the following theorem. 
\begin{theorem}\label{thm:offspring admissible}
    Suppose $\mathcal{E}_{\mathrm{ini}}$ is a $(\beta_{\mathrm{ini}}, \eta_{\mathrm{ini}})$-sheeted set equipped with an $(\mathsf{K}, \Lambda)$-admissible measure $\mu_{\mathcal{E}_{\mathrm{old}}}$. Suppose $e^{-\ell} < \beta_{\mathrm{nw}} \leq   \beta_{\mathrm{ini}}$ and $\eta_{\mathrm{nw}} \leq \eta_{\mathrm{ini}}$ with $\beta_{\boldsymbol{\cdot}} \leq \eta_{\boldsymbol{\cdot}}^{d_0}$ for $\boldsymbol{\cdot} = \mathrm{ini}, \mathrm{nw}$. 

    Then for all $u \in \mathsf{B}_1^U$, there exists a family $\mathcal{F}$ of $(\beta_{\mathrm{nw}}, \eta_{\mathrm{nw}})$-sheeted set equipped with $(O(\mathsf{K}), O(\Lambda))$-admissible measures $\{\mu_{\mathcal{E}}: \mathcal{E} \in \mathcal{F}\}$ and $c_{\mathcal{E}} > 0$ with $\sum_{\mathcal{E}} c_{\mathcal{E}} = 1$ so that for all $u' \in \mathsf{B}_1^U$ and $d' \geq 0$, we have
    \begin{align*}
        \int_{\mathcal{E}_{\mathrm{ini}}} \phi(a_{d'}u' a_{\ell}u.z) \,\mathrm{d}\mu_{\mathcal{E}_{\mathrm{ini}}}(z) = \sum_{\mathcal{E} \in \mathcal{F}} c_{\mathcal{E}} \int_{\mathcal{E}} \phi(a_{d'}u' a_{\ell}u.z) \,\mathrm{d}\mu_{\mathcal{E}}(z)\\
        + O(\mathsf{K}^\star( \Lambda^\star\eta_{\mathrm{nw}}^\star + \eta_{\mathrm{ini}}^\star) \|\phi\|_{\Lip}).
    \end{align*}
\end{theorem}

\begin{proof}
    This is \cite[Lemma D.4]{LMWY25}. The only difference is the initial sheeted set can be taken as a larger set and we distinguish the bounds in admissible measure. However, the arguments are mutatis mutandis. For the estimate on constants for the new admissible measures, we refer to \cite[Lemma 8.8]{LMW22}. 
\end{proof}
\subsection{Dimension, energy and Margulis function}\label{subsec:Margulis function}
For a finite set $F \subset \mathfrak{r}$, we set $\mu_F$ be the normalized counting measure on $F$. We say that the set $F$ has dimension $\geq \alpha$ for scales larger than $\delta$ if there exists $C > 1$ so that 
\begin{align*}
    \mu_F(B(x, r)) \leq Cr^\alpha \quad \forall x \in \mathfrak{r} \text{ and }r \geq \delta.
\end{align*}
In literatures, this is always denoted by $(C, \alpha)$-Frostman-type condition or $(C, \alpha)$-nonconcentration condition. Recall the (modified) $\alpha$-energy of $F$ as follows. 
\begin{align*}
    \mathcal{G}^{(\alpha)}_{F, \delta}(w) = \sum_{w' \in F: w' \neq w} \max\{\|w' - w\|, \delta\}^{-\alpha}.
\end{align*}

We recall the notion of (modified) Margulis function in \cite[Section 7]{LMWY25}. Suppose $\mathcal{E}$ is a sheeted set. For all $z \in \mathcal{E}$, let
\begin{align*}
    I_{\mathcal{E}}(z) = \{w \in \mathfrak{r}: \|w\| < \inj(z), \exp(w).z \in \mathcal{E}\}.
\end{align*}
For every $0 < \delta < 1$ and $0 < \alpha < \dim \mathfrak{r}$, we define the (modified) Margulis function as follows. 
\begin{align*}
    f^{(\alpha)}_{\mathcal{E}, \delta}(z) = \sum_{w \in I_\mathcal{E}(z) \setminus \{0\}} \max\{\|w\|, \delta\}^{-\alpha}.
\end{align*}

We have the following standard connection between those notions, see \cite[Lemma 9.2]{LMW22}. 
\begin{proposition}\label{pro:frostman energy Margulis function}
    Suppose $F \subset B^{\mathfrak{r}}_1$ is a finite set and suppose $\mathcal{E} = \mathsf{E} \exp(F).y$ is a sheeted set. We have the following properties. 
    \begin{enumerate}
        \item Suppose $F$ is a set of dimension $\geq \alpha$ for scales larger than $\delta$, then for all $w \in F$ and $0 < \beta < \alpha$, 
        \begin{align*}
            \mathcal{G}_{F, \delta}^{(\beta)}(w) \leq 2^{\dim\mathfrak{r}}C\Bigl(1 + \frac{1}{1 - 2^{\beta - \alpha}}\Bigr) \#F.
        \end{align*}
        \item Suppose for all $w \in F$ we have $\mathcal{G}_{F, \delta}^{(\alpha)}(w) \leq C \#F$, then for all $z \in \mathcal{E}$, we have
        \begin{align*}
            f_{\mathcal{E}, \delta}^{(\alpha)}(z) \ll C \#F.
        \end{align*}
    \end{enumerate}
\end{proposition}

\section{Proof of the main theorem}\label{sec:proof of main}
As indicated in Subsection~\ref{subsection:different form on balls}, Theorem~\ref{thm:main equidistribution} is equivalent to Theorem~\ref{thm:main equidistribution different form}. We prove the latter in this section. 

\subsection{Outline of the adaption}
Let us first give a brief outline of the adaption. The reader can compare Theorem~\ref{thm:Closing lemma many scale} and \cite[Proposition 4.6]{LMWY25} (also the 'Closing lemma and initial separation' paragraph in proof of \cite[Proposition 8.1]{LMWY25}). For the projection theorem, the reader can compare Theorem~\ref{thm:energy Improvement} with \cite[Theorem 6.1]{LMWY25}. The latter can be applied into the framework in \cite{LMWY25} directly. The former requires extra argument to produce sheeted sets and admissible measures. 

We remark that one can also compare Theorem~\ref{thm:Closing lemma many scale} with \cite[Proposition 6.2]{Yan24} and then adapt the argument in \cite{Yan24} for further dimension improvement using Theorem~\ref{thm:ImprovementMain}. 

\subsection{Initial dimension}\label{subsec:initial dim}
The following consequence provides us sheeted sets and admissible measures from Theorem~\ref{thm:Closing lemma many scale}. 

Let $\beta_{\mathrm{ini}} = \eta_{\mathrm{ini}}^{d_0}$ be two small parameter satisfying $10^{10}C_0\eta_{\mathrm{ini}} \leq \eta_0$. Let $\lambda_{\mathrm{ini}}$ be the normalized Haar measure on
\begin{align*}
    \mathsf{B}^{s, H}_{\beta_{\mathrm{ini}} + 100 \beta_{\mathrm{ini}}^{d_0}} = \mathsf{B}_{\beta_{\mathrm{ini}} + 100 \beta_{\mathrm{ini}}^{d_0}}^{U^-} \mathsf{B}_{\beta_{\mathrm{ini}} + 100 \beta_{\mathrm{ini}}^{d_0}}^{0}.
\end{align*}

\begin{theorem}\label{thm:closing lemma initial dim}
Suppose the data $(G, \Gamma, H, \mathfrak{r}, \mathbf{a}, U)$ satisfy {}\Irrep{}. Then there exist absolute constants $\consta\label{a:closing lemma main} > 1$, $\constc\label{c:closing lemma main} > 1$, $\constd\label{d:closing lemma main} > 1$, $\constE\label{e:closing lemma main1}, \constE\label{e:closing lemma main2}> 1$, $\constm\label{m:closing lemma main} > 1$, $\epsilon_0 > 0$, and $\mathsf{L}$ so that the following holds. For parameters $R \gg 1 $ and $\eta_{\mathrm{ini}} \ll \eta_0$ satisfying $R \gg \eta_{\mathrm{ini}}^{-\ref{e:closing lemma main1}}$ and for all $x_1 \in X_\eta$ and all $D \geq \ref{d:closing lemma main} + 1$, let $\delta_0 = R^{-\frac{1}{\ref{a:closing lemma main}}}$, $t \geq M\log R$ where $M = \ref{m:closing lemma main} + \ref{c:closing lemma main}D$ and $\mu_t = \nu_{t} \ast \delta_{x_1}$. 
    
Suppose that for all periodic orbit $H.x'$ with $\vol(H.x') \leq R$, we have
\begin{align*}
    d_X(x_1, x') > R^{-D}.
\end{align*}
Then there exists a family $\mathcal{F}$ of $(\beta_{\mathrm{ini}}, \eta_{\mathrm{ini}})$-sheeted sets with associated $(\mathsf{L}, \Lambda_{\mathrm{ini}})$-admissible measures $\{\mu_{\mathcal{E}}: \mathcal{E} \in \mathcal{F}\}$ so that the following holds. 
\begin{enumerate}
    \item The constant $\mathsf{L}$ is absolute, $\Lambda_{\mathrm{ini}} \ll \eta_{\mathrm{ini}}^{-\star}$. 
    \item There exists $\{c_\mathcal{E}\}$ with $c_\mathcal{E} > 0$ and $\sum_{\mathcal{E}} c_\mathcal{E} = 1$ so that for all $u' \in \mathsf{B}_1^U$, $d' \geq 0$ and all $\phi \in \mathrm{C}^{\infty}_c(X)$
    \begin{align}
        \int_{X} \phi(a_{d'} u' x) \,\mathrm{d}(\lambda \ast \mu_t)(x) = \sum_{\mathcal{E}} c_\mathcal{E}\int_{X} \phi(a_{d'}u'.x) \,\mathrm{d}\mu_{\mathcal{E}}(x) + O(\mathcal{S}(\phi)\beta_{\mathrm{ini}}^\star).
    \end{align}
    \item For all $(\beta_{\mathrm{ini}}, \eta_{\mathrm{ini}})$-sheeted set $\mathcal{E} \in \mathcal{F}$ with cross-section $F \subset B_{\eta_{\mathrm{ini}}}^{\mathfrak{r}}$. 
    The number of sheets satisfies 
    \begin{align}
        \beta_{\mathrm{ini}}^{\star}\delta_0^{-2\epsilon_0} \leq \#F \leq \beta_{\mathrm{ini}}^{-\star}e^{\lambda_{\max}t}.
    \end{align}
    Moreover, we have the Margulis function estimate
    \begin{align}
        f^{(\epsilon_0)}_{\mathcal{E}, \delta_0}(z) \leq  \beta_{\mathrm{ini}}^{-\ref{e:closing lemma main2}}\#F \quad \forall z \in \mathcal{E}.
    \end{align}
\end{enumerate}
\end{theorem}

\begin{proof}[Proof of Theorem~\ref{thm:closing lemma initial dim}]\label{sec:decompose sketch}
We provide outline here and refer to Appendix~\ref{app:decomposition initial} for the details. For simplicity, we drop the subscript $\mathrm{ini}$ in this sketch. 

Recall $\lambda$ is the normalized Haar measure on $\mathsf{B}^{s, H}_{\beta + 100\beta^{d_0}}$. Theorem~\ref{thm:Closing lemma many scale} implies that for all $y \in X_{3\eta}$, $r_H \leq 10^3C_0\eta$, $r \in [\delta_0, \eta^{d_0}]$
\begin{align}\label{eqn:frostman for admissible}
    (\lambda \ast \mu_t)(\mathsf{B}_{r_H}^H \exp(B_{r}^{\mathfrak{r}}).y) \ll r^{\epsilon_1}.
\end{align}

We decompose the measure $\lambda \ast \mu_t$ into local pieces. It is standard to translate \cref{eqn:frostman for admissible} into energy estimate and Proposition~\ref{pro:frostman energy Margulis function} translates the latter into Margulis function estimate. The remaining issue is for the admissibility. In \cite[Proposition 4.8]{LMW22}, it is proved that the map $B_\beta^H a_t \mathsf{B}_1^U \to B_\beta^H a_t \mathsf{B}_1^U a_{8t}u_rx_1$ is injective for most $r \in [0,1]$. This ensures that each $H$-sheet contributes roughly the same amount of measure and hence admissibility. This is not a priori guaranteed by Theorem~\ref{thm:Closing lemma many scale}. Locally $\lambda \ast \mu_t$ might not looks like a renormalized Haar measure. Moreover, locally $\lambda \ast \mu_t$ might assign different weights for different $H$-sheets. 

We resolve the problems in two steps. First, we decompose $\mu_t$ into local pieces of size $\beta^{d_0}$ in stable and central direction and then smear it using $\lambda$ which is roughly of size $\beta$. This ensures that in size $\beta$ ball near the origin, it looks roughly like Haar measure and the boundary contributes only small error. Next, we decompose the measure once again according to the weight on each $H$-sheets. This provides admissible measures at the end. 
\end{proof}

\subsection{Margulis function estimate}
The following proposition provides a general iterative process for improving the dimension. It is the analog of \cite[Lemma 7.2]{LMWY25} in this setting. 
\begin{proposition}\label{prop: integrated inductive step}
Let $\delta > 0$, $\theta > 0$, and $\alpha\in [\epsilon_0,\dim(\mathfrak{r}) - \theta]$. There exists $\epsilon$ depending only on $\epsilon_0$ and $\theta$ so that the following holds. 

Suppose $\ell > 0, \beta > 0$ and $\Upsilon > 0$ satisfies $e^{-\epsilon^2\ell} < \beta^{100}$ and $0 < \Upsilon \leq e^{1/\beta}$. Suppose that $\mathcal{E}$ is a $(\beta, \eta)$-sheeted set with cross-section $F$ so that
\begin{align*}
f_{\mathcal{E},\delta}^{(\alpha)}(x) \leq \Upsilon \qquad \text{for all $x \in \mathcal{E}$}.
\end{align*}
Assume further $\mathcal{E}$ is assigned with an $(\mathsf{K}, \Lambda)$-admissible measure $\mu_{\mathcal{E}}$, see \cref{sec:boxes}.
Then there exists a family of $(\beta, \eta)$-sheeted sets $\{\mathcal{E}_i\}_i$ with cross-sections $\{F_i\}_i$ and associated $(O(\mathsf{K}), O(\Lambda))$-admissible measures $\{\mu_{\mathcal{E}_i}\}_i$ satisfying the following properties.
\begin{enumerate}
    \item For all $\phi \in \mathrm{C}_c^{\infty}(X) + \C \mathds{1}_X$, $\ell' \geq 0$, and $u' \in \mathsf{B}_1^U$, we have
    \begin{align*}
        \int_{\mathsf{B}_1^U} \int_{\mathcal{E}} \phi(a_{\ell'}u' a_{\ell} u.x) \,\mathrm{d}\mu_{\mathcal{E}}(x) \,\mathrm{d}u
        = \sum_i c_{i}\int_{\mathcal{E}_i} \phi( a_{\ell'}u'.z) \,\mathrm{d}\mu_{\mathcal{E}_i}(z) + O_{\mathsf{K}, \mathsf{\Lambda}}(\mathcal{S}(\phi)\beta^\star)
    \end{align*}
    for some $c_i > 0$ with $\sum_i c_i=1$. The dependence on $\mathsf{K}$ and $\Lambda$ are polynomial. 
    \item For all $i$, we have
    \begin{align*}
    \beta^{\star} \#F 
    \leq \#F_i
    \leq e^{\lambda_{\max}\ell}\#F.
    \end{align*}
    \item For all $i$, we have
    \begin{align*}
        f_{\mathcal{E}_i, \delta'}^{(\alpha)}(x) \leq e^{-\epsilon\ell}\Upsilon +  e^{\alpha\mu_{\max}\ell}\beta^{-\alpha}\#F_i \qquad \text{for all $x \in \mathcal{E}_i$}
    \end{align*}
    where $\delta' = e^{\mu_{\max}\ell}\max\{\delta, \#F^{-\frac{1}{\alpha}}\}$. 
\end{enumerate}
\end{proposition}

\begin{proof}
    The statement can be proved following the proof of \cite[Lemma 7.2]{LMWY25} step-by-step and replacing \cite[Theorem 6.1]{LMWY25} by Theorem~\ref{thm:energy Improvement}. 
\end{proof}

\subsection{Proof of Theorem~\ref{thm:main equidistribution different form}}

\begin{proof}[Proof of Theorem~\ref{thm:main equidistribution different form}]
    The statement can be proved by following \cite[Section 8--9]{LMWY25} step-by-step. In particular, replace the base case of \cite[Lemma 8.2]{LMWY25} by Theorem~\ref{thm:closing lemma initial dim}, \cite[Theorem 6.1]{LMWY25} by Theorem~\ref{thm:energy Improvement} and \cite[Lemma 5.2]{LMWY25} by Lemma~\ref{lem:Venkatesh}, the rest of the arguments are mutatis mutandis. We remark the choice of $\beta, \beta_{\mathrm{ini}}, \eta, \eta_{\mathrm{ini}}$ here. The parameters $\beta, \eta$ in the inductive step are more refined than $\beta_{\mathrm{ini}}, \eta_{\mathrm{ini}}$. Arguing as \cite[Proposition 8.1]{LMWY25} using Theorem~\ref{thm:offspring admissible}, \ref{thm:closing lemma initial dim} and Proposition~\ref{prop: integrated inductive step}, the error at the end is $O(\mathcal{S}(\phi)\Lambda_{\mathrm{ini}} \beta^{\star} + \beta_{\mathrm{ini}}^\star)$ where $\beta$ is of size $R^{-\star}$ and $\Lambda_{\mathrm{ini}} \ll \beta_{\mathrm{ini}}^{-\star}$. It suffices to take $\beta_{\mathrm{ini}} = R^{-\frac{1}{M}}$ for some $M$ large enough. 

    We remark that one can also compare Theorem~\ref{thm:Closing lemma many scale} with \cite[Proposition 6.2]{Yan24} and then adapt the argument in \cite{Yan24} for the dimension improvement step using Theorem~\ref{thm:ImprovementMain}. 
\end{proof}

\part{Appendices}
\appendix
\section{Decomposition of the measure with initial dimension}\label{app:decomposition initial}
This section is devoted to the proof of Theorem~\ref{thm:closing lemma initial dim}. For simplicity, we write $d_S = \dim S$ for all closed subgroup $S < G$. Recall from Subsection~\ref{subsec:sheeted set admissible measure}, we fix $d_0 \geq d_G$. 
\subsection{Boxes}\label{sec:boxes more}
We collect needed results in \cite[Section 7]{LMW22} in this subsection. They will be used in the next section for the proof of Theorem~\ref{thm:closing lemma initial dim}. The results there are stated for 
$H = \SL_2(\R)$, but their proof work in far more generality. We will indicate the needed change in the proof. 

\subsubsection{Covering lemma}
Let
\begin{align*}
    \mathsf{Q}^H_{\eta, \beta, m} = a_{m}\mathsf{B}_{\beta}^{s, H} a_{-m} \mathsf{B}_{\eta}^{U^+}, \quad 
    \mathsf{Q}^G_{\eta, \beta, m} = \mathsf{Q}^H_{\eta, \beta, m} \cdot \exp(B_{2\beta}^{\mathfrak{r}}).
\end{align*}
For simplicity in this section we set $\mathsf{Q}^H_m := \mathsf{Q}^H_{\eta_{\mathrm{ini}}, \beta_{\mathrm{ini}}^{d_0}, m}$ and $\mathsf{Q}^G_m := \mathsf{Q}^G_{\eta_{\mathrm{ini}}, \beta_{\mathrm{ini}}^{d_0}, m}$. 

We also introduce the notion
\begin{align*}
    \Check{\mathsf{Q}}^H_m = (\mathsf{Q}^H_m)^{-1}, \quad
    \Check{\mathsf{Q}}^G_m = (\mathsf{Q}^H_m)^{-1} \exp(B_{2\beta_{\mathrm{ini}}^d}^{\mathfrak{r}}).
\end{align*}

\begin{lemma}[\text{\cite[Lemma 7.1]{LMW22}}]\label{lem: covering multiplicity}
    There exists $K \geq 1$ depends only on $X$ so that for all $m \geq 0$, there is a covering 
    \begin{align*}
        \{\mathsf{Q}^G_{\eta, \beta, m}.y_j: j \in \mathcal{J}_m, y_j \in X_{\frac{3}{2}\eta}\}
    \end{align*}
    of $X_{2\eta}$ with multiplicity $\leq K$. In particular, $\#\mathcal{J}_m \ll \eta^{-d_U}\beta^{-d_0(d_0 - d_U)}e^{m \lambda_{\vol}}$. 
\end{lemma}

Similarly, we have the following lemma. 
\begin{lemma}\label{lem: inverse covering initial}
    There exists $K \geq 1$ depends only on $X$ so that for all $m \geq 0$, there is a covering 
    \begin{align*}
        \{\Check{\mathsf{Q}}^G_{\eta, \beta, m}.y_j: j \in \Check{\mathcal{J}}_m, y_j \in X_{\frac{3}{2}\eta}\}
    \end{align*}
    of $X_{2\eta}$ with multiplicity $\leq K$. In particular, $\#\Check{\mathcal{J}}_m \ll \eta^{-d_U}\beta^{-d_0(d_0 - d_U)}e^{m \lambda_{\vol}}$. 
\end{lemma}

From now in this paper, we fix such cover $\{\Check{\mathsf{Q}}^G_{0}.y_j\}_{j \in \Check{J}_0}$ as in Lemma~\ref{lem: inverse covering initial}. Let $\Check{\mathsf{k}}_0(z) := \#\{j \in \Check{\mathcal{J}}_0: z \in \Check{\mathsf{Q}}^G_{0}.y_j\}$, then $1 \leq \Check{\mathsf{k}}_0(z) \leq K$. Define $\Check{\rho}_0(z) := \frac{1}{\Check{\mathsf{k}}_0(z)}$ and $\Check{\rho}_{0, j} := \Check{\rho}_0 \cdot \mathds{1}_{\Check{\mathsf{Q}}^G_{0}.y_j}$. We have
\begin{align*}
    1/K \leq \Check{\rho}_{0, j} \leq 1, \quad \sum_{j \in \Check{\mathcal{J}}_0} \Check{\rho}_{0, j}(z) = 1 \quad \forall z \in X_{2\eta_{\mathrm{ini}}}.
\end{align*}

\subsubsection{Boxes and complexity}\label{app:decomp boxes complexity}
Let $\mathsf{prd}: \LieH \to H$ be the map defined by
\begin{align*}
    \mathsf{prd}: \LieH = \LieU^- \oplus \LieH_0 \oplus \LieU^+ {}&\to H\\
    (X_{\LieU^-}, X_{\LieH_0}, X_{\LieU^+}) {}&\mapsto \exp(X_{\LieU^-})\exp(X_{\LieH_0})\exp(X_{\LieU^+}).
\end{align*}
A subset $\mathsf{D} \subseteq H$ is called a \textit{box} if there exist boxes $\mathsf{B}_{\LieU^-} \subset \LieU^-$, $\mathsf{B}_{\LieH_0} \subset \LieH_0$, and $\mathsf{B}_{\LieU^+} \subset \LieU^+$ so that 
\begin{align*}
    \mathsf{D} = \mathsf{prd}(\mathsf{B}_{\LieU^-} \times \mathsf{B}_{\LieH_0} \times \mathsf{B}_{\LieU^+}).
\end{align*}

We say that a subset $\Xi \subset H$ has complexity bounded by $L$ (or at most $L$) if $\Xi$ can be written as union of at most $L$ boxes. 
We make the convention that the empty set is a box so that all sets of complexity at most $L$ can be written as $\Xi = \cup_{i = 1}^L \Xi_i$ where $\Xi_i$'s are boxes. 

Let $\Check{\mathsf{prd}}: \LieH \to H$ be the map defined by
\begin{align*}
    \Check{\mathsf{prd}}: \LieH = \LieU^+ \oplus \LieH_0 \oplus \LieU^- {}&\to H\\
    (X_{\LieU^+}, X_{\LieH_0}, X_{\LieU^-}) {}&\mapsto \exp(X_{\LieU^+})\exp(X_{\LieH_0})\exp(X_{\LieU^-}).
\end{align*}

\subsection{Construction of sheeted sets}\label{sec:decompose initial detail}
We proceed the detail of the proof of Theorem~\ref{thm:closing lemma initial dim} according to the outline in Section~\ref{sec:decompose sketch}. We take $\ref{m:closing lemma main}$ large enough so that $e^{-\lambda_{\min}t/10^5} \leq \eta_{\mathrm{ini}}$. 

\subsubsection{Non-divergence result}
The following result assert that the trajectory is away from cusp most of the time. 

\begin{proposition}\label{pro:Non divergence}
    There exists $\mathsf{m} > 0$, $\kappa > 0$ and $C \geq 1$ absolute with the following property. Let $0 < \delta, \eta < 1$ and let $B \subseteq \mathsf{B}_{10}^{U}$ be an open ball with radius $\geq \delta$. For all $x \in X$ and $t \geq \mathsf{m} |\log (\delta\inj(x))| + C$, we have
    \begin{align*}
        m_U(\{u \in B: a_t u.x \notin X_{\eta} \}) \leq C\eta^{\frac{1}{\mathsf{m}}} m_U(B).
    \end{align*}
\end{proposition}
\begin{proof}
    It follows from \cite[Proposition 26, Theorem 16]{SS24} and Chebyshev inequality. See also \cite[Proposition 4.2]{LMWY25}. 
\end{proof}

\subsubsection{Proof of Theorem~\ref{thm:closing lemma initial dim}}We now proceed the proof. As in the sketch in Subsection~\ref{subsec:initial dim}, we drop the subscript $\mathrm{ini}$ avoid complicated notations. \emph{We emphasize that $\beta, \eta$ in the rest of this appendix are different scale from the one used in inductive step. }Let all parameter be as in Theorem~\ref{thm:Closing lemma many scale}. By Theorem~\ref{thm:Closing lemma many scale}, for all $y \in X_{3\eta}$, $r_H \leq \frac{1}{4}\eta$, $r \in [\delta_0, \eta^{d_0}]$, we have
\begin{align*}
    (\lambda \ast \mu_t)((\mathsf{B}_{r_H}^H)^{\pm 1}\exp(B_{r}^{\mathfrak{r}}).y) \ll r^{\epsilon_1}.
\end{align*}

\subsubsection{\text{Boundary effect for $\nu_t$ and $\lambda$}}Due to the boundary effect of balls in $H$, we consider the (coarse) \textit{interior} of $\nu_t$ and $\lambda$. Recall that $\lambda$ is the normalized Haar measure on $\mathsf{B}^{s, H}_{\beta + 100 \beta^{d_0}}$. Let 
\begin{align*}
    \lambda_1 = \lambda|_{\mathsf{B}^{s, H}_{\beta - 100 \beta^{d_0}}}, \quad \mathring{\lambda} = \lambda|_{\mathsf{B}^{s, H}_{\beta}}
\end{align*}
and write
\begin{align*}
    \lambda = \lambda_1 + \lambda_2, \quad\lambda = \mathring{\lambda} + \partial\lambda.
\end{align*}

Recall that $\nu_t = a_t.m_{\mathsf{B}_1^U}$. Let $\nu_{t, 1}'$ be the restriction of $\nu_t$ to $a_t \mathsf{B}_{1 - e^{-t}}^U$. Note that for every $h \in \supp (\nu_{t, 1}')$, we have $\mathsf{B}_1^U h \subseteq \supp (\nu_t)$. 

By Proposition~\ref{pro:Non divergence} applying to $10\eta$ and $x_1 \in X_\eta$ and $B = \mathsf{B}_{1 - e^{-\lambda_{\min}t}}^U$, we can decompose
\begin{align*}
    \nu_t = \nu_{t, 1} + \nu_{t, 2}
\end{align*}
where $\supp(\nu_{t, 1} \ast \delta_{x_1}) \subset X_{10\eta}$, for all $h \in \supp(\nu_{t, 1})$, we have $\mathsf{B}_1^U.h \subseteq \supp(\nu_t)$ and $\nu_{t, 2}(H) \ll \eta^\star$. 

Similarly, write $\nu_t = \mathring{\nu}_{t} + \partial\nu_{t}$ where $\supp(\mathring{\nu}_t \ast \delta_{x_1}) \subset X_{10\eta}$, for all $h \in \supp(\mathring{\nu}_{t})$, we have $\mathsf{B}_{1 - 100\eta}^U.h \subseteq \supp(\nu_t)$ and $\partial\nu_{t}(H) \ll \eta^\star$. 
Note that 
\begin{align*}
    \supp(\nu_{t, 1}) \subset \supp(\mathring{\nu}_t) \quad \supp(\lambda_{1}) \subset \supp(\mathring{\lambda}).
\end{align*}
Let $\sigma$ be the normalized probability Haar measure on $\mathsf{B}_{\beta}^U$. Using the F{\o}lner property of $U$ (see Lemma~\ref{lem:Folner1}), we have that
\begin{align*}
    \lambda \ast \mathring{\nu}_t \ast \delta_{x_1} = \lambda \ast \sigma \ast \mathring{\nu}_t \ast \delta_{x_1} + O(e^{-\star t}).
\end{align*}
It now suffices to decompose $\lambda \ast \sigma \ast \mathring{\nu}_t \ast \delta_{x_1}$ into admissible measures. 

\subsubsection{Decomposition of the space}
For every $j \in \Check{\mathcal{J}}_0$ and every $z \in \supp(\nu_{t, 1} \ast \delta_{x_1}) \cap \Check{\mathsf{Q}}^G_{0}.y_j$, we have that 
\begin{align*}
    z = \mathsf{u}\mathsf{ma} \mathsf{u^-} \exp(w).y_j
\end{align*}
for $\mathsf{u} \in \mathsf{B}_\eta^U$, $\mathsf{mau^-} \in \mathsf{B}_{\beta}^{s, H}$ and $w \in B_{2\beta^{d_0}}^{\mathfrak{r}}$. Note that 
\begin{align*}
    \mathsf{B}_{2\eta}^U.z \subset \supp(\mathring{\nu}_{t} \ast \delta_{x_1}),
\end{align*}
which implies
\begin{align*}
    \mathsf{B}_{\eta}^{U}\mathsf{ma} \mathsf{u^-} \exp(w).y_j \subseteq \supp(\mathring{\nu}_{t} \ast \delta_{x_1}) \cap \Check{\mathsf{Q}}^G_{0}.y_j.
\end{align*}
Therefore, for all $j \in \Check{\mathcal{J}}_0$, we have a decomposition
\begin{align*}
    (\mu_t)|_{\Check{\mathsf{Q}}^G_0.y_j} = \mu_j' + \sum_{i = 1}^{N_j} \sum_{k = 1}^{M_{j, i}} \bar{\mu}_{j, i, k}
\end{align*}
where for all $i, k$ there exist $w_i \in B_{2\beta^{d_0}}^{\mathfrak{r}}$ and $\mathsf{h}_{j, i, k}  \in \mathsf{B}^{s, H}_{\beta}$ so that 
\begin{align*}
    \bar{\mu}_{j, i, k} = (\mathring{\nu}_t \ast \delta_{x_1})|_{\mathsf{B}_{\eta}^U\mathsf{h}_{j, i, k} \exp(w_i).y_j}.
\end{align*}
We also have
\begin{align*}
    \mu_j'(X) \leq (\partial \nu_t \ast \delta_{x_1})(X) \leq \partial \nu_t(H) \ll \eta^{\star}.
\end{align*}

For all $j \in \Check{\mathcal{J}}_0$, consider the set 
\begin{align*}
    \mathfrak{F}_j = \{(w_i, \mathsf{h}_{j, i, k} ): \bar{\mu}_{j, i, k} = (\mathring{\nu}_t \ast \delta_{x_1})|_{\mathsf{B}_{\eta}^U\mathsf{h}_{j, i, k} \exp(w_i).y_j}\}.
\end{align*}
The following lemma follows directly from volume computation. 
\begin{lemma}\label{lem:number of sheet 1}
    We have 
    \begin{align*}
        \# \mathfrak{F}_j \ll \eta^{-d_U}e^{\lambda_{\vol}t}.
    \end{align*}
\end{lemma}

For all $j \in \Check{\mathcal{J}}_0$, $1 \leq i \leq N_j$ and $1 \leq k \leq M_{j,i}$, define $d\mu_{j, i, k}(z) = \Check{\rho}_{0, j}(z)d\bar{\mu}_{j, i, k}(z)$. We have
\begin{align*}
    \mu_t = \mu' + \sum_{j \in \mathcal{J}_0} \sum_{i = 1}^{N_j} \sum_{k = 1}^{M_{i, k}} \mu_{j, i, k}
\end{align*}
where $\mu'(X) \ll \eta^\star$. Let 
\begin{align}
    \hat{c}_j = \sum_{i = 1}^{N_j} \sum_{k = 1}^{M_{i, k}} \mu_{j, i, k}(X).
\end{align}
The following lemma is a direct consequence of pigeonhole principle. 
\begin{lemma}\label{lem:throw away box with small weight}
    If $\hat{c}_j \geq \beta^{2d_0(d_0 - d_U)}$, then $\#\mathfrak{F}_j \gg e^{\lambda_{\vol}t} \beta^{d_0(d_0-d_U)}$. Moreover, 
    \begin{align*}
        1 - \sum_{\hat{c}_j \geq \beta^{2d_0(d_0 - d_U)}} \hat{c}_j = O(\beta^\star).
    \end{align*}
\end{lemma}
\subsubsection{\text{Smearing along the $H$-direction}}
We now smear along the $H$-direction. Let 
\begin{align*}
    \bar{\mu}_{j, i} = \sum_{k = 1}^{M_{j, i}} \bar{\mu}_{j, i, k}, \quad \bar{\mu}_{j} = \sum_{i = 1}^{N_j} \bar{\mu}_{j, i},
\end{align*}
and 
\begin{align*}
    \mu_{j, i} = \sum_{k = 1}^{M_{j, i}} \mu_{j, i, k}, \quad \mu_{j} = \sum_{i = 1}^{N_j} \mu_{j, i}.
\end{align*}
Recall that by definition, $\bar{\mu}_{j, i, k}$ is proportional to the push-forward of the Haar measure on $\mathsf{B}_\eta^{U}$ under $\mathsf{B}_\eta^{U} \to \mathsf{B}_\eta^{U}\mathsf{h}_{j, i, k} \exp(w_i).y_j$. Moreover, the factor is independent to $i$ and $k$. In fact, we have $\bar{\mu}_{j, i, k}(X) \asymp e^{-\lambda_{\vol}t}\eta^{d_U}$. 

\begin{lemma}\label{lem:shear and smear}
    Let $\mu_{j, i, k}^U$ be the Haar measure on $\mathsf{B}_\eta^{U}\mathsf{h}_{j, i, k}$ which is the pull back of $\mu_{j, i, k}$. In particular, it satisfies 
    \begin{align*}
        \mu_{j, i, k}^U(H) = \mu_{j, i, k}(X).
    \end{align*}
    \begin{enumerate}
        \item For all $j, i$, there exists functions $\Check{\varrho}^{U}_{j, i}$ and $\sigma_{j, i}$ so that
        \begin{align*}
            \mathrm{d}\Biggl(\lambda \ast \sigma \ast \Biggl(\sum_{k = 1}^{M_{j, i}} \mu_{j, i, k}^U\Biggr)\Biggr)(h) = \sigma_{j, i}(h)\Check{\varrho}^{U}_{j, i}(h) \,\mathrm{d} m_H(h).
        \end{align*}
        The function $\Check{\varrho}^{U}_{j, i}$ is a Lipschitz function so that $1 \ll \Check{\varrho}^{U}_{j, i} \ll \eta^{-\star}$. The function $\sigma_{j, i}(h)$ satisfies 
        \begin{align*}
            0 \leq \sigma_{j, i}(h) \ll \frac{\mu_{j, i}(X)}{m_{H}(\mathsf{B}^{s, H}_{\beta + 100\beta^{d_0}}\mathsf{B}^U_\eta)}.
        \end{align*}
        Moreover, there exists $C > 0$ so that we have
        \begin{align*}
            \sigma_{j, i}|_{\mathsf{B}^{s, H}_{\beta}\mathsf{B}^U_{\eta - C\beta^{d_0}\eta}} = \frac{\mu_{j, i}(X)}{m_{H}(\mathsf{B}^{s, H}_{\beta + 100\beta^{d_0}}\mathsf{B}^U_\eta)}, \quad \Lip(\Check{\varrho}^{U}_{j, i}|_{\mathsf{B}^{s, H}_{\beta}\mathsf{B}^U_{\eta - C\beta^{d_0}\eta}}) \ll 1.
        \end{align*}
        \item We have
        \begin{align*}
            \Biggl(\lambda \ast \sigma \ast \Biggl(\sum_{k = 1}^{M_{j, i}} \mu_{j, i, k}^U\Biggr)\Biggr)\Bigl(H \setminus \mathsf{B}^{s, H}_{\beta} B_{\eta - C\beta^{d_0}\eta}^U\Bigr) \ll \eta \mu_{j, i}(X).
        \end{align*}
    \end{enumerate}
\end{lemma}

\begin{proof}
    Let $\bar{\mu}_{j, i, k}^U$ be the Haar measure on $\mathsf{B}_\eta^{U}\mathsf{h}_{j, i, k} $ with $\bar{\mu}_{j, i, k}^U(H) = \bar{\mu}_{j, i, k}(X)$. Let $\bar{\varrho}^{U}_{j, i, k}$ be the pull back of the multiplicity function $\Check{\rho}_{0, j}$ on $\mathsf{B}_\eta^{U}\mathsf{h}_{j, i, k}$. We have $\frac{1}{\mathsf{k}} \leq \bar{\varrho}^{U}_{j, i, k} \leq 1$. We have $\mathrm{d}\mu_{j, i, k}^U = \bar{\varrho}^{U}_{j, i, k} \mathrm{d}\bar{\mu}_{j, i, k}^U$. Write the convolution with $\sigma$ as $\mathrm{d}(\sigma \ast \mathrm{d}\mu_{j, i, k}^U) = \varrho^U_{j, i, k} \mathrm{d}\bar{\mu}_{j, i, k}^U$ where $\varrho^U_{j, i, k}$ is a Lipschitz function with $1 \ll \varrho^U_{j, i, k} \leq 1$. Using the F{\o}lner property of $U$, $\Lip(\varrho^U_{j, i, k}) \ll \eta^{-\star}$. 
    
    We now convolve with $\lambda$. We have \begin{align}\label{eqn:shear alon H calculation 1}
    \begin{aligned}
        \mathrm{d}\Biggl(\lambda \ast \sigma \ast \Biggl(\sum_{k = 1}^{M_{j, i}} \mu_{j, i, k}^U\Biggr)\Biggr) =   \Biggl(\sum_{k = 1}^{M_{j, i}} \frac{\mu_{j, i, k}(X)}{V_k} \varrho^U_{j, i, k}\mathds{1}_{\mathsf{B}^{s, H}_{\beta + 100\beta^{d_0}}\mathsf{B}_\eta^U \mathsf{h}_{j, i, k} }\Biggr) \,\mathrm{d}m_H.
    \end{aligned}
    \end{align}
    The numbers $V_k \asymp m_H(\mathsf{B}^{s, H}_{\beta + 100\beta^{d_0}} \mathsf{B}_\eta^U)$. 

    Recall that both $\mathsf{prd}$ and $\Check{\mathsf{prd}}$ are bi-analytic in $\eta_0$-neighborhood of $0$. Since $\mathsf{h}_{j, i, k}  \in \mathsf{B}^{s, H}_{\beta^{d_0}}$, there exists $C > 0$ so that
    \begin{align}\label{eqn:shear along H 2 support}
        \mathsf{B}^{s, H}_{\beta} \mathsf{B}_{\eta - C\beta^{d_0}\eta}^U \subseteq \mathsf{B}^{s, H}_{\beta + 100\beta^{d_0}} \mathsf{B}_\eta^U \mathsf{h}_{j, i, k} .
    \end{align}
    Therefore, on the box $\mathsf{B}^{s, H}_{\beta} \mathsf{B}_{\eta - C\beta^{d_0}\eta}^U$ the function in \cref{eqn:shear alon H calculation 1} can be written as 
    \begin{align*}
        \frac{\mu_{j, i}(X)}{m_H(\mathsf{B}^{s, H}_{\beta + 100\beta^{d_0}} \mathsf{B}_\eta^U)}  \Check{\varrho}^{U}_{j, i}
    \end{align*}
    where $\Check{\varrho}^{U}_{j, i}$ has Lipschitz constant $\ll \eta^{-\star}$ and $1 \ll \Check{\varrho}^{U}_{j, i}\ll 1$. 
    This proves property~(1). Property~(2) follows from a direct calculation. 
\end{proof}

\subsubsection{Decomposition of the local measure according to the weight on $H$-sheets}\label{sec:Decomposition of the local measure according to the weight}
Note that by the dimension estimate in Theorem~\ref{thm:Closing lemma many scale}, for all $j \in \mathcal{J}_0$ and $1 \leq i \leq M_{j, i}$,
\begin{align*}
    \eta^{d_U} e^{-\lambda_{\vol}t} \ll \bar{\mu}_{j, i}(X) \ll \delta_0^{\epsilon_1}.
\end{align*}
Since $\frac{1}{K} \leq \Check{\rho}_{0, j} \leq 1$, there exists an absolute large integer $L$ so that 
\begin{align}\label{eqn:weight on each sheet estimate 1}
    L^{-1}\eta^{d_U} e^{-\lambda_{\vol}t} \leq \mu_{j, i}(X) \leq L \delta_0^{\epsilon_1}.
\end{align}

For all $j \in \mathcal{J}_0$, let
\begin{align*}
    F_j = \{w_i: \bar{\mu}_{j, i, k} = (\mathring{\nu}_t \ast \delta_{x_1})|_{\mathsf{B}_{\eta}^U\mathsf{h}_{j, i, k} \exp(w_i).y_j} \forall k\}.
\end{align*}
By Lemma~\ref{lem:number of sheet 1}, we have
\begin{align*}
    \#F_j \leq \mathfrak{F}_j \ll \eta^{-d_U}e^{\lambda_{\vol}t}.
\end{align*}

Let $\mathsf{L}$ be an integer so that $\mathsf{L} > L$ and also takes care of all constants in Lemma~\ref{lem:shear and smear}. Note that $\mathsf{L}$ is absolute. We now decompose the measure according to its weight on each sheet. 
For all integer $m \geq 0$, let
\begin{align*}
    F_{j, m} = \{w_i \in F_j: \mathsf{L}^{-m}\delta_0^{\epsilon_1} \leq \mu_{j, i}(X) < \mathsf{L}^{-m + 1}\delta_0^{\epsilon_1}\}.
\end{align*}
Since $\mu_{j, i}(X) \geq \mathsf{L}^{-1}\eta^{d_U}e^{-\lambda_{\vol}t}$, the set $F_{j, m} = \emptyset$ for all $m > \lceil\lambda_{\vol}t/\log(\mathsf{L})\rceil$. From now on we only consider $F_{j, m}$ for $1 \leq m \leq \lceil\lambda_{\vol}t/\log(\mathsf{L})\rceil$ and $j \in \mathcal{J}_0$ with 
\begin{align}\label{eqn:condition on weight of box}
    \hat{c}_j = \sum_{i = 1}^{N_j} \sum_{k = 1}^{M_{i, k}} \mu_{j, i, k}(X) \geq \beta^{2d_0(d_0 - d_U)}.
\end{align}
Denote the set consists of such index $j$ by $\mathcal{J}_0'$

For all $1 \leq m \leq \lceil2t/\log(\mathsf{L})\rceil$ so that 
\begin{align}\label{eqn:condition on weight of sheets}
    \sum_{i:w_i \in F_{j, m}} \mu_{j, i}(X) \geq \beta\hat{c}_j \geq \beta^{\star}, 
\end{align}
we have
\begin{align}\label{eqn:number of sheet lower bound}
    \#F_{j, m} \geq \beta^{\star}\delta_0^{-\epsilon_1}.
\end{align}
Denote the set consists of index $m$ satisfying \cref{eqn:condition on weight of sheets} by $\mathcal{M}_j'$

Let 
\begin{align*}
    \hat{c}_{j, m} = \sum_{i:w_i \in F_{j, m}} \mu_{j, i}(X), \quad c_{j, m} = \Biggl(\sum_{j \in \mathcal{J}_0'} \sum_{m \in \mathcal{M}_j'} \hat{c}_{j, m}\Biggr)^{-1} \hat{c}_{j, m}.
\end{align*}

From Lemma~\ref{lem:throw away box with small weight}, we have
    \begin{align*}
        \sum_{j \in \mathcal{J}_0'} \sum_{m \in \mathcal{M}_j'} \hat{c}_{j, m} \geq 1 - O(\beta^\star).
    \end{align*}

For $j \in \mathcal{J}_0'$ and $m \in \mathcal{M}_j'$, we set
\begin{align*}
    \mathcal{E}_{j, m} = \mathsf{E}\exp(F_{j, m}).y_j.
\end{align*}

\begin{lemma}
    For all $j \in \mathcal{J}_0'$ and $m \in \mathcal{M}_j'$, there exists a $(\mathsf{L}^2, \eta, \eta^{-\star})$-admissible measure $\mu_{\mathcal{E}_{j, m}}$ so that for all $\phi \in \mathrm{C}_c(X)$, 
    \begin{align*}
        \Biggl|\int_X \phi \,\mathrm{d}\mu_{\mathcal{E}_{j, m}} - \int_X \phi \,\mathrm{d}\Biggl(\lambda \ast \Biggl(\sum_{i: w_i \in F_{j, m}}\mu_{j, i}\Biggr)\Biggr)\Biggr| \ll \hat{c}_{j, m}\|\phi\|_\infty\eta^\star.
    \end{align*}
\end{lemma}

\begin{proof}
    Let $\mathsf{B}^{s, H}_{\beta}\mathsf{B}^U_{\eta - C\beta^{d_0}\eta}$ be as in the property~(1) of Lemma~\ref{lem:shear and smear}. Let $\mu_{\mathcal{E}_{j, m}}$ be the restriction of 
    \begin{align*}
        \lambda \ast \Biggl(\sum_{i: w_i \in F_{j, m}}\mu_{j, i}\Biggr)
    \end{align*}
    to $\mathsf{B}^{s, H}_{\beta}\mathsf{B}^U_{\eta - C\beta^{d_0}\eta}\exp(F_{j, m}).y_j$ and normalized to probability measure. The inequality follows from Lemma~\ref{lem:shear and smear}. It suffices to show that $\mu_{\mathcal{E}_{j, m}}$ is $\mathsf{L}$-admissible. 
    
    For all $w = w_i \in F_{j, m}$, let 
    \begin{align*}
        \mu_{w} := \frac{m_H(\mathsf{B}^{s, H}_{\beta + 100 \beta^{d_0}}\mathsf{B}^U_\eta)}{\mathsf{L}^{-m}\delta_0^{\epsilon_1}} \lambda \ast \mu_{j, i}|_{\mathsf{B}^{s, H}_{\beta}\mathsf{B}^U_{\eta - C\beta^{d_0}\eta}\exp(w).y_j},
    \end{align*}
    we have
    \begin{align*}
        \mu_{\mathcal{E}_{j, m}} = \frac{1}{\sum_{w \in F_{j, m}}\mu_w(X)}\sum_{w \in F_{j, m}} \mu_w.
    \end{align*}
    By Lemma~\ref{lem:shear and smear}~(1), we have
    \begin{align*}
        \mathrm{d}\mu_{w_i}(z) = \frac{m_H(\mathsf{B}^{s, H}_{\beta + 100 \beta^{d_0}}\mathsf{B}^U_\eta)}{\mathsf{L}^{-m}\delta_0^{\epsilon_1}}\Check{\varrho}_{j, i}(\mathsf{h})\sigma_{j, i}(\mathsf{h}) \, \mathrm{d}m_H(\mathsf{h})
    \end{align*}
    where $z = \mathsf{h}\exp(w_i).y_j$. Moreover, we have
    \begin{align*}
        \mathsf{L}^{-1} \leq \frac{m_H(\mathsf{B}^{s, H}_{\beta + 100 \beta^{d_0}}\mathsf{B}^U_\eta)}{\mathsf{L}^{-m}\delta_0^{\epsilon_1}}\Check{\varrho}_{j, i}(\mathsf{h})\sigma_{j, i}(\mathsf{h}) = \Check{\varrho}_{j, i}(\mathsf{h}) \frac{\mu_{j, i}(X)}{\mathsf{L}^{-m}\delta_0^{\epsilon_1}} \leq \mathsf{L}
    \end{align*}
    for all $\mathsf{h} \in \mathsf{B}^{s, H}_{\beta}\mathsf{B}^U_{\eta - C\beta^{d_0}\eta}$. The Lipschitz property for the function $\Check{\varrho}_{j, i}\sigma_{j, i}$ follows directly from Lemma~\ref{lem:shear and smear}. 
\end{proof}

For $j \in \mathcal{J}_0'$ and $m \in \mathcal{M}_j'$, let
\begin{align*}
    c_{\mathcal{E}_{j, m}} = c_{j, m}.
\end{align*}
From now on, to reduce complicated subscript, we will drop $j,m$ in the subscript. The sum $\sum_{\mathcal{E}}$ will be the same as $\sum_{j \in \mathcal{J}_0'} \sum_{m \in \mathcal{M}_j'}$. 

The above lemma provides a decomposition
\begin{align*}
    \lambda \ast \mu_t = \mu'' + \sum_{\mathcal{E}} c_{\mathcal{E}} \mu_{\mathcal{E}}
\end{align*}
with $\mu''(X) \ll \eta^\star$ and $\mu_{\mathcal{E}}$'s being admissible. 

Therefore, similar to \cite[Theorem 8.4, 8.9]{LMW22} for all $d' \geq 0$ and $u' \in \mathsf{B}_1^U$, we have
\begin{align*}
    \int_{X} \phi(a_{d'} u' x) \,\mathrm{d}(\lambda \ast \mu_t) = \sum_{\mathcal{E}} c_{\mathcal{E}} \int_{X} \phi(a_{d'} u' x) \,\mathrm{d}\mu_{\mathcal{E}} + O(\|\phi\|_\infty \eta^\star).
\end{align*}
This proves property~(1) in Theorem~\ref{thm:closing lemma initial dim}. 

Let $\epsilon_0 = \epsilon_1/2$. We show property~(2) in Theorem~\ref{thm:closing lemma initial dim} holds for this $\epsilon_0$. 
\begin{lemma}
    For all $j$ and $m$ satisfying \cref{eqn:condition on weight of box,eqn:condition on weight of sheets}, Write $\mathcal{E} = \mathcal{E}_{j, m} = \mathsf{E}\exp(F_{j, m}).y_j$ and $F = F_{j, m}$. It satisfies the following conditions. 
    \begin{enumerate}
        \item The number of sheets satisfies
        \begin{align*}
            \beta^{\star}\delta_0^{-2\epsilon_0} \leq \#F \leq \eta^{-d_U}e^{\lambda_{\vol}t}.
        \end{align*}
        \item We have the Margulis function estimate
        \begin{align*}
            f^{(\epsilon_0)}_{\mathcal{E}, \delta_0}(x) \ll \beta^{-\star} \#F \quad \forall x \in \mathcal{E}.
        \end{align*}
    \end{enumerate}
\end{lemma}

\begin{proof}
    Property~(1) follows from Lemma~\ref{lem:number of sheet 1} and \cref{eqn:number of sheet lower bound}. 

    For property~(2), by \cref{pro:frostman energy Margulis function} and $F \subset B_\eta^{\mathfrak{r}}$, it suffices to show that $\mu_{F}$ satisfies
    \begin{align}\label{eqn:final set frostman}
        \mu_F(B_r^{\mathfrak{r}}(w)) \ll \beta^{-\star}r^{\epsilon_1} \quad \forall w \in F \text{ and } \delta_0 \leq r \leq \eta.
    \end{align}    
    For all $w_i \in F = F_{j, m}$, we have
    \begin{align*}
        \mu_{F_{j, m}}(B_r^{\mathfrak{r}}(w_i)) ={}& \frac{\#\{w_{i'} \in F_{j, m}: \|w_{i'} - w_i\| \leq r\}}{\#F_{j, m}}\\
        \leq{}& \frac{\mathsf{L}^{m}\delta_0^{-\epsilon_1}\sum_{i':w_{i'} \in F_{j, m}, \|w_{i'} - w_i\| \leq r} \mu_{j, i'}(X)}{\mathsf{L}^{m - 1}\delta_0^{-\epsilon_1}\sum_{i':w_{i'} \in F_{j, m}}\mu_{j, i'}(X)}\\
        \leq{}& \mathsf{L}\beta^{-\star}\sum_{i':w_{i'} \in F_{j, m}, \|w_{i'} - w_i\| \leq r} \mu_{j, i'}(X).
    \end{align*}
    The last inequality follows from \cref{eqn:condition on weight of sheets}. 

    Recall that $d\mu_{j, i'}(z) = \Check{\rho}_{0, j}(z)d\bar{\mu}_{j, i'}(z)$ where $\Check{\rho}_{0, j} \leq 1$, we have
    \begin{align*}
        \mu_{F_{j, m}}(B_r^{\mathfrak{r}}(w_i)) \leq \mathsf{L} \beta^{-\star} \sum_{i':w_{i'} \in F_{j, m}, \|w_{i'} - w_i\| \leq r} \bar{\mu}_{j, i'}(X).
    \end{align*}
    Since $\bar{\mu}_{j, i'} = \mu_t|_{\Check{\mathsf{Q}}^H_0\exp(w_i).y_j}$, we have
    \begin{align*}
        \mu_{F_{j, m}}(B_r^{\mathfrak{r}}(w_i)) \leq \mathsf{L} \beta^{-\star} \mu_t(\Check{\mathsf{Q}}^H_0 \exp(B_{r}^{\mathfrak{r}}(w_i)).y_j).
    \end{align*}

    Using Lemma~\ref{lem:BCH}, for all $w \in B_{r}^{\mathfrak{r}}(w_i)$, we have
    \begin{align*}
        \exp(w).y_j = \exp(w)\exp(-w_i)\exp(w_i).y_j = \mathsf{h}\exp(\bar{w})\exp(w_i).y_j
    \end{align*}
    where $\|\bar{w}\| \leq 2 \|w - w_i\| \leq 2r$, $\|\mathsf{h} - \Id\| \leq C_0 \eta$. Therefore, 
    \begin{align*}
        \mu_{F_{j, m}}(B_r^{\mathfrak{r}}(w_i)) \leq \mathsf{L} \beta^{-\star} \mu_t(\mathsf{B}_{C_0\eta}^H \exp(B_{2r}^{\mathfrak{r}}(0))\exp(w_i).y_j) \ll \beta^{-\star} r^{\epsilon_1}.
    \end{align*}
    The last inequality follows from Theorem~\ref{thm:Closing lemma many scale} and $100C_0\eta \leq \eta_0$.  
\end{proof}

\section{Estimates on the Brascamp--Lieb constant by Gressman}\label{app:BL}
This appendix is devoted to prove Proposition~\ref{prop:BL constant for H action} and Proposition~\ref{prop:BL constant for unipotent}. The arguments are essentially the arguments in \cite[Section 3]{Gre21}. Let us introduce the following notions. 

From now on in this appendix, we use $\{\tilde{\pi}_j\}_{j = 1}^m$ to denote matrices of sizes $n_j \times n$ with entries to be distinct indeterminants. Since the ring of polynomials over an infinite field is canonically isomorphic to the ring of polynomial functions, we will also use $\{\tilde{\pi}_j\}_{j = 1}^m$ to denote matrices of sizes $n_j \times n$ with arbitrary entries. We say that a polynomial $\Phi$ is polynomial on $\{\tilde{\pi}_j\}_{j = 1}^m$ if it is a polynomials on those indeterminants with coefficients in $\C$. For such polynomials, we define its Hilbert--Schmidt norm $\|\Phi\|_{\mathrm{HS}}$ to be the maximum of $|\Phi|$ on all $m$-tuples of linear maps $\{\pi_j\}_{j = 1}^m$ so that $\|\pi_j\|_{\mathrm{HS}} \leq 1$ for all $j = 1, \ldots, m$. It is comparable (up to absolute constant) to the maximum of absolute value of coefficient of $\Phi$. 

\begin{definition}
Let $\widetilde{\mathrm{IP}}$ be the collection of polynomials $\Phi$ on $\{\tilde{\pi}_j\}_{j = 1}^m$ satisfying the following. 
\begin{enumerate}
\item The polynomials $\Phi$ is homogeneous with degree $d_j > 0$ in each $\tilde{\pi}_j$, i.e., 
\begin{align}\label{eqn:IP homogeneous condition}
\Phi(\{t_j\pi_j\}_{j = 1}^m) = t_1^{d_1} \cdots t_m^{d_m} \Phi(\{\pi_j\}_{j = 1}^m) \text{ for all } t_1, \ldots, t_m \in \R
\end{align}
\item The polynomials $\Phi$ is $\SL_{n_1} \times \cdots \times \SL_{n_m} \times \SL_n$-invariant, i.e., 
\begin{align}\label{eqn:IP invariance condition}
    \Phi(\{A_j\pi_jA^t\}_{j = 1}^m) = \Phi(\{\pi_j\}_{j = 1}^m)
\end{align}
for all  $A \in \SL_n(\R), A_j \in \SL_{n_j}(\R), j = 1, \ldots, m$. 
\item There exists real number $s_{\Phi}$ so that the degrees of $\Phi$ satisfies
\begin{align}\label{eqn:IP degree condition}
    \frac{p_1 n_1}{d_1} = \cdots = \frac{p_m n_m}{d_m} = \frac{1}{s_{\Phi}}.
\end{align}
\end{enumerate}
We set $\mathrm{IP}$ to be the subset of non-zero polynomials in $\widetilde{\mathrm{IP}}$.  We remark that \emph{a priori, $\mathrm{IP}$ might be empty. }
\end{definition}

Elements in $\mathrm{IP}$ (if they exists) provides a bound for the BL constant. 
\begin{lemma}\label{lem:estimate BL constant poly apriori}
    Suppose $\mathrm{IP} \neq \emptyset$. For all $\Phi \in \mathrm{IP}$, we have
    \begin{align}\label{eqn:estimate BL constant poly apriori}
        [\mathrm{BL}(\{\pi_j, p_j\}_{j = 1}^m)]^{-1} \geq \left(\prod_{j = 1}^m n_j^{-\frac{p_j n_j}{2}}\right) \|\Phi\|_{\mathrm{HS}}^{-\frac{1}{s_{\Phi}}}|\Phi(\{\pi_j\}_{j = 1}^m)|^{\frac{1}{s_{\Phi}}}
        \end{align}
        for all $\{\pi_j\}_{j = 1}^m$. 
\end{lemma}

\begin{proof}
    By the invariance of $\Phi$ under the action of $\SL_{n_1} \times \cdots \times \SL_{n_m} \times \SL_{n}$ and homogeneity of $\Phi$, we have
    \begin{align*}
        |\Phi(\{\pi_j\}_{j = 1}^m)| = |\Phi(\{A_j\pi_jA^t\}_{j = 1}^m)| \leq \|\Phi\|_{\mathrm{HS}} \prod_{j = 1}^m \|A_j \pi_j A^t\|_{\mathrm{HS}}^{d_j}.
    \end{align*}
    By the degree condition \cref{eqn:IP degree condition}, we have
    \begin{align*}
        \|\Phi\|_{\mathrm{HS}}^{-\frac{1}{s_{\Phi}}}|\Phi(\{\pi_j\}_{j = 1}^m)|^{\frac{1}{s_{\Phi}}} \leq \prod_{j = 1}^m \|A_j \pi_j A^t\|_{\mathrm{HS}}^{\frac{d_j}{s_{\Phi}}} = \prod_{j = 1}^m \|A_j \pi_j A^t\|_{\mathrm{HS}}^{p_j n_j}
    \end{align*}
    for all  $A \in \SL_n(\R), A_j \in \SL_{n_j}(\R), j = 1, \ldots, m$. The rest follows directly from Lemma~\ref{lem:Lieb constant}. 
\end{proof}

By the Lemma~\ref{lem:estimate BL constant poly apriori}, it suffices to (i) show $\mathrm{IP}$ is non-empty, (ii) provides a subset $\mathrm{IP}_0$ of $\mathrm{IP}$ consisting of polynomials with bounded degree. The following construction translate this into a question in invariant theory when $p_j$'s are non-zero rational numbers. 

\begin{construction}[\text{\cite[Section 3]{Gre21}}]\label{con:invariant theory}
Since $p_j$'s are non-zero rational numbers, there exist positive integers $q, q_1, \ldots, q_n$ so that $p_jn_j = \frac{q_j}{q}$. 
Let 
\begin{align*}
    V = \bigotimes_{j = 1}^m \bigotimes_{i = 1}^q \left(\R^{n_j} \otimes (\R^n)^*\right).
\end{align*}
It a $\R$-linear space with a natural linear action of the real reductive group $S = \SL_{n_1} \times \cdots \times \SL_{n_m} \times \SL_{n}$ as the following. 

For all matrices $A \in \SL_n$, $A_j \in \SL_{n_j}$ for $j = 1, \ldots, m$ and all element $\otimes_{j = 1}^m \otimes_{i = 1}^q (x_j^{(i)} \otimes \psi_j^{(i)}) \in V$, we define
\begin{align*}
    (A_1, \ldots, A_m, A).\left(\otimes_{j = 1}^m \otimes_{i = 1}^q (x_j^{(i)} \otimes \psi_j^{(i)})\right) = \otimes_{j = 1}^m \otimes_{i = 1}^q ((A_j.x_j^{(i)}) \otimes (A.\psi_j^{(i)}))
\end{align*}
where $A.\psi_j^{(i)}$'s are all the natural dual action of the standard action of $\SL_n$ on $\R^n$. 

Note that under the natural isomorphism $\R^{n_j} \otimes (\R^n)^* \cong \Hom(\R^n, \R^{n_j})$, we can identify any element of the form $\tilde{\pi}_1^{\otimes q_1} \otimes \cdots \otimes \tilde{\pi}_m^{\otimes q_m}$ as an element in $V$. The Euclidean inner products on $\R^n$ and $\R^{n_j}$'s naturally induce an inner product and norm on $V$. Under the above identification, we have
\begin{align*}
    \|\tilde{\pi}_1^{\otimes q_1} \otimes \cdots \otimes \tilde{\pi}_m^{\otimes q_m}\| = \prod_{j = 1}^m \prod_{i = 1}^q \|\tilde{\pi}_j\|_{\mathrm{HS}}.
\end{align*}
    
Let $\mathcal{O}(V)^S$ be the ring of invariant polynomials of the above representation. Suppose \emph{$\mathcal{O}(V)^S$ is nontrivial} and $P \in \mathcal{O}(V)^S$ is a homogeneous polynomial with degree $d$, we define $\Phi_P$ to be the polynomial so that its value at any $m$-tuple of matrices $\{\tilde{\pi}_j\}_{j = 1}^m$ is of the following:
    \begin{align}\label{eqn:def of PhiP}
        \Phi_P(\{\tilde{\pi}_j\}_{j = 1}^m) = P(\tilde{\pi}_1^{\otimes q_1} \otimes \cdots \otimes \tilde{\pi}_m^{\otimes q_m}),
    \end{align}
    where again we naturally identify $\tilde{\pi}_1^{\otimes q_1} \otimes \cdots \otimes \tilde{\pi}_m^{\otimes q_m}$ as an element in $V$ using the natural isomorphism $\R^{n_j} \otimes (\R^n)^* \cong \Hom(\R^n, \R^{n_j})$. It is straightforward to check those $\Phi_P$'s satisfy \cref{eqn:IP invariance condition,eqn:IP homogeneous condition,eqn:IP degree condition}. 

    At the end of this construction, we give a remark on the relation between the degrees of $\Phi_{P}$ and $P$. Note that we can choose $q_j$'s and $q$ bounded above by constants depending only on the size of numerator and denominator of $p_j$'s. Also, $\Phi_P$ is homogeneous with degree $d_j = q_j d$ for $j = 1, \ldots, m$. Moreover, 
    \begin{align*}
        \frac{p_j n_j}{d_j} = \frac{p_j n_j}{q_j d} = \frac{1}{qd}.
    \end{align*}
    This shows that $s_{\Phi_P} = qd$.  
\end{construction}

By the above construction, it suffices to (i) show $\mathcal{O}(V)^S$ is non-trivial, (ii) provides a nice generating set of $\mathcal{O}(V)^S$. This is a classical topic in invariant theory. 

\begin{lemma}[\text{\cite[Section 3]{Gre21}}]\label{lem:polynomial bound for BL constant}
    Suppose $p_j$'s are non-zero rational numbers and satisfies
    \begin{align}
        \sum_{j = 1}^m \frac{p_jn_j}{n} = 1.
    \end{align}
    If there exists $\{\pi_j^{(0)}\}_{j = 1}^m$ so that $\mathrm{BL}(\{\pi_j^{(0)}, p_j\}_{j = 1}^m) < +\infty$, then $\mathcal{O}(V)^S$ is non-trivial and as a consequence, $\mathrm{IP} \neq \emptyset$. 
    
    Moreover, there exists a non-empty finite set $\mathrm{IP}_0 \subseteq \mathrm{IP}$ of polynomials \emph{independent} to the previous $\{\pi_j^{(0)}\}_{j = 1}^m$  satisfying the following. 
    \begin{enumerate}
        \item The degree of polynomials in $\mathrm{IP}_0$ satisfy
    \begin{align*}
        \sup_{\Phi \in \mathrm{IP}_0}\deg \Phi \ll_{n, \{n_j\}_{j = 1}^m, \{p_j\}_{j = 1}^m} 1.
    \end{align*}
    Also, the numbers $s_{\Phi}$'s are positive integers and satifies $s_{\Phi} \ll_{n, \{n_j\}_{j = 1}^m, \{p_j\}_{j = 1}^m} 1$. The dependence on $\{p_j\}_{j = 1}^m$ is explicit dependence on the size of their numerators and denominators. 
    \item If $\mathrm{BL}(\{\pi_j, p_j\}_{j = 1}^m) < +\infty$, then
    \begin{align}\label{eqn:estimate BL constant poly positive}
       \sup_{\Phi \in \mathrm{IP}_0} \|\Phi\|_{\mathrm{HS}}^{-\frac{1}{s_{\Phi}}}|\Phi(\{\pi_j\}_{j = 1}^m)|^{\frac{1}{s_{\Phi}}} > 0.
        \end{align}
    \end{enumerate}
\end{lemma}

\begin{proof}
We first show that $\mathcal{O}(V)^S$ is non-trivial if there exists $\{\pi_j^{(0)}, p_j\}_{j = 1}^m$ so that $\mathrm{BL}(\{\pi_j^{(0)}, p_j\}_{j = 1}^m) < +\infty$. Let
\begin{align*}
    v^{(0)} := (\pi_1^{(0)})^{\otimes q_1} \otimes \cdots (\pi_m^{(0)})^{\otimes q_m}.
\end{align*}
By Lemma~\ref{lem:Lieb constant}, we have
\begin{align*}
    \inf_{g \in S} \|g. v^{(0)}\|
    = \inf_{\substack{A \in \SL_n(\R), \mathstrut \\A_j \in \SL_{n_j}(\R), j = 1, \ldots, m}} \prod_{j = 1}^m \| A_j \pi_j^{(0)} A^t\|_{\mathrm{HS}}^{q_j} > 0.
\end{align*}
This implies that the orbit $S.v^{(0)}$ does not contain $0$. 

If the orbit $S.v^{(0)}$ is closed in Hausdorff topology, then by \cite[Corollary 5.3]{Bir71}, it is also closed in Zariski topology and therefore $\mathcal{O}(V)^S$ is non-trivial. 

If not, then by \cite[Lemma 3.3]{RS90}, there exists a vector $v^{(1)} \in V$ so that $S.v^{(1)}$ is closed and a semisimple element $X \in \mathfrak{s}$ so that
\begin{align*}
    v^{(1)} = \lim_{t \to -\infty}\exp(tX).v^{(0)}.
\end{align*}
Note that for all $g \in S$, we have
\begin{align*}
    \|g.v^{(1)}\| = \lim_{t \to -\infty}\|g\exp(tX).v^{(0)}\| \geq \inf_{h \in S}\|h.v^{(0)}\| > 0.
\end{align*}
By \cite[Corollary 5.3]{Bir71}, $S.v^{(1)}$ is closed in Zariski topology and not containing $0$. Therefore, there exists $P \in \mathcal{O}(V)^S$ so that
\begin{align}\label{eqn:separation for apriori data}
    \Phi_P(\{\pi_j^{(0)}\}_{j = 1}^m) = P((\pi_1^{(0)})^{\otimes q_1} \otimes \cdots \otimes (\pi_m^{(0)})^{\otimes q_m}) = P(v^{(1)}) \neq 0.
\end{align}
The proof of the first claim in the lemma is complete. 

By Hilbert's theorem,  $\mathcal{O}(V)^S$ is a finitely generated $\C$-algebra. Let $\{P_i\}_{i \in \mathcal{I}}$ be a set of homogeneous polynomial so that $\{1, P_i\}_{i \in \mathcal{I}}$ generates $\mathcal{O}(V)^S$. By \cite[Section 1.4]{Der04}, there is an algorithm to construct such generators using Buchberger's algorithm and Reynold operator. In this case, the Reynold operator can be explicitly constructed via Cayley's $\Omega$-process (see e.g. \cite[Section 4.3]{Stu93} or \cite[Section 3.2]{Gre21}). Therefore, we can assume 
\begin{align}\label{eqn:degree bound from Cayley}
    \deg P_i \ll_{n, \{n_j\}_{j = 1}^m} 1
\end{align}
for all $i \in \mathcal{I}$ where the implied constants are computable and depend only on $n$ and $\{n_j\}_{j = 1}^m$. Let
\begin{align*}
    \mathrm{IP}_0 = \{\Phi_{P_i}: i \in \mathcal{I}\}.
\end{align*}
Note that $\mathrm{IP}_0$ does not depend on the data $\{\pi_j^{(0)}\}_{j = 1}^m$. The degree bound in property~(1) follows from \cref{eqn:degree bound from Cayley} and the last paragraph of Construction~\ref{con:invariant theory}. 

For property~(2), the argument is exactly the same as the proof of the first claim. Note that the finiteness of Brascamp--Lieb constant implies the corresponding orbit is bounded away from $0$ by Lemma~\ref{lem:Lieb constant}. Since $\{P_i\}_{i \in \mathcal{I}}$ consists of homogeneous generator of $\mathcal{O}(V)^S$, there is one of them separate the orbits from $0$. Similar to \cref{eqn:separation for apriori data}, this proves property~(2) and hence the lemma. 
\end{proof}

\begin{proof}[Proof of Proposition~\ref{prop:BL constant for H action}]
    Let $m$ be the positive integer and  $\mathcal{O}_W \subseteq H^m$ be the Zariski open dense subset given by Proposition~\ref{prop:BL finite for H action}. Let $\mathcal{O}_{W}' = (\mathcal{O}_{W})^t$. For all $(h_1, \ldots, h_m) \in \mathcal{O}_{W}'$, Proposition~\ref{prop:BL finite for H action} implies that $\mathrm{BL}\left(\left\{\pi_{W} \circ h_j, \frac{n}{km}\right\}_{j = 1}^m\right)$ is finite. By Lemma~\ref{lem:polynomial bound for BL constant}, there exists a nonzero homogeneous polynomial $\tilde{\Phi} \in \mathrm{IP}_0$, we have
    \begin{align}\label{eqn:polybound for bl constant in proof 1}
    \begin{aligned}
        {}&[\mathrm{BL}(\{\pi_W \circ h_j, \frac{n}{km}\}_{j = 1}^m)]^{-1}\\
        \geq{}& \left(\prod_{j = 1}^m k^{-\frac{n}{2m}}\right) \|\tilde{\Phi}\|_{\mathrm{HS}}^{-\frac{1}{s_{\tilde{\Phi}}}}|\tilde{\Phi}(\{\pi_W \circ h_j\}_{j = 1}^m)|^{\frac{1}{s_{\tilde{\Phi}}}} > 0.
    \end{aligned}
    \end{align}
    Let
    \begin{align*}
        \Phi(h_1, \ldots, h_m) := \tilde{\Phi}(\{\pi_W \circ h_j\}_{j = 1}^m)
    \end{align*}
    and $s_{\Phi} = s_{\tilde{\Phi}}$. We have $\deg \Phi \ll_V 1$ where the implied constant depends on $n$ and the degree of this polynomial representation. We also have $s_{\Phi} = s_{\tilde{\Phi}}$ is a positive integer and $s_{\Phi} = s_{\tilde{\Phi}} \ll_n 1$. In fact, in this case, we have
    \begin{align*}
        p_j n_j = \frac{n}{km} \cdot k = \frac{n}{m} = \frac{q_j}{q}.
    \end{align*}
    Therefore we can take $q = m \ll_n 1$ and $q_j = n$. 
    Since $\tilde{\Phi} = \tilde{\Phi}_P$ for some homogeneous polynomial $P \in \mathcal{O}(V)^S$ with degree $d \ll_{n, m} 1$, we have
    \begin{align*}
        s_{\Phi} = s_{\tilde{\Phi}} = q d = md \ll_{n} 1.
    \end{align*}
    The last inequality follows from $m \ll_n 1$. 
    By \cref{eqn:polybound for bl constant in proof 1}, we have
    \begin{align*}
        \mathrm{BL}\left(\left\{\pi_{W} \circ h_j, \frac{n}{km}\right\}_{j = 1}^m\right) \ll_{n} \|\Phi\|_{\infty}^{\frac{1}{s_{\Phi}}} |\Phi(h_1, \ldots, h_m)|^{-\frac{1}{s_{\Phi}}},
    \end{align*}
    which completes the proof of the proposition. 
\end{proof}

\begin{proof}[Proof of Proposition~\ref{prop:BL constant for unipotent}]
The proof is identical to the proof of Proposition~\ref{prop:BL constant for H action} if we replace Proposition~\ref{prop:BL finite for H action} by Proposition~\ref{prop:BL finite for U}. 
\end{proof}

\nocite{*}
\bibliographystyle{alpha_name-year-title}
\bibliography{References}
\end{document}